\documentclass[a4paper,11pt, reqno]{amsart}
\usepackage[usenames,dvipsnames]{xcolor}
\usepackage[normalem]{ulem}
\usepackage{pdfsync}

\usepackage{enumerate}
\usepackage{amssymb, amsmath}
\usepackage{mathrsfs}
\usepackage{amscd}
\usepackage[active]{srcltx}
\usepackage{verbatim}
\usepackage[colorlinks,linkcolor={black},citecolor={blue},urlcolor={black}]{hyperref}
\theoremstyle{plain}
\newtheorem{theorem}{Theorem}[section]
\newtheorem{corollary}[theorem]{Corollary}
\newtheorem{lemma}[theorem]{Lemma}
\newtheorem{proposition}[theorem]{Proposition} 

\newtheorem{definition}[theorem]{Definition}

\newtheorem*{definition*}{Definition}

\theoremstyle{remark}
\newtheorem{remark}[theorem]{Remark}

\newtheorem*{claim*}{Claim}
\newtheorem*{remark*}{Remark}
\newtheorem*{example*}{Example}
\newtheorem*{notation*}{Notation}

\numberwithin{equation}{section}

\def\N{{\mathbb N}}

\def\R{{\mathbb R}}

\newcommand{\V}{{\mathbb V}}

\newcommand{\eps}{\varepsilon}

\newcommand{\m}{\boldsymbol{\mathfrak m}}

\newcommand{\dd}{\;\mathrm{d}}

\DeclareMathOperator{\supp}{supp}

\newcommand{\abs}[1]{\vert {#1}\vert}

\DeclareMathOperator{\ent}{Ent}

\DeclareMathOperator{\Ch}{Ch}

\newcommand{\ddt}{\frac{\mathrm{d}}{\mathrm{d}t}}
\newcommand{\ddtr}{\frac{\mathrm{d}}{\mathrm{d}t}^{\kern-3pt+}}

\newcommand{\dds}{\frac{\mathrm{d}}{\mathrm{d}s}}

\newcommand{\cH}{\mathscr{H}}
\newcommand{\cB}{\mathcal{B}}

\newcommand{\cL}{\mathcal{L}}

\newcommand{\cX}{\mathcal{X}}
\newcommand{\cE}{\mathcal{E}}
\newcommand{\cA}{\mathcal{A}}

\newcommand{\ce}{\mathcal{CE}}

\newcommand{\cP}{\mathscr{P}}

\renewcommand{\tilde}{\widetilde}
\newcommand{\e}{\mathrm{e}}

\newcommand{\EVI}{{\sf EVI}}
\newcommand{\EDI}{{\sf EDI}}
\newcommand{\sfd}{{\sf d}}
\newcommand{\BorelSets}[1]{{\mathscr B}(#1)}

\newcommand{\rmD} {{\rm D}}
\newcommand{\rmI}{{\rm I}}
\newcommand{\eeta}{{\mbox{\boldmath$\eta$}}}
\newcommand{\ssigma}{{\mbox{\boldmath$\sigma$}}}

\newcommand{\ppi}{{\mbox{\boldmath$\pi$}}}
\newcommand{\xX}{\boldsymbol X}
\newcommand{\calX}{{\mathcal X}}
\newcommand{\xx}{\boldsymbol x}
\newcommand{\yy}{\boldsymbol y}
\newcommand{\zz}{\boldsymbol z}
\newcommand{\sxx}{\mbox{$\scriptsize\boldsymbol x$}}
\newcommand{\szz}{\mbox{$\scriptsize\boldsymbol z$}}
\newcommand{\sfp}{\mathsf p}
\newcommand{\sfc}{{\mathsf c}}
\newcommand{\Action}{{\mathscr A}}

\let\cal\mathcal

\newcommand{\BE}[2]{{\mathsf{BE}(#1,#2)}}
\newcommand{\sfP}{{\mathsf P}}
\newcommand{\sfS}{{\mathsf S}}
\newcommand{\sfh}{{\mathsf h}}
\newcommand{\rme}{\mathrm e}
\newcommand{\cAs}{\cA_\cE^*}

\newcommand{\OOO}{}

\newcommand{\nc}{\normalcolor}

\begin{document}

\title[Extended metric spaces]{Optimal transport, Cheeger energies and
  contractivity \\ of dynamic 
  transport distances in extended spaces}

\author{Luigi Ambrosio}
   \address{Scuola Normale Superiore, Pisa} 
   \email{luigi.ambrosio@sns.it}
   \author{Matthias Erbar}
   \address{University of Bonn} \email{erbar@iam.uni-bonn.de}
 \author{Giuseppe Savar\'e}
   \address{Pavia University}\email{giuseppe.savare@unipv.it}

\thanks{}



\dedicatory{Dedicated to J.L.~Vazquez in occasion of his 70th birthday}

\date\today

\maketitle

\begin{abstract}
  We introduce the setting of \emph{extended metric-topological
    measure spaces} as a general ``Wiener like'' framework for optimal
  transport problems and nonsmooth metric analysis in 
  infinite dimension. 
  
  After a brief review of optimal transport tools for general Radon
  measures, 
  we discuss the notions of the Cheeger energy, 
  of the Radon measures concentrated on absolutely continuous curves, 
  and of the induced ``dynamic transport
  distances''.  We study their main properties and their
  links with the theory of Dirichlet forms and the
  Bakry-\'Emery curvature condition,
  in particular concerning the contractivity properties 
  and the EVI formulation of the induced Heat semigroup.
\end{abstract}

\tableofcontents

\section{Introduction}

In the last years many papers have been devoted to the investigation of the connection between gradient contractivity,
contractivity of transport distances and lower bounds on Ricci curvature and to the connection between metric and differentiable 
structures. In these investigations one can take as starting point either a metric measure space $(X,\sfd,\m)$ or a Dirichlet form $\cE$ in $L^2(X,\m)$. 
In particular \cite{AGS12} provided key connections between the two viewpoints, proving that under mild regularity
assumptions the distance $\sfd_\cE$ generated out of the Dirichlet form as in \cite{Biroli-Mosco95} induces a metric energy
(called Cheeger energy in \cite{AGS11a}, \cite{AGS11b}) equal to $\cE$, and that Ricci lower bounds can be equivalently
stated either in terms of the Bakry-\'Emery gradient $K$-contractivity
condition $\BE K\infty$, $K\in\R$,
\begin{equation}
  \label{eq:7}
  \Gamma (\sfP_tf)\leq \e^{-2Kt}\,\sfP_t\Gamma(f)
\end{equation}
(here $\sfP$ is the semigroup induced by $\cE$), or
in terms of $K$-convexity of the entropy along Wasserstein geodesics (see also \cite{Koskela_Zhou,Koskela_Zhou_Shanmugalingam} and
also \cite{EKS} for extensions to the
case when upper bounds on the dimension are considered).
The crucial link between the two formulations is provided by 
the characterization of the semigroup $\sfP$ as the 
$\mathrm{EVI}_K$-gradient flow (see \eqref{eq:evik_intro}
below) of the entropy in the Wasserstein
space. 

A typical assumption made in the above-mentioned papers is that the topological/measure structure is induced by the distance, and that
the distance is finite: for instance, when one takes $\cE$ as starting point, one assumes that the topology induced by $\sfd_\cE$ coincides
with the initial topology of the space. However, there exist examples where the topology induced by the natural distance is too fine and the 
distance can be even infinite: the simplest and probably most studied and natural example is the so-called Wiener space, i.e.~a Gaussian
measure space endowed with the Cameron-Martin distance. 

The main goal of this and of the forthcoming paper \cite{AES} is a deeper investigation of the above-mentioned problems in
\textit{extended} metric structures, where extended metric spaces are sets $X$ endowed with a symmetric
and triangular $\sfd:X\times X\to [0,\infty]$, with $\sfd(x,y)=0$ iff $x=y$. Extended distances arise in a natural way
either by taking the supremum $\sup |f(x)-f(y)|$ along a set $\mathcal F$ of functions which separate the points
of $X$ (this is precisely what happens with $\sfd_\cE$), by construction of length distances and more generally by
action minimization. At this level many extension of the classical metric theory, for instance the existence of metric
derivatives $|\dot x|(t)$ for absolutely continuous curves $x(t)$ are fairly trivial, 
since $\sfd$ induces equivalence classes in $X$ which are classical metric spaces; on the other hand,
already the example of the Wiener space shows that when when we are given a reference measure $\m$ on $X$
it is very hard to work with the quotient structure, and it is  much better to consider the space as a whole; also in many
cases it happens that we are given a topology $\tau$ in $X$, coarser than the topology induced by the extended
distance. 
We axiomatize this richer structure with the concept of \textit{extended metric-topological space} $(X,\tau,\sfd)$, 
characterized by the existence of a family $\mathcal A$ of bounded functions which separate the points
of $X$ and generate both the Hausdorff topology $\tau$ and the distance $\sfd(x,y)$, the latter with the formula
$\sup_{f\in{\mathcal A}}|f(x)-f(y)|$. We denote in the sequel the algebra ${\rm Lip}_b(X,\tau,\sfd)$ of bounded,
$\sfd$-Lipschitz and $\tau$-continuous functions, which includes $\mathcal A$ and generates $\tau$ as well.
In view of the applications we have in mind in \cite{AES}, we are not assuming that the family $\mathcal A$ is countable
and, correspondingly, we do not need extra assumptions on $\tau$; the complete regularity of $\tau$ is implied
by the $\mathcal A$-generating property and it will be sufficient for our purpose, provided the reference measure $\m$ is Radon.  

Now we pass to a more detailed description of the content of the paper. Section~\ref{sec:1} is devoted to some measure-theoretic
preliminaries, mostly borrowed from the very comprehensive monographs \cite{Bogachev,Schwartz73}. In particular we introduce the
class of Radon measures, denoted by $\cP(X)$, and recall the basic compactness theorem for families of probability measures
(see Theorem~\ref{thm:compa-Riesz}). Then we recall the dual
formulation of the optimal transport problem when the marginals are
Radon measures, following \cite{Kellerer84} (see also \cite{Zaev} for the analysis of optimal transport
problem in a very general setup).

Section~\ref{sec:2} contains basic and already well-established results of the metric theory, with their easy adaptation
to the extended setting. The only, but essential, new ingredient is a self-improvement principle for solutions to the so-called $\EVI_K$
(evolution variational inequality) gradient flows along a semigroup ${\sf S}$
\begin{equation}\label{eq:evik_intro}
  \ddtr\frac12 \sfd^2({\sf S}_tx,y) +
    \frac{K}{2}\sfd^2({\sf S}_tx,y)\leq F(y) - F({\sf S}_tx)\qquad\forall t>0\;,
\end{equation}
(for all $y$ at a finite $\sfd$-distance from some ${\sf S}_tx$)
which allows to some extent to pass from \eqref{eq:evik_intro} to the same inequality
for the (extended) length distances $\sfd_\ell$, $\bar\sfd_\ell$ induced by $\sfd$ (the former defined in \eqref{eq:71} by $\eps$-chains, the
latter defined in \eqref{eq:81} by the minimization of the length of curves), see 
Theorem~\ref{thm:self-improvement} for a precise statement which involves a powerful integral formulation
of \eqref{eq:evik_intro}. Moreover in Corollary~\ref{cor:appgeo} we derive from
$\EVI_K$ a discrete convexity property relative to $\sfd_\ell$ that, under suitable compactness assumptions, can be improved to convexity along all
geodesics.

In Section~\ref{sec:3} we introduce metric-topological spaces $(X,\tau,\sfd)$ and we discuss a few preliminary properties
of them, in particular the density of ${\rm Lip}_b(X,\tau,\sfd)$ in $L^2(X,\m)$ for $\m$ Radon, compactness of measures
in the space $X^D$ of paths and lower semicontinuity of the $p$-action, $p\in (1,\infty)$, in this context. In Section~\ref{sec:4} we extend to
this setting (a priori neither separable nor metrizable) basic results relative to the extended distance in $\cP(X)$ induced
by the quadratic cost $\sfd^2$. Denoting this distance by $W_\sfd$, we prove compactness and lower semicontinuity theorems, the implication
from $W_\sfd$-convergence to weak convergence 
and the basic superposition theorem which extends to a
non-Polish setup the recent paper \cite{Li14}. Thanks to this result, $2$-absolutely continuous curves $\mu_t$ 
in $(\cP(X),W_2)$ 
have a lifting to the space $X^{[0,1]}$, i.e. there exists $\eeta\in\cP(X^{[0,1]})$
concentrated on $2$-absolutely continuous paths $\eta:[0,1]\to (X,\sfd)$ whose marginals are $\mu_t$ and satisfying
$$
\int |\dot\eta(t)|^2\dd\eeta(\eta)=|\dot\mu_t|^2
\qquad\text{for a.e. $t\in (0,1)$}\;.
$$

In Section~\ref{sec:Chee} we recall the basic construction of the so-called Cheeger energy \cite{Cheeger00}, adapted to the extended
setting $(X,\tau,\sfd)$ with a reference measure $\m\in\cP(X)$. Following with minor variants \cite{AGS11a} (these variants allow to bypass
some measurability issues relative to the asymptotic Lipschitz constant) we set
$$
\Ch(f):=\inf\liminf_{n\to\infty}\int g_n^2\dd\m\;,
$$
where the infimum runs among all sequences $(f_n)\subset {\rm Lip}_b(X,\tau,\sfd)$ with $\lim_n\int|f_n-f|^2\dd\m=0$ and all $\m$-measurable functions
$g_n\geq {\rm Lip}_a(f_n,\cdot)$ $\m$-a.e. in $X$ (where ${\rm Lip}_a$ is the so-called asymptotic Lipschitz constant, see \eqref{eq:deflipa}).
Together with the construction of $\Ch$ there is the construction of a local object, called minimal relaxed slope and denoted with
$|\rmD f|_w$, which provides integral representation to $\Ch$ via $\Ch(f)=\int |\rmD f|^2_w\dd\m$ when $\Ch(f)$ is finite. As shown
in \cite{Cheeger00} and \cite{AGS11a}, many classical properties of Sobolev functions extend to this setting; in addition, defining
$\Delta f$ as the element with minimal norm in the subdifferential $\partial\Ch (f)$, it is well-defined a Heat flow in $L^2(X,\m)$
(linear iff $\Ch$ is quadratic) $\sfP$, given by $\frac d{dt}\sfP_tf=\Delta \sfP_t f$, according to the evolution theory for maximal monotone operators in
Hilbert spaces.

The aim of Section~\ref{sec:extendedcpa} is to introduce two more extended distances induced by $\Ch$, in the subset $\cP^a(X)\subset\cP(X)$
of measure absolutely continuous w.r.t.~$\m$. The first one, denoted by $W_{\Ch}$, is a length distance in
$\cP^a(X)$ whose definition is inspired by the Benamou-Brenier formula
$$
W_{\Ch}^2(\rho_0\m,\rho_1\m):=\inf\left\{\int_0^1\|\rho_t'\|^2\dd t:\ \rho_t\in {\sf CE}^2(X,\Ch,\m)\right\}\;.
$$
Here $\|\rho_t'\|$ is the least function $c(t)$ in $L^2(0,1)$ satisfying
$$
\biggl|\int f\rho_s\dd\m-\int f\rho_t\dd\m\biggr|\leq
\int_s^t c(r)\biggl(\int |\rmD f|_w^2\rho_r\dd\m\biggr)^{1/2}\dd r\qquad\forall f\in{\rm Lip}_b(X,\tau,\sfd)
$$
and the property above defines the class of curves in ${\sf CE}^2(X,\Ch,\m)$.  
The second one, denoted $W_{\Ch,*}$, has a dual character and it is defined by
$$
W_{\Ch,*}^2(\rho_0,\rho_1):=2\sup_\phi\int (\phi_1\rho_1-\phi_0\rho_0)\dd\m\;,
$$
where the supremum runs along all the ``formal" subsolutions to the Hamilton-Jacobi
equation $\ddt \phi_t+|\rmD\phi_t|_w^2/2=0$.
At this high level of generality, we are only able to prove that $W_{\sfd}\leq W_{\Ch,*}\leq W_{\Ch}$; however, when we pass from
the global to the infinitesimal behaviour, these distances reveal much
closer connections. Indeed, under the Bakry-\'Emery gradient contractivity condition
$|\rmD \sfP_t f|^2_w\leq \rme^{-2Kt}\sfP_t|\rmD f|_w^2$ for some $K\in\R$ one can prove that the length distance associated to $W_{\Ch,*}$ is $W_{\Ch}$ (see
Remark~\ref{rem:diffusion}); in addition, along curves $\mu_t=\sfP_t\rho\,\m$ with $\rho\in L^\infty_+(X,\m)$ probability density,
these distances are finite and the metric derivatives w.r.t.~all these distances coincide a.e. in $(0,\infty)$ (see Corollary~\ref{cor:eqmet}). 

Building on this and refining the analysis made in \cite{AGS11a}, we prove in the subsequent Section~\ref{sec:identiflows} that the metric gradient flows
in $\cP^a(X)$ of the relative entropy functional
\begin{equation}
\ent(\rho\,\m):=\int\rho\log\rho\dd\m\label{eq:27}
\end{equation}
w.r.t.~the distances $W_{\sfd}$, $W_{\Ch,*}$ and $W_{\Ch}$ coincide with the Heat semigroup $\sfP$ in the $\sfP$-invariant
class of bounded probability densities under a mild lower semicontinuity assumption on $|\rmD^-\ent|$, the descending slope
of $\ent$ w.r.t. $W_\sfd$. 

Section~\ref{sec:stabchee} contains a key stability result for Cheeger's energies, 
which deals with the case of a monotone family of $(\tau\times\tau)$-continuous 
distances $\sfd_i$ approximating from below $\sfd$, as in the definition of extended metric-topological space.
We prove convergence of the correspondent gradient flows and the formula
$$
\Ch=\bigl(\inf_{i\in I} \Ch_i\bigr)_*\;,
$$
where $\Ch_i$ is Cheeger's functional relative to $\sfd_i$ and $()_*$ denotes the lower semicontinuous envelope
in $L^2(X,\m)$. As a byproduct, we get also convergence of the corresponding $L^2$ Heat flows. In view of the applications given
in Section~\ref{sec:transfer} we include in the convergence result also the case when $\sfd_i$ are semi distances, i.e. 
$\sfd_i(x,y)=0$ does not imply $x=y$. This inclusion requires an adaptation of the construction of $\Ch$ to the semimetric
setting.

In Section~\ref{sec:Energy} we take the point of view of a strongly local and Markovian Dirichlet form $\cE$ endowed with
a carr\'e du champs operator $\Gamma$. In this context the definitions of $W_{\Ch}$ and $W_{\Ch,*}$ can be immediately adapted,
formally replacing the minimal relaxed slope $|\rmD f|_w$ with $\sqrt{\Gamma(f)}$. Denoting by $W_{\cE}$ and $W_{\cE,*}$ the
corresponding extended distances, also in this context the
$K$-gradient contractivity condition \eqref{eq:7} 
yields that 
$W_{\cE}$ is the length distance associated to $W_{\cE,*}$; furthermore,
we prove in Theorem~\ref{thm:contraction-BE} that
\eqref{eq:7}
implies $K$-contractivity of both squared distances w.r.t.~to $\sfP^\cE$; if $L^2(X,\m)$ is separable 
we prove a partial converse, namely $K$-contractivity
of $W_{\cE}^2$ implies the $K$-gradient contractivity \eqref{eq:7}. 

Section~\ref{sec:actionnew} provides the $\EVI_K$ property \eqref{eq:evik_intro} of ${\sf S}=\sfP^\cE$ 
relative to the extended distances $W_{\cE,*}$ and $W_{\cE}$. The proof first provides a duality 
estimate involving $W_{\cE,*}$ and then, using the self-improvement principle of Section~\ref{sec:2} and the relation
between $W_{\cE,*}$ and $W_{\cE}$, the final result. In conjunction with the compactness properties of the sublevels of $\ent$, the
$\EVI_K$ property yields geodesic convexity, relative to $W_\cE$, of the sublevels of the entropy. In addition, when $K\geq 0$
also the sets $\{\rho\m:\ \|\rho\|_\infty\leq c\}$ are convex and when $K>0$ the Logarithmic Sobolev Inequality and the
Talagrand Transport Inequality hold.

In Section~\ref{sec:transfer} we analyze more in detail the connection between the metric and differentiable structures. More precisely, starting
from an extended metric-topological space $(X,\sfd,\tau)$ endowed with a reference measure $\m\in\cP(X)$ we can build $\Ch$ and, assuming
$\Ch$ to be quadratic, ask whether $\Ch$ fits into the theory of Dirichlet forms; the answer is affirmative and we prove, following essentially
\cite{AGS11b}, that $\Ch$ is a strongly local and Markovian Dirichlet form, and that $|\rmD f|_w^2$ corresponds to $\Gamma(f)$. Conversely,
given $\cE$ we can easily build an extended metric measure structure by selecting a family $\mathcal A\subset\{f:\ \Gamma(f)\leq 1\}$ of pointwise defined
functions which separate the points of $X$. It is important to understand to what extent these two constructions at the level of the energies and of the distances
are each the inverse of the other, namely
$$
\cE\rightarrow\sfd_{\cE}\rightarrow\Ch_{\sfd_\cE}=\cE?\qquad \quad\sfd\rightarrow\Ch\rightarrow\sfd_{\Ch}=\sfd?
$$
We know from \cite{Sturm_single} that, in general, even for Dirichlet
forms and for length distances \cite{Stollmann10}, we can't expect that the answer is always affirmative.
In order to understand this question we identify special properties of Cheeger's energies $\Ch$ and of distances $\sfd_{\Ch}$ associated to them,
when one chooses as $\mathcal A$, in the construction of $\sfd_{\Ch}$, the class $\{f\in C_b(X):\ |\rmD f|_w\leq 1\}$. For Cheeger's energies, the
special property is the so-called $\tau$-upper regularity, already
identified in \cite{AGS12} and here proved and adapted to the extended
setting: according to this property $|\rmD f|_w$ can be approximated 
by $\tau$-upper semicontinuous functions bounding the gradients of approximating functions $f_n$ relative
to finite distances, see Definition~\ref{def:upper_regularity_cE} for the precise statement. At the level of distances, the special property we need is that functions in the class
$\{f\in C_b(X):\ |\rmD f|_w\leq 1\}$ have to be 1-Lipschitz w.r.t.~$\sfd$.

We prove that $\cE\leq \Ch_{\sfd_\cE}$ and that equality holds iff the choice of $\mathcal A$ ensures the
$\tau$-upper regularity of $\cE$ and that when this happens several other identifications occur, see Theorem~\ref{thm:ultrap} 
for the precise statement; the proof, in part adapted from \cite{AGS12}, relies on the identification results established in the previous sections and particularly
on Section~\ref{sec:stabchee}. Notice that in the ``regular'' setting of \cite{AGS12} where
$\tau$ coincides with the topology induced by $\sfd_\cE$,
$\tau$-regularity can be also obtained as a consequence of the
Bakry-\'Emery $\BE
K\infty$ condition and a weak Feller property of $\sfP$. 

On the other hand, in Theorem~\ref{thm:ultrap1} we prove that $\sfd_{\Ch}$ is always larger than
$\sfd$,  and that equality holds if and only any function in $\{f\in C_b(X):\ |\rmD f|_w\leq 1\}$ is 1-Lipschitz w.r.t.~$\sfd$. These results
are independent of the doubling and Poincar\'e assumption considered in \cite{Koskela_Zhou}, \cite{Koskela_Zhou_Shanmugalingam}
and generalize those of \cite{AGS12} to the extended setup.

Section~\ref{sec:examples} describes classical examples of extended spaces (Wiener spaces, configuration spaces) and shows
how they fit in our framework.

Let us conclude by pointing out potential developments that we plan to investigate, at least in part, in \cite{AES}. The first one deals
with the so-called measurable distances, namely distances which are pointwise defined only at the level of subsets of positive
$\m$-measure of the space. This class of distances appears for instance in \cite{Hino-Ramirez}, in connection with the short time behaviour
of the heat kernel, and it has been deeply investigated in \cite{Weaver1}, \cite{Weaver2}, in particular looking for (extended) metric realizations
of measurable distances. Another direction comes from Gromov-Milman's theory of concentration, nicely developed in the recent monograph
\cite{Shioya}: indeed, the notion of pyramid of metric measure spaces investigated in \cite{Shioya} displays some analogy with the monotone
approximation property of the extended distance used in our axiomatization. In addition, the convergence result proved in Section~\ref{sec:stabchee}
at the level of the Cheeger energies (for the special case of the monotone approximation) should be compared with the various convergence
results (which essentially use, instead, the algebra of bounded 1-Lipschitz functions) developed in \cite{Shioya}.

\smallskip
\noindent
{\bf Acknowledgements.} We thank V.~Bogachev and D.~Zaev for very
useful technical and bibliographical suggestions.\\
The first two authors acknowledge the support
of the ERC ADG GeMeThNES. The first and third author 
have been partially supported by PRIN10-11 grant from MIUR for the
project  \emph{Calculus of Variations.} The second author also acknowledes support by the
German Research Foundation through the Collaborative
Research Center 1060 ``The Mathematics of Emergent Effects''.

\section{Preliminaries}\label{sec:1}

\subsection{Measure-theoretic notation,  Radon measures, weak and
  narrow topology}

For a Hausdorff topological space $(X,\tau)$, we denote by $C_b(X)$ the space of
bounded continuous functions $f:(X,\tau)\to\R$ and by $\BorelSets{\tau}$ the Borel
$\sigma$-algebra of $\tau$. Throughout this paper 
\begin{equation}
\text{$\cP(X)$ denotes the class of Radon probability measures in $X$,}
\label{eq:1}
\end{equation}
i.e. Borel probability measures $\mu$ having the property that 
\begin{equation}
\text{for any $B\in\BorelSets{\tau}$ and any $\epsilon>0$ there exists a compact set $K\subset B$ with 
$\mu(B\setminus K)<\epsilon\;$}
\label{eq:21}
\end{equation}
Notice that the Radon property is in general stronger than the tightness one,
for which the inner approximation \eqref{eq:21} is required only for $B=X$. 

Radon measures have stronger additivity and continuity properties in connection with open
sets and lower semicontinuous functions; in particular we shall use this version of the monotone convergence theorem
(see \cite[Lem.~7.2.6]{Bogachev})
\begin{equation}\label{eq:Beppo_Levi_general}
\lim_{i\in I}\int f_i\dd\mu=\int\lim_{i\in I}f_i\dd\mu
\end{equation}
valid for Radon measures $\mu$ and for nondecreasing nets of
$\tau$-lower semicontinuous and equibounded functions $f_i$.
By truncation, we can apply the same property to nondecreasing nets of $\tau$-lower semicontinuous $f_i:X\to [0,\infty]$.
 
By the very definition of Radon topological space \cite[Ch.~II, Sect. 3]{Schwartz73}, 
every Borel measure in a Radon space is Radon: such class of spaces includes
locally compact spaces with a countable base of open sets,
Polish, Lusin and Souslin spaces \cite[Thm.~9 \& 10, p.~122]{Schwartz73}.
In particular the notation \eqref{eq:1} is consistent with the standard one adopted e.g.~in
\cite{AGS11a,Ambrosio-Gigli-Savare08,Vil03}, where Polish or second countable locally compact spaces are
considered.

The narrow topology on $\cP(X)$ can be defined as the coarsest topology for which all maps 
\begin{equation}
  \label{eq:5}
  \mu\mapsto \int h\dd\mu\qquad\text{from $\cP(X)$ into $\R$}
\end{equation}
are lower semicontinuous as $h:X\to \R$ varies in the set of bounded
lower semicontinuous functions \cite[p. 370]{Schwartz73}.
It can be shown \cite[p. 371]{Schwartz73} 
that it is a Hausdorff topology on $\cP(X)$; when
$(X,\tau)$ is completely regular, i.e. 
\begin{equation}
  \label{eq:CR}
  \begin{gathered}
    \text{for any closed set $F\subset X$ and any $x_0\in X\setminus F$}\\
    \text{there exists $f\in C_b(X)$ with $f(x_0)>0$ and $f\equiv 0$ on $F$,} 
  \end{gathered}
\end{equation}
the narrow topology coincides with the usual weak one, induced by the duality with $C_b(X)$.
In fact, in a completely regular space every bounded lower semicontinuous function $h$
is the upper envelope of the directed set $D_h:=\{f\in C_b(X):f\le h\}$, so
that \eqref{eq:Beppo_Levi_general} shows that 
$\int h\dd\mu=\sup\big\{\int f\dd\mu:f\in D_h\big\}$.

One of the advantages to use the narrow topology in $\cP(X)$ when
$(X,\tau)$ is a Hausdorff topological space
is the following sufficient condition for the compactness 
\cite[Theorem 3, p.~379]{Schwartz73}
(in completely regular spaces it 
is a consequence of Prokhorov theorem). 

\begin{theorem}\label{thm:compa-Riesz}
Let $(X,\tau)$ be a Hausdorff topological space. Assume that a collection ${\mathcal M}\subset\cP(X)$ 
is equi-tight, i.e.
\begin{equation}
  \label{eq:22}
  \text{for every $\eps>0$ there exists a compact set $K_\eps\subset
    X$ such that}\quad
  \sup_{\mu\in \mathcal M}\mu(X\setminus K_\eps)\le \eps\; .
\end{equation}
 Then ${\mathcal M}$ has limit points in the class $\cP(X)$ w.r.t.~the 
narrow topology (in particular, the weak topology induced by $C_b(X)$ when $(X,\tau)$ is
completely regular). 
\end{theorem}

\subsection{Transport plans, gluing, optimal transport and duality}

Let  $(X,\tau)$ and $(Y,\sigma)$ be
Hausdorff topological spaces and let $\mu\in \cP(X)$.
We say that a map $f:X\to Y$ is Lusin $\mu$-measurable
(see e.g.~\cite[Chap.~I, Sect.~5]{Schwartz73}) if 
\begin{equation}
  \label{eq:23}
  \begin{gathered}
    \text{for every compact set $K\subset X$ and every $\delta>0$ there
      exists a compact set }
    \\
    \text{$K_\delta\subset K$ such that $f$ restricted to
      $K_\delta$ is continuous and $\mu(K\setminus K_\delta)\le
      \delta$\;.}
  \end{gathered}
\end{equation}
Notice that since $\mu$ is a Radon measure 
the approximation property \eqref{eq:23} holds in fact for 
every $K\in\BorelSets{\tau}$.

If $f$ is Lusin $\mu$-measurable than it is also Borel
$\mu$-measurable (i.e.~$f^{-1}(B)$ is $\mu$-measurable for every
$B\in \BorelSets{\sigma}$); the converse is known to be true 
if $(Y,\sigma)$ is separable and metrizable  \cite[Thm.~5, p.~26]{Schwartz73}.

If $f$ is Lusin $\mu$-measurable 
we denote by $f_\sharp\mu\in \cP(Y)$ the push-forward of $\mu$ via $f$:
it is the Radon measure defined by $f_\sharp \mu(B):=\mu(f^{-1}(B))$ for every 
Borel set $B\in \BorelSets Y$ (Lusin's $\mu$-measurability of $f$ is assumed in order to 
guarantee the Radon property of $f_\sharp\mu$).

If $\mu,\,\nu\in \cP(X)$, we will denote by $\Gamma(\mu,\nu)$
the class of admissible transport plans between $\mu$ and $\nu$, i.e.
Radon  probability measures in 
$\cP(X\times X)$ with marginals $\mu$ and $\nu$ respectively:
\begin{equation}
  \begin{aligned}
    \Gamma(\mu,\nu)=\Big\{\ppi\in \cP(X\times X):\  {}&\pi(A\times
    X)=\mu(A) \ \text{ for all
      $A\in\BorelSets{\tau}$}\\
    &\pi(X\times B)=\nu(B)\ \text{ for all
      $B\in\BorelSets{\tau}$}\ \Big\}\;.
  \end{aligned}
\label{eq:4}
\end{equation}
It is worth noticing that $\Gamma(\mu,\nu)$ is non empty 
(since it contains the unique Radon extension {to $\BorelSets{X\times X}$} of 
the product measure $\mu\times \nu$, see \cite[p. 73]{Schwartz73}),
convex and compact with respect to the narrow topology,
by Theorem \ref{thm:compa-Riesz}.

We shall use the following gluing lemma:

\begin{lemma}[Gluing lemma]\label{lem:gluing}
  Let $I=\{0,1,\ldots,N\}$, $N\ge 2$, and let $(X_i)_{i\in I}$ be Hausdorff topological spaces
  with $\xX:=\Pi_{i\in I}X_i$ and corresponding projection $\mathsf p^i$. Let
  $\ppi_i$, $i=1,\ldots,N$, be Radon measures in
  $X_{i-1}\times X_i$ satisfying the compatibility conditions
  \begin{displaymath}
  \begin{aligned}
    \int \phi(y)\dd\ppi_i(x,y)
    &=\int\phi(x)\dd\ppi_{i+1}(x,y)
      \qquad\forall\,
      \phi\in C_b(X_i)
      \end{aligned}
      \quad\text{for every }i\in\{1,\ldots,N-1\}\;.
    \end{displaymath}
    Then there exist  Radon measures $\ppi$
    in $\xX$ such that
    $(\mathsf p^{i-1},\mathsf p^i)_\sharp \ppi=\ppi_i$ 
    for every $i\in \{1,\ldots,N-1\}$.
    The same property (with obvious modifications) holds 
    in the case $I=\N$ of a countable set of indexes.
\end{lemma}
\begin{proof} When $I$ is finite, 
  the proof is well-known in Polish spaces (see for instance
\cite[Lem.~5.3.2]{Ambrosio-Gigli-Savare08}) via disintegrations; 
however an alternative proof via Hahn-Banach theorem and
Riesz-Markov-Kakutani theorem \cite[Thm.~7.3.10 and Thm.~7.10.4]{Bogachev} 
is possible in compact Hausdorff spaces,
as indicated in \cite[Exer.~7.9]{Vil03} in the case $N=2$, and then proceeding by induction. 
By a simple exhaustion argument the result extends to Radon measures
in Hausdorff topological spaces.

In the case $I=\N$ we argue as in the proof of
\cite[Lem.~5.3.4]{Ambrosio-Gigli-Savare08}, applying the general version of 
Kolmogorov-Prokhorov theorem given in
\cite[Thm.~21, p.~74 and its Corollary p.~81]{Schwartz73}.
\end{proof}

In the class of Radon probability measures in $X$, we consider the optimal transport problem
\begin{equation}
\inf\biggl\{\int_{X\times X}{\sf c}\dd\ppi:\
\ppi\in\Gamma(\mu,\nu)\biggr\}\;
\end{equation}
where $\mathsf c$ is a bounded cost function defined in $X\times X$.
{In the following proposition we denote by $\int_* f\dd\nu$ the inner integral, namely the supremum
of $\int g\dd\nu$ among all $\nu$-measurable $g$ with $g\leq f$ pointwise. Even though 
the natural setting for duality theorems is provided by costs measurable w.r.t. the product $\sigma$-algebra,
we will need to apply the duality theorem with lower semicontinuous costs $\sfc\geq 0$. In this case the duality theorem
still holds, when Radon measures are involved; this can be seen, for instance, proving via \eqref{eq:Beppo_Levi_general} the existence of
$\tilde\sfc\leq\sfc$ measurable w.r.t. to $\BorelSets\tau\times\BorelSets\tau$ with $\int\sfc\dd\ppi=\int\tilde\sfc\dd\ppi$.} 
\begin{proposition}[Duality] \label{prop:duality_Ruschendorf}
Let $\mathsf c:X\times X\to\R$ be a bounded and either
$(\tau\times\tau)$-lower semicontinuous or $\BorelSets\tau\times\BorelSets\tau$-measurable
function. For all $\mu,\,\nu\in\cP(X)$ one has
\begin{equation}\label{eq:duality0}
\inf\biggl\{\int_{X\times X}{\sf c}\dd\ppi:\
\ppi\in\Gamma(\mu,\nu)\biggr\}=\sup\left\{
  \int_*\psi\dd\nu-\int\phi\dd\mu\right\}
\end{equation}
where the supremum runs in any of the following three classes:
\begin{itemize}
\item[(a)] $\phi,\,\psi$ $\BorelSets{\tau}$-measurable 
  with $\psi(y)-\phi(x)\leq {\sf c}(x,y)$ and $\int|\psi|\dd\nu+\int|\phi|\dd\mu<\infty$;
\item[(b)] $\phi,\,\psi$ bounded $\BorelSets{\tau}$-measurable 
  with $\psi(y)-\phi(x)\leq {\sf c}(x,y)$;
\item[(c)] $(\phi,\psi)$ with $\psi(y)=\phi^{{\sf c}}(y)=\inf_x\phi(x)+{\sf c}(x,y)$ and $\phi$ belonging to the class
\begin{equation}\label{eq:good_tests}
{\mathcal F}:=\left\{\phi:X\to [0,\infty):\ \text{$\phi\in C(K)$ for some compact  
    $K\subset X$, $\phi\vert_{X\setminus K}\equiv c\geq\max_K\phi$}\right\}\;.
\end{equation}
\end{itemize}
\end{proposition}

Notice that in the cases (a) and (b) one can replace the inner integral
$\int_*\psi\dd\nu$ with $\int\psi\dd\nu$ in \eqref{eq:duality0}.

\begin{proof} At this level of generality, 
the validity of \eqref{eq:duality0} with the choice (a) has been
proved in {\cite[Prop.~1.31, Thm. 2.14]{Kellerer84}} for both classes of costs $\sfc$.
(b) still follows by \cite[Lem.~1.8]{Kellerer84}:
in fact one can assume that $\phi,\psi$ take values
in the interval $[\inf \mathsf c-\frac 12 \sup\mathsf c,\frac 12 \sup\mathsf c]$. 
Let us pass now from (b) to (c). To this aim we can use once more the Radon assumption on $\mu$, which implies
the validity of Lusin's theorem (see \cite[Thm.~7.1.13]{Bogachev}) to find $\phi_n\in{\mathcal F}$ equibounded
with $\mu(\{\phi\neq\phi_n\})\downarrow 0$ and $\phi_n\geq\phi$
pointwise. 
Since
$$
\int_*\phi_n^{{\sf c}}\dd\nu\geq\int_*\phi^{{\sf c}}\dd\nu\geq\int\psi\dd\nu
$$
we conclude. 
\end{proof}

In some sections of this paper a reference measure $\m\in\cP(X)$ will be fixed; we shall denote by $\cP^a(X)$
the subclass of measures $\mu\ll\m$ and by $\ent$ the relative entropy w.r.t.~$\m$, equal to $+\infty$ on 
$\cP(X)\setminus\cP^a(X)$ and given by
\begin{equation}
\ent(\rho\m):=\int \rho\log\rho\dd\m\label{eq:28}
\end{equation}
otherwise. Since $\m$ will be fixed, we do not emphasize it in the
notations $\cP^a(X)$, $\ent$ and we will also use the short notation $\|f\|_p$ for
$\|f\|_{L^p(X,\m)}$, $p\in [1,\infty]$.

\section{Extended metric spaces}\label{sec:2}

In this section we introduce some basic analytic tools for analysis in
metric spaces. Most of them 
extend in a natural way
to extended metric spaces, defined below.

\begin{definition}[Extended metric spaces]\label{def:luft2}
We say that $(X,\sfd)$ is an extended metric space if $\sfd:X\times X\to [0,\infty]$ is a symmetric function satisfying the
triangle inequality, with $\sfd(x,y)=0$ iff $x=y$.
\end{definition}

When the condition $\sfd(x,y)=0$ iff $x=y$ is weakened to $\sfd(x,x)=0$ we say that $(X,\sfd)$ is a  (extended) semimetric space.
 
The main example of extended semimetric space we have in mind arises from the construction
\begin{equation}\label{eq:basic_construction}
\sfd(x,y):=\sup_{f\in\mathcal A}|f(x)-f(y)|\;,
\end{equation}
where $\mathcal A$ is any class of bounded functions $f:X\to\R$; it is an extended metric space iff 
$\mathcal A$ separates the points of $X$. Other natural examples arise by the construction of length distances 
(derived from distances or from action minimization),  which need not to be finite.

Since an extended metric space can be seen as the disjoint union of the equivalence classes induced by the
equivalence relation
$$
x\sim y\qquad\Longleftrightarrow\qquad \sfd(x,y)<\infty
$$
and since the equivalence classes are indeed metric spaces, many results and definitions extend with no effort to
extended metric spaces. For instance, we say that an extended metric space $(X,\sfd)$ is complete (resp. geodesic, length,...)
if all metric spaces $X_{[x]}=\{y:\ y\sim x\}$ are complete (resp. geodesic, length,...). In particular any extended metric space
has a unique, up to isometries, (extended) metric completion.

\subsection{Absolutely continuous curves and upper gradients}
For $D\subset\R$ we denote by $X^D$ the space of maps  $\eta:D\to X$ (no continuity is required in general) and
by $\e_t:X^D\to X$ the evaluation maps at time $t\in D$, namely $\e_t(\eta):=\eta(t)$. 

Let $UC(D;(X,\sfd))\subset X^D$ be the space of $\sfd$-uniformly continuous maps.
For $p\in (1,\infty)$, we denote $AC^p(D;(X,\sfd))\subset UC(D;(X,\sfd))$ the subspace of $p$-absolutely continuous maps w.r.t.~$\sfd$, 
satisfying
\begin{equation}
\Action_p(\eta,D):=\sup\left\{\sum_{i=0}^{n-1}\frac{\sfd^p(\eta(t_{i+1}),\eta(t_i))}{(t_{i+1}-t_i)^{p-1}}:\
t_0<t_1<\cdots<t_{n-1}<t_n,\,\,t_i\in D\right\}
<\infty\;.\label{eq:19}
\end{equation}
Notice that, thanks to the triangle inequality, the supremum above can also be realized as a monotone nondecreasing 
limit in the directed set of partitions of $D$, with the order induced by the set-theoretic inclusion.
When the domain $D$ is clear from the context we simplify the notation, writing $\Action_p(\eta)$. If $(X,\sfd)$ is complete, it
is easily seen that any $\eta\in UC(D;(X,\sfd))$ has a unique $\sfd$-continuous extension to $\overline{D}$, and
this extension is still $\sfd$-uniformly continuous; if $\eta\in AC^p(D;(X,\sfd))$ then the extension belongs to
$AC^p(\overline{D};(X,\sfd))$ and the $p$-action $\Action_p$ remains the same.
The metric derivative $|\dot\eta|:(0,T)\to [0,\infty]$ of $\eta\in AC^p([0,T];(X,\sfd))$ is the Borel map defined by
$$
|\dot\eta|(t):=\limsup_{s\to t}\frac{\sfd(\eta(s),\eta(t))}{|s-t|}
$$
and it can be proved (see for instance \cite[Thm.~1.1.2]{Ambrosio-Gigli-Savare08}) that the $\limsup$ is a limit a.e. in $(0,T)$, 
that $|\dot\eta|\in L^p(0,T)$ and that $\sfd(\eta(s),\eta(t))\leq\int_s^t|\dot\eta|(r)\dd r$ for all $0\leq s\leq t\leq T$. Furthermore, $|\dot\eta|$
is the smallest $L^1$ function with this property, up to Lebesgue negligible sets, and one can easily prove that
\begin{equation}\label{eq:Dec26}
\Action_p(\eta)=\int_0^T|\dot\eta(t)|^p\dd t\qquad\forall \eta\in AC^p([0,T];(X,\sfd))\;.
\end{equation}
We also recall the notion of upper gradient: we say that $g:X\to [0,\infty]$ is an upper gradient of $f:X\to\R$ (relative to $\sfd$) if
$t\mapsto g(\eta(t))|\dot\eta |(t)$ is Lebesgue measurable in $(0,1)$ and
$$
|f(\eta(1))-f(\eta(0))|\leq\int_0^1g(\eta(r))|\dot\eta |(r)\dd r
$$
for any $\eta\in AC^p([0,1];(X,\sfd))$, $p\in (1,\infty)$ (the dependence on $p$ is harmless, thanks to reparameterizations). 

\subsection{Gradient flows}
Now we introduce the main concepts of gradient flows used in this
paper: the first one is based on the energy-dissipation inequality
and the second is  characterized by a family of 
evolution variational inequalities, 
see \cite{Ambrosio-Gigli-Savare08} and \cite{Daneri-Savare} for much more on this topic. 
Both can be easily generalized to the extended setting and the case of gradient flows in Hilbert spaces,
detailed in Proposition~\ref{prop:evik}, is a very particular and important example.

Let $F:X\to\R\cup\{+\infty\}$ with (non empty) domain $D(F)=\{F<\infty\}$. The slope $|\rmD F|$  
and the descending slope
$|\rmD^- F|(x)$ of $F$ at $x\in D(F)$ are respectively defined by
\begin{equation}\label{eq:descending_slope}
|\rmD F|(x):=\limsup_{y\to x}\frac{|F(y)-F(x)|}{\sfd(y,x)}\;,\qquad
|\rmD^- F|(x):=\limsup_{y\to x}\frac{[F(y)-F(x)]^-}{\sfd(y,x)}\;,
\end{equation}
with the convention $|\rmD F|(x)=|\rmD^- F|(x)=0$ if $x$ is a $\sfd$-isolated point.

We say that a locally
absolutely continuous curve $x\in AC^2_{\rm loc}((0,\infty);(D(F),\sfd))$ is 
a metric gradient curve of $F$ in the energy-dissipation sense 
if the \emph{Energy Dissipation Inequality} 
\begin{equation}\label{eq:EDI}
  \tag{\EDI}
  F(x(t))+\frac 12\int_s^t \Big(|\dot{x}|^2(r)+|\rmD^-
  F|^2(x(r))\Big)\dd r\leq F(x(s))
\end{equation}
holds for all $s,\,t\in (0,\infty)$ with $s<t$ and also for $s=0$ when the curve
$x$ admits a continuous extension (still denoted by $x$)
to $[0,\infty)$ (recall the definition \eqref{eq:descending_slope} of $|\rmD^- F|$). 

Notice that if $|\rmD^- F|$ is an upper gradient of $F$ we can use the inequality
$$
F(x(s))\leq F(x(t))+\int_s^t|\rmD^- F|(x(r))|\dot{x}(r)|\dd r
$$
to obtain that equality holds in \eqref{eq:EDI}, that $F\circ x$ is locally absolutely continuous in $(0,\infty)$
and that $|\dot{x}|^2=|\rmD^- F|^2\circ x=-\bigl(F\circ x\bigr)'$ a.e. in $(0,\infty)$. {When $x$
has a continuous extension to $[0,\infty)$ and $F(x_0)<\infty$, the absolute continuity holds in all
compact intervals of $[0,\infty)$.}

For $K\in\R$ we introduce the function
$\rmI_K:[0,\infty)\to[0,\infty)$ defined by
$$
\text{$\rmI_0(t)=t$,} \qquad \rmI_{K}(t):=\int_0^t \e^{Ks}\,\dd s =\frac{\e^{K t}-1}{K}\qquad\text{if $K\neq 0$}\;.
$$

{In the following definition of $\EVI_K$ gradient flow we consider some lower semicontinuity conditions
that are automatically implied by the local absolute continuity of $x$ when $F$ is lower semicontinuous and $\sfd$ is
a finite distance. Moreover, the specification of the initial condition is made in such a way that even initial conditions
not in $\overline{D(F)}$, but at a finite distance from $D(F)$, can be considered.}

\begin{definition}[$\EVI_K$ gradient flows] \label{def:evikgf}
Let $K\in\R$. We say that $x\in AC^2_{\rm loc}((0,\infty);(D(F),\sfd))$ is a
$\EVI_K$ gradient curve of $F$ if $t\mapsto F(x_t)$ is lower semicontinuous in $(0,\infty)$ and 
for all $y\in D(F)$ satisfying $\sfd(y,x_t)<\infty$ for some (and then all) $t\in (0,\infty)$ one has 
 \begin{align}\label{def:evik}
   \tag{\EVI$_K$}
   \ddtr\frac12 \sfd^2(x_t,y) +
    \frac{K}{2}\sfd^2(x_t,y)\leq F(y) - F(x_t)\qquad\forall t>0\;,
  \end{align}
where $\mathrm d^+/\mathrm d t$ denotes the upper right derivative. \\
If $\bar x$ has finite distance from $D(F)$, we say that $x$ starts from $\bar x$
if $\liminf_{t\downarrow 0}F(x_t)\ge F(\bar x)$ and 
$\lim_{t\downarrow0}\sfd(x_t,y)\to \sfd(\bar x,y)$ for every 
{$y\in \overline{D(F)}$}.
\end{definition}

{Let us now point out some direct consequences of the definition of $\EVI_K$ gradient curve.

\noindent
{\it Monotonicity of the energy, uniqueness and contractivity.} It is not difficult to show that, for $\EVI_K$ gradient curves $x$, the map 
$t\mapsto F(x_t)$ is non-increasing in $(0,\infty)$ and $F(x_t)\to F(\bar x)$ if $x$ starts from $\bar x$. 
Moreover one has the contractivity property
\begin{displaymath}
\sfd(x_{t+s},y_{t'+s})\leq \rme^{-Ks}\sfd(x_t,y_{t'})\qquad s,\,t,\,t'>0 \; .
\end{displaymath}
By approximation this inequality can be extended to the case when either $t=0$ or $t'=0$ (but not both), 
for $\EVI_K$ curves starting respectively from $\bar x$, $\bar y$.

If $t=t'=0$ the inequality holds in the weaker form
\begin{displaymath}
\sfd(x_s,y_s)\leq \rme^{-Ks}\inf_{y\in D(F)}\sfd(\bar x,y)+\sfd(y,\bar y)\qquad s\geq 0 \;,
\end{displaymath}
which reduces to the standard one when at least one of the initial points belongs to $\overline{D(F)}$.
In particular the $\EVI_K$ gradient curve starting from  $\bar x\in \overline{D(F)}$  
is unique and satisfies $\sfd(x(t),\bar x)\to0$ as $t\downarrow0$.

\noindent 
{\it Integral version and regularization.} Integrating in $(0,t)$ the differential inequality \eqref{def:evik} written in the form
    $\ddtr \bigl(\frac {\rme^{Kt}}2\sfd^2(x_t,y)\bigr)\leq \rme^{Kt}(F(y)-F({\sf S}_tx))$ 
    and using the monotonicity of $F({\sf S}_tx)$  and the convergence
    $\sfd(x_s,y)\to \sfd(\bar x,y)$ as $s\downarrow 0$ we get 
    \begin{equation}
      \label{eq:75}
      \frac 12 \sfd^2(x_t,y)-\frac{\rme^{-Kt }}2\sfd^2(\bar x,y)\le
      \rmI_{-K}(t)\big(F(y)-F(x_t)\big)
    \end{equation}
    for every $y\in D(F)$ with $\sfd(\bar x,y)<\infty$, $t>0$.
    
    In particular we obtain the regularization estimate (see also \eqref{eq:regulaK} below)
    \begin{equation}\label{eq:may_31}
    F(x_t)\leq F(y)+\frac{1}{2{\mathrm I}_K(t)}\sfd^2(\bar x,y)\qquad\forall y\in D(F),\,t>0\ . 
    \end{equation}

\begin{definition}[$\EVI_K$-semigroup of $F$]\label{def:evi_semigroup}
Let $D\subset X$ and let $\sfS:D\times [0,\infty)\to D$ be a semigroup.
We say that $\sfS$ is an $\EVI_K$ gradient flow of $F$ in $D$ if: 
\begin{itemize}
\item[(i)] $D\supset\overline{D(F)}$ and every $x\in D$ has finite distance from $D(F)$;
\item[(ii)] for every $x\in D$ one has that $t\mapsto \sfS_t x$ is an
  {{\rm EVI}$_K$} gradient curve of $F$ starting from $x$.
\item[(iii)] $\sfS$ is $K$-contractive in $D$, i.e.~$\sfd(\sfS_t
  x,\sfS_t y)\le \rme^{-K t}\sfd(x,y)$ for every $x,\,y\in D$. 
\end{itemize}
\end{definition}
It is not hard to show that if (i) and (ii) hold, then
condition (iii) is always satisfied in
$\overline{D(F)}$, so that 
it is usually omitted in all the cases when $D=\overline{D(F)}$,
in particular when $D(F)$ is dense in $X$.

In general existence is much harder to
prove and depends on structural properties of $(X,\sfd)$ and $F$. 
A particularly important case is provided by
lower semicontinuous convex functionals in Hilbert spaces, that we are now briefly recalling
(see for instance \cite[Sec.~1.4]{Ambrosio-Gigli-Savare08} 
for the ``metric'' approach and \cite{Brezis73} for the classic formulation via maximal monotone
operators).

If $(X,\|\cdot\|)$ is a  Hilbert space,
recall that the subdifferential $\partial F(x)$ of a convex functional $F:X\to (-\infty,\infty]$ at $x\in D(F)$
is the convex closed set (possibly empty)
$$
\partial F(x):=\left\{\xi\in X:\ F(y)\geq F(x)+\langle \xi,y-x\rangle\,\,\forall y\in X\right\}\;.
$$

\begin{proposition}[Gradient flows of convex functionals in Hilbert spaces]\label{prop:evik}
Assume that $X$ is a Hilbert space, with distance $\sfd$ induced by the norm, and that
$F:X\to [0,\infty]$ is convex and lower semicontinuous. Then, the following properties hold:
\begin{itemize}
\item[(a)] {the concepts of metric gradient curve and $\EVI_0$ gradient flow starting from $\bar x$
coincide, for locally absolutely continuous curves $x:[0,\infty)\to X$ with $x_0=\bar x\in\overline{D(F)}$;} 
\item[(b)] For all $\bar x\in\overline{D(F)}$, there exists a unique metric gradient curve $x(t)$ of $F$
starting from $\bar x$. The map $\sfS_t:\bar x\mapsto x(t)$ defines an $\mathsf{EVI}_0$
semigroup of $F$ in $\overline{D(F)}$.
\item[(c)] If $x(t)$ is a metric gradient curve of $F$, then $-x_+'(t)=\lim_{h\downarrow 0}(x(t)-x(t+h))/h$ 
exists for all $t>0$ and it coincides with the element with minimal norm of $\partial F(x(t))$.
\item[(d)] If $x(t),\,y(t)$ are metric gradient curves of $F$, then $\sfd(x(t),y(t))\leq\sfd(x(s),y(s))$, $0<s\leq t<\infty$.
\item[(e)] If $x(t)$ is the metric gradient curve starting from $\bar x$ the following regularization estimates hold:
\begin{equation}\label{eq:regulaK}
F(x(t))\leq F(z)+\frac {1}{2t}\sfd^2(z,\bar x)\;,\quad
|\rmD^- F|^2(x(t))\leq |\rmD^- F|^2(\bar x)+ \frac 1{t^2} \sfd^2(z,\bar x)
\end{equation}
for all $z\in D(F)$ and $t>0$.
\end{itemize}
\end{proposition}
\subsection{EVI flows, length distances and geodesic convexity}
  
If $(X,\sfd)$ is an extended metric space, 
we denote by $\sfd_\ell\geq\sfd$ the 
extended length distance associated to
$\sfd$: since we are not making any completeness assumption at this
level, it can be defined as
\begin{equation}
  \label{eq:71}
  \begin{aligned}
    \sfd_\ell(y,z):={}&\sup_{\eps>0}\sfd^\eps(y,z)=\lim_{\eps\downarrow0}\sfd^\eps(y,z),\quad\text{where}\\
    \sfd^\eps(y,z):={}&\inf\Big\{\sum_{n=1}^N\sfd(x_n,x_{n-1}):x_0=y,\
    x_N=z,\ \sfd(x_{n-1},x_n)\le \eps\Big\}\;.
  \end{aligned}
\end{equation}
A second way to generate a length distance from $\sfd$
consists in minimizing the length of all the absolutely continuous curves 
connecting two points: we set
\begin{equation}
  \label{eq:81}
  \bar \sfd_\ell(y,z):=
  \inf\Big\{\int_0^1 |\dot x|(t)\dd t:\ x\in AC([0,1];(X,\sfd)),\ x_0=y,\ x_1=z\Big\}\; .
\end{equation}
Notice that if $x\in AC([0,1];(X,\sfd))$ is an absolutely continuous
curve connecting $y$ to $z$ we easily get
\begin{equation}
  \label{eq:72}
  \sfd_\ell(y,z)\le \int_0^1 |\dot x|(t)\dd t\;,
  \quad\text{so that }\sfd_\ell\le \bar\sfd_\ell\; .
\end{equation}
Motivated by \eqref{eq:72} we call $\bar\sfd_\ell$ 
the \textit{upper} length distance associated to $\sfd$.

$(X,\sfd)$ is called a \emph{length} space if $\sfd_\ell=\sfd$.
It is not difficult to check that $\sfd$ is a length distance
if and only if $\text{for every }x,\,y\in X$ with $\sfd(x,y)<\infty$
and all $\eps>0$ there exists an $\eps$-middle point $z_\eps$, characterized by
\begin{equation}
  \label{eq:73}
  \sfd(x,z_\eps)\le \frac 12\sfd(y,z)+\eps\;,\quad
  \sfd(z_\eps,y)\le \frac 12\sfd(y,z)+\eps\;.
\end{equation}
With this in mind, one can easily check that $\sfd_\ell$
and $\bar \sfd_\ell$ are length distances. If moreover $(X,\sfd)$ is complete one has $\sfd_\ell=\bar\sfd_\ell$ ({a simple
proof can be achieved by selecting, given $x$ and $y$, $\eps/4^i$-midpoints, $i\geq 0$, recursively; thanks to the completeness, 
the process converges to an absolutely continuous curve from $x$ to $y$ with length less than $\sfd_\ell(x,y)+2\eps$}).
In the next theorem we describe a new self-improvement principle for
$\EVI_K$ gradient flows.

\begin{theorem}[Self-improvement of $\EVI_K$]\label{thm:self-improvement}
  Let $F:X\to\R\cup\{+\infty\}$ be a proper functional and 
  let $\sfS$ be an $\EVI_K$ gradient flow of $F$ in $X$ relative to $\sfd$. 
  Then $\sfS$ is an $\EVI_K$ gradient flow of $F$ relative to $\sfd_\ell$ and $\bar\sfd_\ell$.
\end{theorem}
  \begin{proof}
    We discuss the case of $\sfd_\ell$; the proof for $\bar \sfd_\ell$ is analogous, working with continuous rather than 
    discrete curves.
    
    Let us fix $y\in D(F)$, $x\in X$ with $\sfd_\ell(x,y)<\infty$; 
    the existence of $\eps$-middle points as in \eqref{eq:73} for arbitrary couples at finite
    $\sfd_\ell$-distance easily shows that for every $\eps>0$ and
    $N\geq 1$ there exist points $x_i$, $i=0,\ldots,N$, 
    with $x_0=y$, $x_N=x$ and ${\sfd}(x_i,x_{i+1})\le \sfd_\ell(x,y) 2^{\eps}/N$, $i=0,\ldots,(N-1)$.
    Set $s_i:=i/N$, $x_i^t:={\sf S}_{t s_i}x_i$, and $\delta=1/N$;
    notice that for $t>0$ we have $x^t_n\in D(F)$, by our assumptions on ${\sf S}$. 
    Using first the contractivity of ${\sf S}$ and then
    \eqref{eq:75} (with $y=x^t_i={\sf S}_{ts_i}x_i$) we obtain
  \begin{align}\nonumber
    &\frac1{2}\sfd^2(x^t_{i+1},x^t_i)-\rme^{-Kt \delta}\rme^{-2K ts_i}
      \frac12 {\sfd^2}(x_{i+1},x_i)\\\nonumber
     &\leq\frac1{2}\sfd^2({\sf S}_{t\delta}{\sf
       S}_{ts_i}x_{i+1},{\sf S}_{ts_i}x_i)-
         \rme^{-Kt\delta}
         \frac 12 {\sfd^2}({\sf S}_{ts_i}x_{i+1},{\sf S}_{ts_i}x_i)
          \\\label{eq:small-triangle}
     &\leq 
     \rmI_{{-K}}(t\delta) \big[F(x^t_i) - F(x^t_{i+1})\big]\;.
 \end{align}
 Since $\delta^{-1}=N$ we have 
 \begin{align}
   \label{eq:76}
   \sum_{i=0}^{N-1}\frac1{\delta}\sfd^2(x^t_{{i+1}},x^t_{i}) 
       &\geq\left(\sum_{i=0}^{N-1}\sfd(x^t_{{i+1}},x^t_{i})\right)^2\;\\
           \sum_{i=0}^{N-1}\frac1{\delta}\sfd^2(x^t_{{i+1}},x^t_{i}) 
    &\geq 
    \max_i \frac{\sfd^2(x^t_{{i+1}},x^t_{i})}{\delta}
    \label{eq:77}
  \end{align}
  and
  \begin{align*}
   \sum_{i=0}^{N-1}\frac1{\delta}\rme^{-K t\delta}
   \rme^{-2Kts_i}
    {\sfd^2}(x_{i+1},x_i)
   &\le 2^{2\eps} \rme^{-Kt\delta} \sfd_\ell^2(x,y)\sum_i \delta
   \rme^{-2Kts_i}
   \\&
   \le 2^{2\eps}  \rme^{-Kt\delta}\sfd_\ell^2(x,y)\Big(\int_0^1\e^{-2Kts}\dd s+\omega(\delta)\Big)
 \end{align*}
 with $\omega(\delta)\to0$ as $\delta\to0$.
 Thus, dividing by $\delta$ and summing up \eqref{eq:small-triangle} we obtain
 \begin{align*}
   \frac 12\sum_{i=0}^{N-1}\frac1{\delta}\sfd^2(x^t_{{i+1}},x^t_{i}) 
       \le 
       2^{2\eps}  \rme^{-Kt\delta}\frac 12\sfd_\ell^2(x,y)\Big(t^{-1}\rmI_{-2K}(t)+\omega(\delta)\Big)
       +\frac{\rmI_{-K}(\delta t)}{\delta}\big( F(y) - F({\sf S}_tx)\big)\;.
 \end{align*}
 Taking the limit along a family of $\eps\downarrow0$ and $N\uparrow\infty$,
 \eqref{eq:76} and \eqref{eq:77} show that the $\liminf$ of the left hand side
 of the last inequality provides un upper bound of 
 $\frac 12\sfd_\ell^2(\sfS_t x,y)$, and this yields
 $$
 \frac1{2}\sfd_\ell^2({\sf S}_tx,y) -\frac{\rmI_{-2K}(t)}{2t}\sfd^2_\ell(x,y)\leq {t}(F(y) - F({\sf S}_tx))\;.
 $$
 This implies the differential inequality at $t=0$ and then \eqref{def:evik} for the extended distance
 $\sfd_\ell$, thanks to the semigroup property of $\sfS$.
 \end{proof}
\begin{corollary}[Approximate geodesic convexity of $F$]
  \label{cor:appgeo}
  Under the same assumption of the previous Theorem, 
  let us choose $x,\,y\in D(F)$ with $\sfd_\ell(x,y)<\infty$, $\eps>0$, and points 
  $x_0,\ldots,x_N$, $N\in \N$, corresponding to a uniform 
  partition $s_n=n/N$ of $[0,1]$ such that
  \begin{equation}
    \label{eq:78}
    x_0=x,\ x_N=y,\quad
    {\sfd_\ell(x_i,x_{i+1})\le \frac 1N\sqrt{\sfd^2_\ell(x,y)+\eps^2}}\,\,\,\,\,\,0\leq i\leq (N-1)\;.\quad 
  \end{equation}
  Then for all $t>0$ one has
  \begin{equation}
    \label{eq:79}
    F(\sfS_tx_n)\le (1-s_n)F(x)+s_n F(y)-\frac{K}2s_n(1-s_n)\sfd_\ell^2(x,y)
    +\frac {\eps^2}{2\rmI_{K}(t)}s_n(1-s_n)\;.
  \end{equation}
  In particular, $(D(F),\sfd_\ell)$ and all sublevels of $F$ are length spaces. 
  If moreover the sublevels $\{F\leq c\}$, $c\in\R$, are complete in $(X,\sfd)$,
  then $\sfd_\ell=\bar\sfd_\ell$ on $D(F)\times D(F)$.
\end{corollary}
\begin{proof}
The inequality \eqref{eq:79} can be obtained from the $\EVI_K$ property {relative to $\sfd_\ell$ (whose
validity is ensured by Theorem~\ref{thm:self-improvement})} arguing as in \cite[Thm.~3.2]{Daneri-Savare1}.
The equality $\sfd_\ell=\bar\sfd_\ell$ on the sublevels and then on $D(F)$ follows by the equality between the
two length distances induced by complete distances.
\end{proof}
An important application of the above result, that we will exploit in Section~\ref{sec:actionnew}
{(specifically, with $F=\ent$ and $\tau$ equal to the weak $L^1(X,\m)$-topology)}
concerns the case when the completeness of the sublevels $\{F\leq c\}$ can be improved to compactness with respect to a Hausdorff
topology $\tau$ on $X$ for which $\sfd$ is lower semicontinuous. In this case it is not difficult to prove  that $(D(F),\sfd_\ell)$ is a
\emph{geodesic} space and that (see \cite{Daneri-Savare}) the $\EVI_K$ property relative to $\sfd_\ell$ yields 
$K$-convexity of $F$, i.e.~for every $x,\,y\in D(F)$ there exists a constant speed geodesic
$x:[0,1]\to X$ relative to $\sfd_\ell$ connecting $x$ to $y$ such that
\begin{equation}
  \label{eq:80}
  \sfd_\ell(x_s,x_t)=|t-s|\sfd_\ell(x,y),\quad
  F(x_t)\le (1-t)F(x)+t F(y)-\frac{K}2 t(1-t)\sfd_\ell^2(x,y)\;.
\end{equation}
\section{Extended metric-topological spaces}\label{sec:3}
We axiomatize metric-topological spaces by adding two (somehow competing) 
compatibility conditions between the distance $\sfd$ and the topology $\tau$.

\begin{definition}[Extended metric-topological spaces]\label{def:luft1}
Let $(X,\sfd)$ be an extended metric space and let
$\tau$ be a Hausdorff topology in $X$. 
We say that $(X,\tau,\sfd)$ is an extended
metric-topological space if:
\begin{itemize}
\item[(a)] there exists a family of $(\tau\times\tau)$-continuous bounded semidistances $\sfd_i:X\times X\to [0,\infty)$, $i\in I$, 
with  $\sfd=\sup_i\sfd_i$.
\item[(b)] the topology $\tau$ is generated by the family of functions
\begin{equation}\label{eq:defliptau}
{\rm Lip}_b(X,\tau,\sfd):=\left\{f:X\to\R:\ \text{$f$ is bounded, $\sfd$-Lipschitz, $\tau$-continuous}\right\}.
\end{equation}
\end{itemize}
\end{definition}
For every $L\geq 0$ we will also set
\begin{equation}
  \label{eq:9}
  {\rm Lip}_{b,L}(X,\tau,\sfd):=
  \Big\{f\in {\rm Lip}_b(X,\tau,\sfd):\ |f(x)-f(y)|\le
  L\sfd(x,y)\quad\forall\,x,y\in X\Big\}\;
\end{equation}
Let us make a few comments on the above definition.

\subsubsection*{Boundedness of $\sfd_i$.}
The boundedness assumption on $\sfd_i$ in (a) is clearly not restrictive, possibly replacing $\sfd_i$ by $\sfd_i\wedge n$ and
enlarging the index set. 

\subsubsection*{Directed families of distances.}
Possibly passing from the index set $I$ to the collection of its finite subsets we can assume with no loss of generality
that $I$ is a directed set. We shall often make this assumption in the
sequel. 

\subsubsection*{Extended metric-topological structures generated by 
separating family of functions.}
\label{ssubsec:family}
If we don't take the extended metric structure as a starting 
point, a definition easily seen to be equivalent can be given starting from a class ${\mathcal A}$ of functions which separate the points of $X$; then
$\tau$ is defined as the topology generated by $\mathcal A$ and the extended distance $\sfd(x,y)$ can be obtained by taking the supremum
of $|f(x)-f(y)|$ as $f$ runs in $\mathcal A$, as in \eqref{eq:basic_construction}.
Notice that in this case we can take as topology $\tau$ the coarsest topology
that makes all functions in $\mathcal A$ continuous, which is easily seen to be Hausdorff; 
since ${\rm Lip}_b(X,\tau,\sfd)$ contains $\mathcal A$ by construction,
it turns out that condition (a) above is satisfied and condition (b) is satisfied as well, with $\sfd_i(x,y)=|f_i(x)-f_i(y)|$,
$\mathcal A=\{f_i\}_{i\in I}$. In addition, if the functions in $\mathcal A$ are already continuous for some preexisting
topology $\tau_0$ in $X$, this construction provides a topology $\tau$ coarser than $\tau_0$.
\subsubsection*{Relations between $\tau$ and $\sfd$.}
Notice that condition (a) yields 
\begin{equation}
  \label{eq:2}
  \sfd\quad\text{is $(\tau\times \tau)$-lower semicontinuous
    in $X\times X$}
\end{equation}
and for every net $(x_j)_{j\in J}$ and every $x\in X$
\begin{equation}
  \text{$\sfd(x_j,x)\to 0$ implies $x_j\to x$ w.r.t.~the topology
    $\tau$}\;
\label{eq:3}
\end{equation} 
Indeed,
one can use assumption (b) and observe that $f(x_j)\to f(x)$ for all $f\in {\rm Lip}_b(X,\tau,\sfd)$.
 
 \subsubsection*{Lipschitz functions generating $\tau$.}
We might equivalently express condition (b) by assuming that
there exists a smaller family $\mathcal F\subset
\mathrm{Lip}(X,\tau,\sfd)$ which generates the topology $\tau$, 
i.e.~for every net $(x_j)_{j\in J}$ in $X$ and
every $x\in X$
\begin{equation}
  \label{eq:6}
  \text{$f(x_j)\to f(x)$ for all $f\in\mathcal F$ implies $x_j\to x$ w.r.t.~the topology
    $\tau$}\;.
\end{equation}
In fact, by suitably modifying the set of distances $\sfd_i$
without changing $\tau$ and $\sfd$, we can obtain an
equivalent characterization of extended metric-topological spaces:
\begin{lemma}
  \label{lem:obvious-but-useful}
  $(X,\tau,\sfd)$ is an extended metric-topological space
  according to Definition~\ref{def:luft1} 
  if and only if $(X,\sfd)$ is an extended metric space
  and there exists a family of $(\tau\times\tau)$-continuous
  and bounded semidistances $\sfd_i$, $i\in I$, 
  such that $\sfd=\sup_{i\in I}\sfd_i$ and 
  for every net $(x_{j})_{j\in J}$ in $X$ and $x\in X$
  \begin{equation}
    \label{eq:8}
    \lim_{j\in J}x_j=x\quad \text{w.r.t.~}\tau
    \qquad\Longleftrightarrow\qquad
    \lim_{j\in J}\sfd_i(x_j,x)=0\quad \text{for every }i\in I\;.
  \end{equation}
\end{lemma}
\begin{proof}
  If $(X,\tau,\sfd)$ is an extended metric-topological space
  according to Definition~\ref{def:luft1} we can simply 
  consider the new collection of semidistances 
  of the form
  \begin{equation}
    \label{eq:10}
    \sfd_f(x,y):=|f(x)-f(y)|,\qquad f\in \mathcal F={\rm
      Lip}_{b,1}(X,\tau,\sfd)\;
  \end{equation}
  Conversely, if a family of $(\tau\times\tau)$-continuous and bounded 
  semidistances $\sfd_i$, $i\in I$, satisfy $\sfd=\sup_{i\in
    I}\sfd_i$ and \eqref{eq:8}, then it is 
  easy to check that 
  $\cal F:=\big\{\sfd_i(\cdot,z):i\in I,\ z\in X\big\}$ 
  is contained in $\mathrm{Lip}_b(X,\tau,\sfd)$ and
  generates the topology $\tau$. 
\end{proof} 
 
\subsubsection*{Approximation of continuous functions.}
Every $f\in C_b(X)$ admits the useful representation formula
\begin{equation}
  \label{eq:12}
  \begin{aligned}
    f(x)&=\sup_{g\in L^-(f)}g(x)\quad\text{with}\quad L^-(f):=\Big\{g\in \mathrm{Lip}_b(X,\tau,\sfd),\quad g\le
    f\text{ in }X\Big\}\;\\
    f(x)&=\inf_{h\in L^+(f)}h(x)\quad\text{with}\quad L^+(f):=\Big\{h\in \mathrm{Lip}_b(X,\tau,\sfd),\quad h\ge
    f\text{ in }X\Big\}\;
  \end{aligned}
\end{equation}
which can be proved by 
passing to the limit with respect to $(i,n)\in I\times \N$ 
in the $\inf$ and $\sup$ regularizations of $f$
\begin{equation}
  \label{eq:13}
  g_{i,n}(x):=\inf_{y\in X} f(y)+n\sfd_i(x,y),\qquad
  h_{i,n}(x):=\sup_{y\in X} f(y)-n\sfd_i(x,y)
\end{equation}
associated to any family of $(\tau\times\tau)$-continuous distances $\sfd_i$ satisfying
\eqref{eq:8}.

Properties \eqref{eq:2} and \eqref{eq:3} show that  Definition~\ref{def:luft1} is consistent
with the axiomatization of extended metric spaces proposed in
\cite{AGS11a}. Notice however that we assume neither that $\tau$ is Polish
(by working directly with Radon measures)
nor that $(X,\sfd)$ is complete. 
\subsubsection*{Completion.}
Thanks to \eqref{eq:defliptau}, extended metric-topological
spaces behave well w.r.t.~completion: denoting by $(\tilde
X,\tilde\sfd)$ the abstract completion of $(X,\sfd)$, every 
function $f$ in $\mathrm{Lip}_b(X,\tau,\sfd)$ admits a unique
Lipschitz extension $\tilde f$ to $\tilde X$ and we can thus 
introduce the topology $\tilde\tau$ generated
by $\{\tilde f:f\in \mathrm{Lip}_b(X,\tau,\sfd)\}$. It is not difficult
to check that $(\tilde X,\tilde \tau,\tilde\sfd)$ 
is an extended metric-topological space and that
$\tau$ is the restriction of $\tilde\tau$ to $X$.
\subsubsection*{A canonical compactification.}
Consider a set $F$ and the space 
$\mathcal X:=\R^F$ endowed with the topology $\tau_{\mathcal X}$ of pointwise convergence.
On $\mathcal X$ we consider the extended distance
\begin{displaymath}
  \sfd_{\mathcal X}(\xx,\boldsymbol y):=\sup_{f\in F}|x_f-y_f|\quad\text{with}\quad
  \xx=(x_f)_{f\in F},\ 
  \boldsymbol y=(y_f)_{f\in F}\in \mathcal X\;.
\end{displaymath}
$(\cX,\tau_\cX,\sfd_\cX)$ provides a natural class of example of extended
metric-topological space
depending on the index set $F$.
Extended metric-topological spaces can always be embedded
in a compact subset of some space $\calX=\R^F$, preserving the metric-topological structure.
As index set $F$ and embedding $\iota$ we can always choose 
\begin{equation}
  \label{eq:11}
  F=\mathrm{Lip}_{b,1}(X,\tau,\sfd),\quad
  \iota:X\to \R^F,\quad \iota(x)_f:=f(x)\quad f\in F\;
\end{equation}
as is typical for the Stone-$\check {\mathrm C}$ech compactification of completely
regular spaces. The proof of the following result is standard.

\begin{lemma}
  \label{lem:compactification}
  If $(X,\tau,\sfd)$ is an extended metric-topological space according to
  Definition \ref{def:luft1}, then the map $\iota$ in \eqref{eq:11} 
  is an homeomorphism of $X$ with $\tilde X=\iota(X)\subset \cal X=\R^F$,
  it is an isometry from $(X,\sfd)$ to $(\tilde X,\sfd_\calX)$,
  and $\tilde X$ has a $\tau_\cX$-compact closure in $\calX$.
\end{lemma}

One more important consequence, from the 
measure-theoretic point of view, is the following result
concerning the complete regularity, according to \eqref{eq:CR}, of 
metric-topological spaces.
\begin{lemma}[Complete regularity]\label{lem:compreg}
Any extended metric-topological space $(X,\tau,\sfd)$ is completely regular.
\end{lemma}
\begin{proof}
Let $x_0\in U\in\tau$ with $U\cap F=\emptyset$. By condition (b), we can assume
that $U$ has the form $\cap_j\{f_j>0\}$ for some finite family of functions $f_j\in {\rm Lip}_b(X,\tau,\sfd)$. Then the
function $g=\min_j f_j^+$ is null on $F$ and $g(x_0)>0$.
\end{proof}
A first consequence of the $(\tau\times\tau)$-lower semicontinuity of $\sfd$, technically relevant for us, is
the $\tau$-lower semicontinuity of the function
\begin{equation}\label{eq:dKappa}
\sfd_K(x):=\min_{y\in K}\sfd(x,y)
\end{equation}
for any compact set $K\subset X$. The function $\sfd_K$, whose 0 level set is $K$, provides in our context
an analogous of the \textit{perfect} regularity condition in Topology. Notice also that $\sfd_K$ is the monotone limit of the
$\tau$-continuous functions 
\begin{equation}\label{eq:xv_or}
\sfd^i_K(x):=\min_{y\in K}\sfd_i(x,y)\;.
\end{equation}
Let us now introduce some additional concepts where the topological and the metric structure interact.
We denote by $\BorelSets{\tau,\sfd}$ the $\sigma$-algebra generated by ${\rm Lip}_b(X,\tau,\sfd)$;
obviously $\BorelSets{\tau,\sfd}\subset\BorelSets{\tau}$, but a kind of converse is provided by the following lemma.

\begin{lemma} \label{lem:sigmaalgebras} 
In an extended metric-topological space $(X,\tau,\sfd)$, any $\mu\in\cP(X)$ is uniquely determined by its values on
$\BorelSets{\tau,\sfd}$. In addition $\cup_i{\rm Lip}_b(X,\tau,\sfd_i)$ separates points of $X$ and it is dense in $L^2(X,\m)$.
In particular ${\rm Lip}_b(X,\tau,\sfd)$ separates points of $X$ and it is dense in $L^2(X,\m)$.
\end{lemma}
\begin{proof}
For the first statement it is sufficient to remark that the sets $\{\sfd^i_K>0\}$ are $\tau$-open, 
belong to $\BorelSets{\tau,\sfd}$ and monotonically converge to $\{\sfd_K>0\}=X\setminus K$. Then the Radon
property of $\mu$ ensures that $\mu(X\setminus K)=\lim_i\mu(\{\sfd_K^i>0\})$.
Since $\{\sfd(\cdot,x)\}_{x\in X}$ separates points of $X$, the family $\{\sfd_i(\cdot,x)\}_{x\in X,\,i\in I}$
is contained in ${\rm Lip}_b(X,\tau,\sfd)$ and separates points as well.
To prove the density, it is sufficient to show the implication
$$
\int f\phi\dd\m=0\qquad\forall\phi\in\bigcup_{i\in I}{\rm Lip}_b(X,\tau,\sfd_i)
\qquad\Longrightarrow\qquad
f=0\;.
$$
Clearly $\int f\dd\m=0$ and, arguing by contradiction, it is not restrictive to assume $\int |f|\dd\m=2$.
Splitting $f$ in positive and negative part we can consider $\mu^\pm:=f^\pm\m\in\cP(X)$ to obtain that $\mu^+=\mu^-$,
which is a contradiction. 
\end{proof}

On $X^D$ we put the product topology $\tau^{\otimes D}$, i.e. the coarsest topology making all $\e_t$ continuous.
We denote for simplicity this topology by $\tau^*$. Notice that $(X^D,\tau^*)$ is Hausdorff, because 
$\{\e_t\}_{t\in D}$ separates points of $X^D$; if $(X,\tau)$ is completely regular then $(X^D,\tau^*)$
is completely regular as well, since it is a product of completely regular spaces \cite[Thm.~33.2]{Munkres00}.

By continuity, the push-forward operator induced by $\e_t$
maps Radon measures in $X^D$ to Radon measures in $X$.
Let us now check that the class $UC(D;(X,\sfd))$ of $\sfd$-uniformly continuous paths $\eta:D\to X$ is a Borel
subset of $X^D$. Indeed, it is easily seen that its complement is described by
$$
\bigcup_{k=1}^\infty\bigcap_{\ell=1}^\infty\bigcup_{i\in I}
\left\{\eta:\ \text{$\exists s,\,t\in D$ with $|s-t|\leq\frac 1\ell$, $\sfd_i(\eta(s),\eta(t))>\frac 1k$}
\right\}\;.
$$
On the other hand, for $k$, $\ell$ and $i$ fixed, the complement of the set above is
$$
\bigcap_{s,\,t\in D,\,\ell|s-t|\leq 1}\left\{\eta:\ \sfd_i(\eta(s),\eta(t))\leq\frac 1k\right\}
$$
and therefore is closed in $X^D$ (by the continuity of $(\e_s,\e_t):X^D\rightarrow X\times X$).

It is easy to check that $\sfd\circ (\e_s,\e_t)=\sup_i\sfd_i\circ (\e_s,\e_t)$ is $\tau^*$-lower semicontinuous 
in $X^D$ for all $s,\,t\in D$. This can be used to prove that $\Action_p$ is $\tau^*$-lower semicontinuous in $X^D$ and, 
in particular, that $AC^p(D;(X,\sfd))$ is a Borel subset of $X^D$, more precisely a countable union of closed sets. 
\begin{theorem}[Compactness of probabilities in $X^D$]\label{thm:tightness}
Assume that $p\in (1,\infty)$, $T>0$ and $D\subset\R$ is countable.
Let  $(\eeta_i)_{i\in I}$ be a net of Radon probability measures in $(X^D,\tau^*)$ such that
for all $t\in D$ the family $\{(\e_t)_\#\eeta_i\}_{i\in I}$ is equi-tight in $\cP(X)$.\\
Then $\{\eeta_i\}_{i\in I}$ has limit points $\eeta\in\cP(X^D)$ in the weak topology induced by $C_b(X^D)$ and any 
such limit point, along a subnet $\beta:L\to I$, satisfies:
\begin{equation}\label{eq:lscaction}
\int \Action_p(\eta,D')\dd\eeta(\eta)\leq\liminf_{\ell\in L}
\int \Action_p(\eta,D')\dd\eeta_{\beta(\ell)}(\eta)\qquad\forall D'\subset D\;.
\end{equation}
\end{theorem}
\begin{proof} The family $\{\eeta_i\}_{i\in I}$ is equi-tight: indeed, enumerating by $\{d_k\}_{k\in\N}$ the elements
of $D$, it suffices to find compact sets $K_{k,n}$ such that $\sup_i(\e_{d_k})_\#\eeta_i(X\setminus K_{k,n})\leq 2^{-k-n}$ and to consider the
sets
$$
\Gamma_n:=\bigcap_{k=0}^\infty\big\{\eta\in X^D:\ \eta(d_k)\in K_{k,n}\big\}
$$
which are compact in $X^D$ and satisfy $\sup_i\eeta_i(X^D\setminus\Gamma_n)\leq 2^{1-n}$. It follows that
we can apply Theorem~\ref{thm:compa-Riesz}.

In order to prove \eqref{eq:lscaction} for a limit point $\eeta\in\cP(X)$ we use the continuity of $\eta\mapsto
\sfd_i(\eta(s),\eta(t))$ and the Radon property of $\eeta$ to get
$$
\int\sfd^p(\eta(s),\eta(t))\dd\eeta(\eta)\leq\liminf_{\ell\in L}
\int\sfd^p(\eta(s),\eta(t))\dd\eeta_{\beta(\ell)}(\eta)\;.
$$
Then, given $D'\subset D$, from the superadditivity of $\liminf$ we obtain, for all choices of $t_0,\ldots,t_n\in D'$,
\begin{eqnarray*}
\int\sum_{j=0}^{n-1}\frac{\sfd^p(\eta(t_{j+1}),\eta(t_j))}{(t_{j+1}-t_j)^{p-1}}\dd\eeta(\eta)&\leq&\liminf_{\ell\in L}
\int\sum_{j=0}^{n-1}\frac{\sfd^p(\eta(t_{j+1}),\eta(t_j))}{(t_{j+1}-t_j)^{p-1}}\dd\eeta_{\beta(\ell)}(\eta)\\&\leq&
\liminf_{\ell\in L}\int\Action_p(\eta,D')\dd\eeta_{\beta(\ell)}(\eta)\;.
\end{eqnarray*}
Since the sums $\sum_i \sfd^p(\eta(t_{j+1}),\eta(t_j))/(t_{j+1}-t_j)^{p-1}$ become larger as the partition gets finer,
we can use once more the fact that $\eeta$ is Radon to conclude.
\end{proof}

Finally, we can add a measure structure in extended metric-topological spaces as follows.

\begin{definition}[Extended metric measure space]\label{def:extmm}
We say that $(X,\tau,\sfd,\m)$ is an extended metric measure space if:
\begin{itemize}
\item[(a)] $(X,\tau,\sfd)$ is an extended metric-topological space;
\item[(b)] $\m\in\cP(X)$, i.e. $\m$ is a Radon probability measure in $(X,\BorelSets{\tau})$.
\end{itemize}
\end{definition}

\section{The Wasserstein space over an extended metric-topological space}\label{sec:4}

Throughout this section $(X,\tau,\sfd)$ is an extended metric-topological space.

\subsection{The extended Wasserstein distance between Radon measures}
In the class $\cP(X)$ of Radon probability measures in $X$, we define
the (quadratic) Wasserstein extended  distance $W_\sfd(\mu,\nu)$ by
\begin{equation}
W_\sfd^2(\mu,\nu):=\inf\biggl\{\int_{X\times X}\sfd^2\dd\ppi:\
\ppi\in\Gamma(\mu,\nu)\biggr\}\;
\end{equation}
where we recall that $\Gamma(\mu,\nu)$ is the class of admissible
transport plans between $\mu$ and $\nu$, see \eqref{eq:4}. 

In our context, since $(X,\tau)$ is completely regular thanks to
Lemma~\ref{lem:compreg}, we know that narrow and weak topology coincide. 
From now on, unless otherwise stated, by weak convergence we always mean convergence in the duality with $C_b$,
the corresponding topology will be denoted by $\tau_{\cP}$.

The lower semicontinuity of the cost $\sfd^2$ and the tightness of the marginals ensure respectively lower semicontinuity of the transportation
cost and compactness w.r.t.~weak convergence of the class of the admissible transport plans $\Gamma(\mu,\nu)$, hence existence
of optimal plans. We provide a more general statement in the next theorem.

\begin{theorem}[Compactness and joint lower semicontinuity]\label{thm:compact_joint}
Let $I$ be a directed set and  assume that $\mu_i,\,\nu_i\in\cP(X)$ weakly converge to $\mu,\,\nu\in\cP(X)$ respectively.  
Then, for any choice of $\ppi_i\in\Gamma(\mu_i,\nu_i)$,  one has:
\begin{itemize}
\item[(a)] the net $(\ppi_i)_{i\in I}$ has limit points w.r.t.~weak convergence and any limit point $\ppi$ belongs to $\Gamma(\mu,\nu)$;
\item[(b)] if $\ppi_i$ weakly converges to $\ppi$, and ${\sf c}_i:X\times X\to [0,\infty]$ is a monotone
family of  $(\tau\times\tau)$-lower semicontinuous functions, then
$$
\liminf_{i\in I}\int {\sf c}_i\dd\ppi_i\geq\int {\sf c}\dd\ppi\qquad\text{with ${\sf c}:=\sup_i{\sf c}_i$}\; ;
$$
\item[(c)] if $\mu_i=\mu$, $\nu_i=\nu$, and if $\sfd_i:X\times X\to [0,\infty]$ is a monotone
family of  $(\tau\times\tau)$-continuous distances with $\lim_i\sfd_i=\sfd$, 
then $W_{\sfd_i}$ monotonically converges to $W_{\sfd}$. 
\end{itemize}
\end{theorem} 
\begin{proof} Statement (a) is a direct consequence of Theorem~\ref{thm:compa-Riesz}, since $(X\times X,\tau\times\tau)$ is
completely regular. Statement (b) follows by \eqref{eq:Beppo_Levi_general}.
Statement (c) follows by (b) and (a) with $\mu_i=\mu$, $\nu_i=\nu$, choosing optimal plans $\ppi_i$
relative to $\sfd_i^2$ and extracting a weakly convergent subnet from the $\ppi_i$.
\end{proof}

We claim that the duality formula
\begin{equation}\label{eq:dualityQ}
\frac 12 W_\sfd^2(\mu,\nu)=\sup_{\phi\in{\mathcal F}}
 \int Q_1\phi\dd\nu-\int\phi\dd\mu\;
\end{equation}
holds, where ${\mathcal F}$ is defined in \eqref{eq:good_tests} and $Q_t\phi$ is defined by the Hopf-Lax formula
\begin{equation}\label{eq:HopfLax}
Q_t\phi(y):=\inf_{x\in X} \phi(x)+\frac 1{2t}\sfd^2(x,y)\;.
\end{equation}
In fact, from the
very definition of $\mathcal F$ in \eqref{eq:good_tests} it follows that
\begin{equation}\label{eq:good_tests1}
Q_t\phi(y)=C\wedge\min_{x\in K}\Big(\phi(x)+\frac 1{2t}\sfd^2(x,y)\Big)
\end{equation}
for some compact set $K\subset X$ with $\phi\in C(K)$ and $C\geq\max_K\phi$, hence $Q_t\phi$ are $\tau$-lower semicontinuous and we can replace
$\int_*Q_1\phi\dd\nu$ with $\int Q_1\phi\dd\nu$. {In addition, the compactness of $K$ ensures that
$$
\min_{x\in K}\Big(\phi(x)+\frac 1{2t}\sfd_i^2(x,y)\Big)\uparrow \min_{x\in K}\Big(\phi(x)+\frac 1{2t}\sfd^2(x,y)\Big)\;
$$
hence from Proposition~\ref{prop:duality_Ruschendorf} (which deals with bounded cost functions) with $\frac 12\sfd_i^2$
and statement (c) of the previous theorem we obtain \eqref{eq:dualityQ}.}

We will occasionally use also the extended Wasserstein distance $W_{\sfd,1}$ in $\cP(X)$ obtained by minimizing $\int\sfd\dd\ppi$
in the class of admissible transport plans, and the corresponding duality formula
\begin{equation}\label{eq:dualityQ1}
W_{\sfd,1}(\mu,\nu)=\sup\left\{\int f \dd (\mu-\nu):\ f\in {\rm Lip}_b(X,\tau,\sfd),\,\,{\rm Lip}(f)\leq 1\right\}\;.
\end{equation}
Also the proof of \eqref{eq:dualityQ1} can be obtained from Proposition~\ref{prop:duality_Ruschendorf} with 
${\sf c}=\sfd_i$, considering
the pairs $(-\phi_i^{{\sf c}},\phi_i^{{\sf c}})$ with
$$
\phi_i^{{\sf c}}(x):=C\wedge\min_{x\in K}\Big(\phi(x)+\sfd_i(x,y)\Big)\in {\rm Lip}_b(X,\tau,\sfd)
$$
and $\phi\in\mathcal F$, $K\subset X$ compact set as in \eqref{eq:good_tests}, $C\geq\max_K\phi$.

\begin{proposition} [$W_{\sfd}$ convergence implies $\tau_{\cP}$ convergence]
\label{prop:in_un_senso}
A net $\{\mu_j\}_{j\in J}$ weakly converges to $\mu$ in $\cP(X)$ if and only if 
\begin{equation}
  \label{eq:18}
    \lim_{j\in J}\int f\dd\mu_j=\int f\dd\mu\quad \text{for every }f\in \mathrm{Lip}(X,\tau,\sfd)\;.
\end{equation}
In particular, if $\{\mu_j\}_{j\in J} $ converges to $\mu$
w.r.t.~$W_\sfd$, then it also converges w.r.t.~the weak topology induced by $C_b(X)$. 
\end{proposition}
\begin{proof} 
Let us first prove that \eqref{eq:18} 
is sufficient to prove the weak convergence of $\mu_j$ in $\cP(X)$ 
(the converse implication is trivial).
If $f\in C_b(X)$ we consider the directed set 
$L^-(f)$ defined in \eqref{eq:12}, obtaining
\begin{displaymath}
  \liminf_{j\in J}\int f\dd\mu_j\ge\lim_{j\in J}\int g\dd\mu_j=
  \int g\dd\mu\quad \text{for every }g\in L^-(f)\;
\end{displaymath}
so that 
\begin{displaymath}
  \liminf_{j\in J}\int f\dd\mu_j\ge\sup_{g\in L^-(f)}
  \int g\dd\mu=\lim_{g\in L^-(f)}\int g\dd\mu=\int f\dd\mu\;
\end{displaymath}
since $L^-(f)$ is a directed set with respect to the natural ordering
of functions and $\mu$ is a Radon measure.
Changing $f$ in $-f$ we get the opposite inequality for the $\limsup$,
thus proving that $\lim_{j\in J}\int f\dd\mu_j=\int f\dd\mu$.

If  $\{\mu_j\}_{j\in J}$ is a net convergent w.r.t.~$W_\sfd$ to $\mu\in \cP(X)$
the inequality (ensured by $W_{\sfd,1}\leq W_\sfd$ and \eqref{eq:dualityQ1}) 
\begin{equation}
\biggl|\int f\dd\mu-\int f\dd\mu_j\biggr|\leq {\rm Lip}(f)W_\sfd(\mu_j,\mu)\label{eq:17}
\end{equation}
yields \eqref{eq:18} and therefore the weak convergence of $\mu_j$.
\end{proof}
{The following result shows the flexibility of our axiomatization: the extended metric-topological
structure can be lifted from the space $X$ to the space of probabilities on $X$.}
\begin{proposition}[Extended metric-topological structure on $\cP(X)$] 
\label{prop:extended-PX}
$(\cP(X),\tau_{\cP},W_\sfd)$ is an extended metric-topological space. 
\end{proposition}
\begin{proof}
Let us first show that condition (a) of Definition~\ref{def:luft1} is satisfied. 

We consider the set
$F= \mathrm{Lip}_{b,1}(X,\tau,\sfd)$ and
we denote by $I$ the collection 
(directed set) of the  finite subsets of $F$
ordered by inclusion.
For every $i\in I$ we can set
\begin{gather*}
  M_i:= \sup_{x\in X}\sup_{f\in i} |f(x)|\;\qquad
  \sfd_i(x,y):=\sup_{f\in i} |f(x)-f(y)|\;.
\end{gather*}
Since we already noticed that $W_{\sfd_i}\uparrow W_\sfd$,
it remains to prove that each semidistance $W_{\sfd_i}$ is
$(\tau_\cP\times\tau_\cP)$-continuous. 
So, for a fixed $i=\{f_1,\ldots,f_N\}\subset F$ we
consider the $\tau$-continuous map $\kappa:X\to X_i=[-M_i,M_i]^N$ and
the distance $\delta$ in $X_i$ defined by
\begin{displaymath}
  \kappa(x):=(f_1(x),\ldots,f_N(x)),\quad 
  \delta(\xx,\yy):=\sup_{n=1,\ldots, N}|x_n-y_n|,\quad \xx,\yy\in \R^N \;,
\end{displaymath}
so that $\kappa$ is $1$-Lipschitz.
Since in $\cP(X_i)$ the weak topology coincides with the topology
induced by the Wasserstein distance $W_\delta$, $\kappa_\sharp$ is continuous from $(\cP(X),\tau_\cP)$ to 
$(\cP(X_i),W_\delta)$; it is therefore sufficient to prove that 
\begin{equation}
  \label{eq:16}
  W_{\sfd_i}(\mu,\nu)\le W_\delta(\kappa_\sharp\mu,\kappa_\sharp\nu)
  \quad
  \text{for every }\mu,\nu\in \cP(X)\;
\end{equation}
which in fact yields the equality, since the opposite inequality is
trivial.
To prove \eqref{eq:16} we apply Proposition~\ref{prop:duality_Ruschendorf} with $\sfc=\sfd_i^2$ to find a 
sequence of uniformly bounded Borel functions $\phi_n,\,\psi_n$
such that 
\begin{displaymath}
  \psi_n(y)-\phi_n(x)\le \sfd_i^2(x,y)\quad \text{for every } x,y\in X,\quad
  W_{\sfd_i}^2(\mu,\nu)=\lim_{n\to\infty}
  \int\psi_n\dd\nu-\int \phi_n\dd\mu\;.
\end{displaymath}
Possibly replacing $\psi_n$ by $\phi_n^\sfc$ and $\phi_n$ by
$\phi_n^{\sfc\sfc}$ and using the fact that $\sfc$-concave functions 
are $\sfd_i$-Lipschitz (and therefore $\tau$-continuous) 
we may assume that $\phi_n,\psi_n\in \mathrm{Lip}_b(X,\sfd_i)$.
We can then define functions
$  \tilde\phi_n,\tilde\psi_n:\kappa(X_i)\to \R$ by
\begin{displaymath}
  \tilde\phi_n(\kappa(x)):=\phi_n(x),
  \quad
  \tilde\psi_n(\kappa(x)):=\psi_n(x)\quad 
  x\in X\;
\end{displaymath}
the definition being consistent since $\kappa(x)=\kappa(y)$ yields
$\sfd_i(x,y)=0$ and therefore $\phi_n(x)=\phi_n(y)$ and
$\psi_n(x)=\psi_n(y)$. Since moreover
$\delta(\kappa(x),\kappa(y))=\sfd_i(x,y)$, we can easily prove that 
$\tilde\phi_n,\tilde\psi_n$ are Lipschitz continuous w.r.t. $\delta$ in
$\kappa(X)$, and therefore they admit a unique Lipschitz continuous
extension
(still denoted by $\tilde\phi_n,\,\tilde\psi_n$) to the compact set
$\overline{\kappa(X)}$.
Notice that both the supports $\supp(\kappa_\sharp \mu)$ and
$\supp(\kappa_\sharp\nu)$ in $X_i$ are contained in
$\overline{\kappa(X)}.$ 
Moreover the relation
\begin{displaymath}
  \tilde\psi_n(y)-\tilde\phi_n(x)\le \delta(x,y)\quad \text{for every }x,\,y\in \kappa(X)
\end{displaymath}
extends by continuity to $\overline{\kappa(X)}$, so that
\begin{displaymath}
  W_\delta^2(\kappa_\sharp\mu,\kappa_\sharp\nu)\ge \int \tilde\psi_n\dd
  \kappa_\sharp\nu-
  \int\tilde\phi_n\dd\kappa_\sharp\mu=
  \int \psi_n\dd\nu-
  \int\phi_n\dd\mu\;
\end{displaymath}
proving \eqref{eq:16}.

In order to prove that also condition (b) of Definition~\ref{def:luft1} is satisfied,
we just observe that the family $\mathcal F$ of real functions on $\cP(X)$ of the form
\begin{displaymath}
  f[\mu]:\mu\mapsto \int f\dd\mu,\quad f\in \mathrm{Lip}_{b,1}(X,\tau,\sfd)
\end{displaymath}
is included in $\mathrm{Lip}_{b,1}(\cP(X),\tau_\cP,W_\sfd)$ thanks to the very
definition of weak convergence and to \eqref{eq:17}.
On the other hand, Proposition~\ref{prop:in_un_senso}
shows that $\tau_\cP$ is induced by $\mathcal F$.
\end{proof}

\begin{proposition}[Completeness of $(\cP(X),W_\sfd)$]
  \label{prop:completeness}
  If $(X,\tau,\sfd)$ is an extended metric-topological space
  and $(X,\sfd)$ is complete, then also $(\cP(X),W_\sfd)$ is complete.
\end{proposition}
\begin{proof}
  Let $(\mu_n)_{n\in \N}\subset \cP(X)$ be a sequence satisfying
  $\sum_{n=1}^\infty   W_\sfd(\mu_n,\mu_{n+1})<\infty$.
  We argue as in the proof of \cite[Prop.~7.1.5]{Ambrosio-Gigli-Savare08}:
  if $\ppi_n\in \Gamma(\mu_{n},\mu_{n+1})$ are optimal transport plans,
  by applying Lemma~\ref{lem:gluing}
  we can find a Radon measure $\ppi\in \cP(\xX)$, $\xX=X^\N$, 
  such that $(\sfp^{n},\sfp^{n+1})_\sharp \ppi=\ppi_n$, $n\in \N$.
    
  We thus have $\sum_{n=1}^N \int
  \sfd(\sfp^n,\sfp^{n+1})\dd\ppi<\infty$ so that 
  the sequence $n\mapsto \sfp^n(\xx)$ is a Cauchy sequence for 
  $\ppi$-a.e.~$\xx\in\xX$. Denoting by $\sfp(\xx)$ its pointwise limit, 
  and applying Egoroff Theorem \cite[Thm.~6, p.~28]{Schwartz73}
  we can find for every $\eps>0$ compact sets $K_\eps\subset \xX$
  with $\ppi(\xX\setminus K_\eps)<\eps$ 
  such that the restrictions of $\sfp^n$ to $K_\eps$ converge
  uniformly, i.e.~
  \begin{equation}
    \label{eq:15}
    \lim_{n\to\infty}\sup_{\sxx\in
      K_\eps}\sfd(\sfp^n(\xx),\sfp(\xx))=0\;.
  \end{equation}
  Let us prove that $\sfp$ is a Lusin $\ppi$-measurable map by showing that the
  restriction of $\sfp$ to each $K_\eps$ is continuous: 
  the latter property will be a consequence of the fact that
  $\sfd$-uniform limit of $\tau$-continuous maps is $\tau$-continuous.
  In fact, for every $\xx\in K_\eps$ and every $\tau$-neighbourhood 
  $V$ of $\sfp(\xx)$ in $X$, we may find 
  functions $f_j\in \mathrm{Lip}_{b,1}(X,\tau,\sfd)$, $j=1,\ldots,J$,
  and $\delta>0$ such that 
  \begin{displaymath}
    |f_j(\xx)-f_j(\yy)|\le \delta\quad \text{for every $j=1,\ldots,J$}
    \qquad
    \Rightarrow
    \qquad
    \yy\in V\;.
  \end{displaymath}
  If $n\in \N$ is sufficiently big so that $\sup_{\szz\in
    K_\eps}\sfd(\sfp^n(\zz),\sfp(\zz))<\delta/3$
  and $U$ is any neighbourhood in $\xX$ such that 
  every $\yy\in U$ satisfies $|f_j(\sfp^n(\xx))-f_j(\sfp^n(\yy))|<\delta/3$ 
  for every $j=1,\ldots,N$, 
  we conclude that for every $\yy\in U\cap K_\eps$
  \begin{align*}
    |f_j(\sfp(\xx))-f_j(\sfp(\yy))|&\le 
    2\sup_{\szz\in K_\eps}
    \sfd(\sfp(\zz)),\sfp^n(\zz))+|f_j(\sfp^n(\xx))-f_j(\sfp^n(\yy))|\\
    &\le \delta\qquad
    \text{for every $j=1,\ldots,N$\;,}
  \end{align*}
  so that $\sfp(\yy)\in V$.
  
  Being $\sfp$ a Lusin $\ppi$-measurable map, $\mu:=\sfp_\sharp\ppi$
  is a Radon measure in $X$ and
  \begin{displaymath}
    W_\sfd(\mu_n,\mu)\le 
    \bigl(\int \sfd^2(\sfp^n,\sfp)\dd\ppi\bigr)^{1/2}\le 
    \sum_{m=n}^\infty \bigl(\int\sfd^2(\sfp^m,\sfp^{m+1})\dd\ppi\bigr)^{1/2}=
    \sum_{m=n}^\infty W_\sfd(\mu_m,\mu_{m+1})\;,
  \end{displaymath}
  which shows that $\lim_{n\to\infty}W_\sfd(\mu_n,\mu)=0$.
\end{proof}

\subsection{The superposition principle for extended metric-topological spaces}
The next proposition is a small variant of the superposition principle recently proved in \cite[Thm.~3.1]{Li14} for extended metric measure
spaces. We provide a slightly different proof, since no Polish
assumption on $(X,\tau)$ is made here, only the complete regularity of $\tau$ following by Lemma~\ref{lem:compreg} plays a role.

\begin{proposition}[Superposition]\label{prop:lisini}
Assume that $(X,\sfd)$ is complete and
let $\mu_t\in AC^2([0,T];(\cP(X),W_\sfd))$. Then there exists
$\eeta\in\cP(X^{[0,T]})$ concentrated on $AC^2([0,T];(X,\sfd))$ with
 $(e_t)_\#\eeta=\mu_t$ for all $t\in [0,T]$ and
 \begin{equation}\label{eq:sharpmet}
\int |\dot\eta(t)|^2\dd\eeta(\eta)= |\dot\mu_t|^2\qquad\text{for a.e. $t\in (0,T)$\;.} 
 \end{equation}
\end{proposition}
\begin{proof} Let us assume for simplicity $T=1$. We set $D=\cup_n D_n$ with $D_n=\{j/2^n:\ 0\leq j\leq 2^n\}$. 

\noindent
{\bf Step 1.} We build $\ssigma\in\cP(X^D)$ concentrated on $AC^2(D;(X,\sfd))$ satisfying $(\e_t)_\#\ssigma=\mu_t$
for all $t\in D$ and
\begin{equation}\label{eq:lisini1}
\int \sum_{k=i}^{j-1}\frac{|\eta(k/2^n)-\eta((k+1)/2^n)|^2}{2^{-n}}\dd\ssigma(\eta)\leq\int_{i/2^n}^{j/2^n}|\dot\mu_r|^p\dd r
\quad\forall i,\,j\in \{0,\ldots,2^n\},\,\,i<j,\,\,n\geq 1\;.
\end{equation} 
To this aim, taking Theorem~\ref{thm:tightness}, \eqref{eq:lscaction} and the fact that $D_n\uparrow D$ into account, it is sufficient to 
build a family of approximations $\ssigma_n\in\cP(X^D)$ satisfying 
$(\e_t)_\#\ssigma_n=\mu_t$ for all $t\in D_n$ and
\begin{equation}\label{eq:lisini2}
\int \sum_{k=i}^{j-1}\frac{|\eta(k/2^n)-\eta((k+1)/2^n)|^2}{2^{-n}}\dd\ssigma_n(\eta)\leq\int_{i/2^n}^{j/2^n}|\dot\mu_r|^2\dd r
\qquad\forall i,\,j\in \{0,\ldots,2^n\},\,\,i<j\;.
\end{equation}
The construction of $\ssigma_n$ is a simple application of Lemma~\ref{lem:gluing}: it suffices to choose optimal
plans $\ppi_i$ from $\mu_{i/2^n}$ to $\mu_{(i+1)/2^n}$, $0\leq i\leq 2^n$, and then use the lemma to find a Radon probability
measure $\ppi$ in $X^{2^n+1}$ having $\ppi_i$ as double marginals. Then one can define $\ssigma_n\in\cP(X^D)$ 
as the push forward of $\ppi$ via the continuous map from $X^{2^n+1}$ to $X^D$ defined by
$$
(x_0,\ldots,x_n)\mapsto \eta(t)=\begin{cases}
x_i&\text{if $\frac{i}{2^n}\leq t<\frac{i+1}{2^n}$, $0\leq i\leq n-1$}\\
x_n &\text{if $t=1$}\;.
\end{cases}
$$
\noindent
{\bf Step 2.} Given $\ssigma$ as in Step 1, we notice that $\ssigma$ is concentrated on the union of the closed sets
$\Gamma_k=\{\eta:\ \Action_2(\eta,D)\leq k\}$. Since $(X,\sfd)$ is complete we can consider the extension map
${\sf ext}:\cup_k\Gamma_k\to AC^2([0,1];(X,\sfd))$ and build $\eeta$ as the image under $\ssigma$ of ${\sf ext}$. To show
that $\eeta$ is well defined and it is a Radon measure in $\cP(X^{[0,T]})$
we need to show that ${\sf ext}$ (arbitrarily defined out of $\cup_k\Gamma_k)$ is Lusin 
$\ssigma$-measurable. To this purpose, it is sufficient to prove that ${\sf ext}:\Gamma_k\to X^{[0,T]}$ is
continuous. Let $(\eta_i)_{i\in I}$ be a net in $\Gamma_k$ convergent to $\eta$ and let $\tilde\eta_i$, $\tilde\eta$ be
the corresponding extensions to $[0,1]$. By the definition of product topology we need only to prove that
$\tilde\eta_i(t)\to\tilde\eta(t)$ in $(X,\tau)$ for all $t\in [0,1]$. Since $\tau$ is generated by ${\rm Lip}_b(X,\tau,\sfd)$
we need only to prove that $f(\tilde\eta_i(t))\to f(\tilde\eta(t))$ for all $f\in {\rm Lip}_b(X,\tau,\sfd)$. This is trivial if
$t\in D$ (because $\eta_i\to\eta$ in $X^D$), in the general case one can use the estimate
$$
|f(\tilde\eta_i(t))-f(\tilde\eta_i(s))|\leq {\rm Lip}(f)\sfd(\tilde\eta_i(t),\tilde\eta_i(s))\leq
{\rm Lip}(f)\sqrt{ k|t-s|}
$$
and the analogous one for $\tilde\eta$ to conclude.

Having proved that $\eeta$ is well defined $\cP(X^{[0,T]})$
and it is concentrated on the Borel set $AC^2([0,1];(X,\sfd))$,  
we notice that by construction one has $(\e_t)_\#\eeta=\mu_t$
for all $t\in D$. On the other hand, $\mu_t\in AC^2([0,T];(\cP(X),W_\sfd))$
implies, thanks to Proposition~\ref{prop:in_un_senso}, that $t\mapsto\mu_t$ is continuous w.r.t.~$\tau_{\cP}$. Since $\sfd$-convergence implies
$\tau$-convergence also $t\mapsto (\e_t)_\#\eeta$ is continuous w.r.t.~$\tau_{\cP}$, therefore
$(\e_t)_\#\eeta=\mu_t$ for all $t\in [0,1]$. An analogous approximation argument gives
\begin{equation}\label{eq:lisini3}
\int \Action_2(\eta,[s,t])\dd\eeta(\eta)\leq\int_s^t|\dot\mu_r|^2\dd r \qquad\forall s,\,t\in [0,1],\,\,s\leq t
\end{equation} 
starting from \eqref{eq:lisini1}.

Finally, from \eqref{eq:lisini3} and Fubini's theorem we get
\begin{equation}\label{eq:lisini4}
\int |\dot\eta(t)|^2\dd\eeta(\eta)\leq|\dot\mu_t|^2\qquad\text{for a.e. $t\in (0,1)$}\;.
\end{equation}
On the other hand, since $(\e_s,\e_t)_\#\eeta\in \Gamma(\mu_s,\mu_t)$ one has
$$
W_\sfd^2(\mu_s,\mu_t)\leq\int\sfd^2(\eta(s),\eta(t))\dd\eeta(\eta)\leq
(t-s)\int\int_s^t|\dot\eta|^2(r)\dd r\dd\eeta(\eta)\;,
$$
hence for a.e. $t\in (0,1)$ the converse inequality to \eqref{eq:lisini4} holds.
\end{proof}

\section{Cheeger energy and minimal relaxed slope}\label{sec:Chee}

Throughout this section $(X,\tau,\sfd,\m)$ is an extended metric measure spaces according to Definition~\ref{def:extmm}. 
In this section we provide basic calculus results already developed in \cite{AGS11a}, 
with minor variants in the definitions that do not really affect the proofs. 

For $f\in {\rm Lip}_b(X,\tau,\sfd)$, the asymptotic Lipschitz constant ${\rm Lip}_a(f,x):X\to [0,\infty]$ is defined by
\begin{equation}\label{eq:deflipa}
{\rm Lip}_a(f,x)=\lim_{r\downarrow 0}{\rm Lip}_a(f,x,r)
\quad\text{with}\quad
{\rm Lip}_a(f,x,r):=\sup_{\sfd(y,x)\lor\sfd(z,x)<r,\,\sfd(y,z)>0}\frac{|f(y)-f(z)|}{\sfd(y,z)}\;,
\end{equation}
and with the usual convention ${\rm Lip}_a(f,x)=0$ at $\sfd$-isolated points $x$. By construction the function ${\rm Lip}_a(f,\cdot)$
is $\sfd$-upper semicontinuous. In the standard case when $\sfd$ is a finite distance and $\tau$ is the metric topology it follows
that ${\rm Lip}_a(f,\cdot)$ is also $\tau$-upper semicontinuous. 

\begin{definition}[Cheeger energy] For all $f\in L^2(X,\m)$ we set
$$
\Ch(f):=\inf\liminf_{n\to\infty}\int g_n^2\dd\m,
\qquad
D(\Ch):=\big\{f\in L^2(X,\m):\ \Ch(f)<\infty\big\}\ ,
$$
where the infimum runs among all sequences $(f_n)\subset {\rm Lip}_b(X,\tau,\sfd)$ with $\lim_n\int|f_n-f|^2\dd\m=0$ and all $\m$-measurable functions
$g_n\geq {\rm Lip}_a(f_n,\cdot)$ $\m$-a.e. in $X$. 
\end{definition} 
Motivated by the previous definition we may define, for $f\in {\rm Lip}(X,\tau,\sfd)$, 
${\rm Lip}_a^*(f,\cdot)$ as the (essential) least upper bound of all $\m$-measurable functions larger $\m$-a.e. than ${\rm Lip}_a(f,\cdot)$.
Then, $\Ch$ we can be equivalently defined by minimizing $\liminf_n\int ({\rm Lip}_a^*(f_n,\cdot))^2\dd\m$ among all
sequences $(f_n)\subset {\rm Lip}_b(X,\tau,\sfd)$ with $\lim_n\int|f_n-f|^2\dd\m=0$.
 
The concept of minimal relaxed slope is closely related to the definition of $\Ch$. First, one defines \emph{relaxed slope}
of $f$ any function $G\geq g$, with $g$ weak $L^2(X,\m)$ limit point as $n\to\infty$ of ${\rm Lip}_a^*(f_n,\cdot)$, where $f_n\in {\rm Lip}_b(X,\tau,\sfd)$ and $f_n\to f$ in
$L^2(X,\m)$. It can be proved (\cite[Lem.~4.3]{AGS11a}) that the class of relaxed slopes is a convex closed subset of $L^2(X,\m)$,
not empty if and only if $f\in D(\Ch)$. The minimal relaxed slope, denoted $|\rmD f|_w$ (and occasionally by
$|\rmD f|_{w,\sfd}$ to emphasize its dependence on $\sfd$), is the relaxed slope with
smallest $L^2(X,\m)$ norm. 

In analogy with the classical case, for all $f\in L^2(X,\m)$ with $\partial \Ch(f)\neq\emptyset$
we denote by $\Delta f$ the element with minimal $L^2(X,\m)$ norm in $\partial \frac 12 \Ch(f)$.

We now recall some basic calculus rules and more precise relations between $\Ch$ and the minimal relaxed slope.
Properties (g) and (h) below involve the notion of test plan, recalled below.

\begin{definition}[Test plan]
We say that $\eeta\in\cP(X^{[0,1]})$ is a $2$-test plan (relative to $\m$) if $\eeta$ is concentrated on $AC^2([0,1];(X,\sfd))$ and 
there exists $C\in [0,\infty)$ satisfying $(\e_t)_\#\eeta\leq C\m$ for all $t\in [0,1]$. The least constant $C$ with this property will be denoted by $C(\eeta)$. 
\end{definition}

\begin{proposition}\label{prop:calculus} The following properties hold:
\begin{itemize}
\item[(a)] For all $f,\,g\in D(\Ch)$, $\alpha,\beta\in \R$
  \begin{equation}
    \label{eq:14}
    |\rmD (\alpha f+\beta g)|_w\le |\alpha|\,|\rmD f|_w+|\beta|\,|\rmD g|_w\ ;
  \end{equation}
  in particular
  $\Ch$ and $(\Ch)^{1/2}$ are convex and lower semicontinuous functionals in $L^2(X,\m)$, with a dense domain.
\item[(b)] For all $f\in D(\Ch)$ one has $\Ch(f)=\int|\rmD f|_w^2\dd\m$ and there exist $f_n\in {\rm Lip}_b(X,\tau,\sfd)$ with
$f_n\to f$ in $L^2(X,\m)$ and ${\rm Lip}_a^*(f_n,\cdot)\to |\rmD f|_w$ in $L^2(X,\m)$.
\item[(c)] $|\rmD f|_w=|\rmD g|_w$ $\m$-a.e. in $\{f=g\}$ for all $f,\,g\in D(\Ch)$.
\item[(d)] $|\rmD f|_w\leq {\rm Lip}_a^*(f,\cdot)$ $\m$-a.e. in $X$ for all $f\in {\rm Lip}_b(X,\tau,\sfd)$.
\item[(e)] $|\rmD \phi(f)|_w=|\phi'(f)||\rmD f|_w$ $\m$-a.e. in $X$, for all $f\in D(\Ch)$ and $\phi:\R\to\R$ Lipschitz.
\item[(f)]  $\int  f\Delta g\dd\m\leq \int |\rmD f|_w|\rmD g|_w\dd\m$ for all $f\in D(\Ch)$, $g\in D(\Delta)$.
\item[(g)] If $\eeta\in\cP(X^{[0,1]})$ is a test plan, then for all $f\in D(\Ch)$ one has
$$
|f(\eta(1))-f(\eta(0))|\leq \int_0^1|\rmD f|_w(\eta(s))|\dot\eta(s)|\dd s
\quad\text{for $\eeta$-a.e. $\eta$\;.}
$$
\item[(h)] If $\eeta\in\cP(X^{[0,1]})$ is a test plan, then for all $f\in D(\Ch)$ one has
$$
\limsup\limits_{t\downarrow0}\int \frac{|f(\eta(t))-f(\eta(0))|^2}{(E_t(\eta))^2}\dd\eeta\leq\int|\rmD f|^2_w(\eta(0))\dd\eeta(\eta)\;,
$$
where $E_t(\eta):=\sqrt{t\int_0^t|\dot\eta(s)|^2\dd s}$.
\end{itemize}
\end{proposition}
\begin{proof} The properties from (a) to (e) are proved in
Lemma~4.3, Lemma~4.4, Theorem 4.5, and Proposition~4.8 of \cite{AGS11a}. 
The proof of (f) relies on
the convexity inequality $|\rmD (f+\eps g)|_w\leq |\rmD f|_w+\eps|\rmD g|$ with $\eps>0$, see
Proposition~4.15 of \cite{AGS11a}. 
Property (g) is proved in 
Theorem~5.14, Corollary~5.15 (see also Corollary~3.15) of \cite{AGS11a}, using Mazur's lemma, property (b) and
the upper gradient property of the asymptotic Lipschitz constant. Finally, (h) follows by (g) and H\"older's inequality,
which give
$$
\int \frac{|f(\eta(t))-f(\eta(0))|^2}{(E_t(\eta))^2}\dd\eeta\leq\frac 1t\int_0^t\int |\rmD f|_w^2\dd (\e_s)_\#\eeta \dd s\;.
$$ 
\end{proof}

\begin{corollary} \label{cor:2.9} If $(X,\sfd)$ is complete,
for all $g\in D(\Ch)$ and all $\mu_t=\rho_t\m\in AC^2([0,T];(\cP(X),W_\sfd))$ with
$\sup_t \|\rho_t\|_{L^\infty(X,\m)} <\infty$ one has
$$
\biggl|\int g \rho_T\dd\m-\int g \rho_0 \dd\m\biggr|\leq
\int_0^T\biggl(\int  |\rmD g|_w^2\rho_t\dd\m\biggr)^{1/2}|\dot\mu_t|\dd t\;.
$$
\end{corollary}
\begin{proof} 
By Proposition~\ref{prop:lisini} we can find $\eeta\in\cP(X^{[0,T]})$ concentrated on $AC^2([0,T];(X,\sfd))$ with
 $(e_t)_\#\eeta=\rho_t\m$ for all $t\in [0,T]$ and satisfying \eqref{eq:sharpmet},
so that Proposition~\ref{prop:calculus}\OOO (g)
\nc gives
$$
g(\eta(T))\leq g(\eta(0))+ \int_0^T|\rmD f|_w(\eta(t))|\dot\eta(t)|\dd t
\quad\text{for $\eeta$-a.e. $\eta$\;}
$$
By integrating this inequality and using H\"older's inequality with $t$ fixed together with \eqref{eq:sharpmet} the proof is
achieved.
\end{proof}

In the sequel we denote by
$$
{\mathsf F}(\rho):=4\Ch(\sqrt{\rho})\ ,
\qquad\rho\in L^1_+(X,\m)
$$
the so-called Fisher information functional. Let us recall its main properties (see \cite{AGS11a} for the simple proof).

\begin{proposition} $\mathsf F$ is a convex and $L^1$-lower semicontinuous functional in $L^1_+(X,\m)$. If
$\sqrt{\rho}\in D(\Ch)$ we have the equivalent representation
$$
{\mathsf F}(\rho)=\int_{\{\rho>0\}}\frac{|\rmD \rho|_w^2}{\rho}\dd\m\;.
$$
\end{proposition}

We denote by $\sfP_t$ the $L^2(X,\m)$ (metric) gradient flow of the convex and lower semicontinuous functional $\frac 12\Ch$. Since
$D(\Ch)$ includes ${\rm Lip}_b(X,\tau,\sfd)$ which is dense in $L^2(X,\m)$, $\sfP_t$ is a contraction semigroup in $L^2(X,\m)$, characterized by
\begin{equation}\label{eq:easter4}
\ddt \sfP_t f=\Delta \sfP_t f\qquad\text{for a.e. $t>0$}\;.
\end{equation}
Besides the general properties of gradient flows in Hilbert spaces listed in Proposition~\ref{prop:evik}, we recall that $\sfP_t$ satisfies the maximum and
minimum principle (i.e. if $c\leq f\leq C$ $\m$-a.e. in $X$, then $c\leq \sfP_t f\leq C$ $\m$-a.e. in $X$ for all
$t\geq 0$) and that
\begin{equation}\label{eq:contiCh}
  \biggl| \ddt \int f \sfP_t\rho\dd\m\biggr|\leq
{\mathsf F}^{1/2}(\sfP_t\rho)\biggl(\int |\rmD f|_w^2 \sfP_t\rho\dd\m\biggr)^{1/2}\qquad\text{for a.e. $t\in (0,\infty)$}\;.
\end{equation}
The maximum and minimum principle can be derived from Proposition~\ref{prop:calculus}(e), while
\eqref{eq:contiCh} is a direct consequence of Proposition~\ref{prop:calculus}(f) and of \eqref{eq:easter4}.

Note also the following fact (whose proof can be obtained
by a simple regularization argument, since $\m$ is finite, see
\cite{AGS11a}): for all $f\in L^2_+(X,\m)$ the function
$t\mapsto\int f_t\log f_t\dd\m$ is absolutely continuous in $[0,\infty)$ and
\begin{equation}\label{eq:easter5}
\ddt \int f_t\log f_t\dd\m=-\int_{\{f_t>0\}}\frac{|\rmD
  f_t|_w^2}{f_t}\dd\m\qquad\text{for a.e. $t>0$\ .}
\end{equation}
In particular, the right hand side is locally integrable in $[0,\infty)$. 

\section{Extended distances in $\cP^a(X)$}\label{sec:extendedcpa}

In this section \nc we introduce a class of absolutely continuous curves in an extended metric measure
space $(X,\tau,\sfd,\m)$, following the analogy with
\cite[Thm.~8.3.1]{Ambrosio-Gigli-Savare08}, \cite{AT14}, \cite{Gigli-Bangxian}. 

We first introduce a Banach structure on 
two Sobolev classes of test functions, $D(\Ch)$ and the algebra 
\begin{equation}
  \label{eq:24}
  \cA_{\Ch}:=\big\{f\in D(\Ch):\ f,\,|\rmD f|_w\in L^\infty(X,\m)\big\}\;
\end{equation}
 which obviously
includes ${\rm Lip}_b(X,\tau,\sfd)$ and it is 
dense in $L^2(X,\m)$; a simple truncation argument also shows that 
$\cA_{\Ch}$ is dense in $L^p(X,\m)$ for every $p\in [1,\infty)$. 

Since $f\mapsto \sqrt{\Ch(f)}$ is convex in $D(\Ch)$ 
and the function $(x,y)\mapsto \sqrt{x^2+y^2}$ is a norm in $\R^2$, it
is easy to check that  
\begin{equation}
  \|f\|_{\Ch}:=\Big(\|f\|_2^2+\Ch(f)\Big)^{1/2}=
  \Big(\int \big(|f|^2+|\rmD f|_w^2\big)\dd\m\Big)^{1/2}\;,
  \label{eq:55}
\end{equation}
is a norm in $D(\Ch)$; the lower semicontinuity of $\Ch$ with respect
to $L^2$-convergence also shows that
$\big(D(\Ch),\|\cdot\|_{D(\Ch)}\big)$ 
is a Banach space. Similarly, 
$\mathcal A_{\Ch}$ is a Banach algebra w.r.t.~the norm
\begin{equation}
  \label{eq:26}
  \big\|f\big\|_{\cA_{\Ch}}:=\big\|f\big\|_\infty+\big\|\,|\rmD f|_w\,\big\|_\infty\;.
\end{equation}
$D(\Ch)$ and the algebra $\cA_{\Ch}$ are not separable in general, 
but since their norms are lower semicontinuous w.r.t.~the $L^2$
convergence, they are $F_\sigma$ and thus Borel subsets of $L^2(X,\m)$.
When we will consider measurability of maps $\phi$ with values in
$D(\Ch)$ or $\cA_{\Ch}$,
we will always refer to its Borel $\sigma$-algebra inherited from 
the $L^2$ topology.
\subsection{The dynamic approach and the continuity inequality}
\nc
\begin{definition}[Continuity inequality]\label{def:conti}
Given a family of probability densities $\rho_t$, $t\in [0,T]$, we write
$\rho_t\in {\sf CE}^2(X,\Ch,\m)$ if there exists $c\in L^2(0,T)$ satisfying
\begin{equation}\label{eq:criterion_CE}
\biggl|\int f\rho_t\dd\m-\int f\rho_s\dd\m\biggr|\leq\int_s^tc(r)\biggl(\int |\rmD f|_w^2\rho_r\dd\m\biggr)^{1/2}\dd r
\end{equation}
for all $f\in\mathcal A_{\Ch}$ and all $0\leq s\leq t\leq T$. The
least $c$ in \eqref{eq:criterion_CE} is denoted $\|\rho_t'\|$.
\end{definition}
For simplicity of notation we do not emphasize the $T$ dependence in the previous definition. 

We will often deal with the case when $\rho_t$ are 
essentially bounded, uniformly w.r.t. time: in this case, by using the density of $\cA_{\Ch}$ in
$L^1(X,\m)$, it is then easy to check that $t\mapsto \rho_t$ is
weakly$^*$ continuous with values in $L^\infty(X,\m)$ (see also the next Theorem~\ref{thm:abs_char});
we will write $\rho\in C_{\rm w*}([0,T];L^\infty(X,\m))$. 

\begin{remark}\label{rem:heat}{\rm
It is a direct consequence of \eqref{eq:contiCh} and \eqref{eq:easter5} that $t\mapsto \rho_t:=\sfP_t\rho$ belongs to ${\sf CE}^2(X,\Ch,\m)$
for all $\rho\in L^2_+(X,\m)$, with
$$
\|\rho_t'\|^2\leq \int_{\{\rho_t>0\}}\frac{|\rmD\rho_t|^2_w}{\rho_t}\dd\m\qquad\text{for a.e. $t\in (0,\infty)$\;.}
$$
}\end{remark}

In the following theorem we show a ``differential'' characterization
of absolutely continuous curves in $(\cP(X),W_\sfd)$, which provides a key link between the metric and the differentiable viewpoints. 

\begin{theorem}[Differential characterization of absolutely continuous curves] \label{thm:abs_char}\quad\\
For all $\rho_t\in {\sf CE}^2(X,\Ch,\m)$ one has $\mu_t=\rho_t\m\in AC^2([0,T];(\cP(X),W_\sfd))$ and
\begin{equation}\label{eq:abs2}
|\dot\mu_t|\leq\|\rho_t'\|\qquad\text{for a.e. $t\in (0,T)$\;.}
\end{equation} 
Conversely, if $(X,\sfd)$ is complete, $\mu_t=\rho_t\m\in
AC^2([0,T];(\cP(X),W_\sfd))$ 
\OOO and $\sup_{t\in [0,T]}\|\rho_t\|_\infty <\infty$, \nc
then
$\rho_t\in{\sf CE}^2(X,\Ch,\m)$ and
\begin{equation}\label{eq:abs1}
\|\rho_t'\|\leq |\dot\mu_t|\qquad\text{for a.e. $t\in (0,T)$\;.}
\end{equation}
\end{theorem}
\begin{proof} We show the second part of the statement and \eqref{eq:abs1}.
Let $\mu_t=\rho_t\m\in AC^2([0,T];(\cP(X),W_\sfd))$ with 
essentially bounded densities $\rho_t$.

The inequality \eqref{eq:17} shows that $t\mapsto\int f\dd\mu_t$ is absolutely continuous in $[0,T]$ for all $f\in {\rm Lip}_b(X,\tau,\sfd)$. In addition, 
Corollary~\ref{cor:2.9} provides the inequality
$$
\biggl|\int f\rho_s\dd \m-\int f\rho_t\dd\m\biggr|\leq\int_s^t|\dot\mu_r|\biggl(\int |\rmD f|^2_w\rho_r\dd\m\biggr)^{1/2} \dd r 
$$
for $0\leq s\leq t\leq T$. By the density of ${\rm Lip}_b(X,\tau,\sfd)$ in $\mathcal A_{\Ch}$ provided by
Proposition~\ref{prop:calculus}(b) the inequality extends to all $f\in\mathcal A_{\Ch}$. 

We provide a proof of the converse implication and the converse inequality \eqref{eq:abs2}, along the lines of \cite{Kuwada10}, assuming
for simplicity $T=1$. First we notice that the property $\rho_t\in {\sf CE}(X,\Ch,\m)$ is stable under convolution w.r.t.~the time parameter: more precisely, if we extend
$t\mapsto\rho_t$ by continuity and with constant values to $(-\infty,0)\cup (1,\infty)$, then $\rho_{\eps,t}:=\rho_t\ast\chi_\eps$ still belongs
to ${\sf CE}(X,\Ch,\m)$ and $\|\rho_{\eps,t}'\|^2\leq \|\rho_t'\|^2\ast\chi_\eps$. For this reason, in the proof of this implication we can assume
with no loss of generality that $t\mapsto\rho_t$ is continuous w.r.t.~the $L^1(X,\m)$ topology.
We start from the duality formula \eqref{eq:dualityQ}. Let $\phi:X\to [0,\infty)$ be such that $\phi\in C(K)$ and
$\phi\equiv\max_K\phi$ on $X\setminus K$, with $K\subset X$ compact not empty. 
Under this restriction on $\phi$, we have already seen that $Q_\eps\phi$ can be represented in the form
\eqref{eq:good_tests1}, and that $Q_\eps\phi$ is $\sfd$-Lipschitz, Borel (because it is $\tau$-lower semicontinuous),
nonnegative and bounded. In addition $Q_\eps\phi\uparrow\phi$ and 
$Q_1(Q_\eps\phi)\uparrow Q_1\phi$ as $\eps\downarrow 0$. 

Set now $\varphi:=Q_\eps\phi$ for some $\eps>0$ and observe that $Q_t\varphi$, $t\in [0,1]$, are uniformly
$\sfd$-Lipschitz and that the map $t\mapsto Q_t\varphi$ is
Lipschitz from $[0,1]$ with values in $L^\infty(X,\m)$. By applying \cite[Lem.~4.3.4]{Ambrosio-Gigli-Savare08} to the function
$(s,t)\mapsto\int \rho_sQ_t\varphi\dd\m$ we obtain that $t\mapsto\int\rho_tQ_t\varphi$ is absolutely continuous
in $[0,1]$ and that its derivative can be estimated from above by
$$
\limsup_{s\to t}\frac{1}{|s-t|}{\biggl|\int(\rho_s-\rho_t) Q_t\varphi\dd\m\biggr|}+
\limsup_{s\to t}\frac{1}{|s-t|}{\biggl|\int\rho_t (Q_s\varphi-Q_t\varphi)\dd\m\biggr|}\;.
$$
Using the inequality
$$
\biggl|\int \rho_tQ_t\varphi\dd\m-\int \rho_sQ_t\varphi\dd\m\biggr|\leq\int_s^t\|\rho_r'\|\biggl(\int |\rmD Q_t\varphi|_w^2\rho_r\dd\m\biggr)^{1/2}\dd r
$$
we estimate the first limsup, at Lebesgue points $t$ of $s\mapsto|\rho_s'|$, by
$$
\|\rho_t'\|\biggl(\int |\rmD Q_t\varphi|_w^2\rho_t\dd\m\biggr)^{1/2}
$$
(here we used also the strong continuity of $s\mapsto\rho_s$). Estimating the second $\limsup$ with Fatou's lemma 
and using Proposition~\ref{prop:calculus}(d) gives
$$
\int (\rho_1Q_1\varphi-\rho_0\varphi)\dd\m\leq\int_0^1\biggl[
\|\rho_t'\|\biggl(\int ({\rm Lip}_a^*(Q_t\varphi,\cdot))^2\rho_t\dd\m\biggr)^{1/2}+
\int \rho_t\xi_t\dd\m\biggr]\dd t\;,
$$
where  $\xi_t$ is the bounded Borel function
$$
\xi_t:=\limsup_{s\to t}\frac {Q_s\varphi-Q_t\varphi}{s-t} \;.
$$
Now we use the pointwise subsolution property
\begin{equation}\label{eq:subsolutionQt}
\frac 12 \bigl({\rm Lip}^*_a(Q_t\varphi,\cdot)\bigr)^2\leq -\xi_t\qquad\text{$\m$-a.e. in $X$}
\end{equation}
for a.e. $t\in (0,1)$ (whose proof follows as in \cite[Thm.~3.5]{AGS11a}, where it is stated in a weaker form with the slope
in place of the asymptotic Lipschitz constant, see also \cite{ACDm}) 
and the Young inequality to get $\int (\rho_1Q_1\varphi-\rho_0\varphi)\dd\m\leq\frac 12\int_0^1\|\rho_t'\|^2\dd t$.
Remembering that $\varphi=Q_\eps\phi$, we can let $\eps\downarrow 0$ and use the arbitrariness of $\phi$ to get
$$
W_{\sfd}^2(\rho_1\m,\rho_0\m)\leq\int_0^1\|\rho_t'\|^2\dd t\;.
$$
By applying this inequality to a rescaled version of $\rho$ we obtain
$W_{\sfd}^2(\rho_t\m,\rho_s\m)\leq (s-t)\int_t^s\|\rho_r'\|^2\dd r$ for all $s,\,t\in [0,1]$ with $s>t$,  
so that by differentiation the inequality \eqref{eq:abs2} follows at a.e. $t\in (0,1)$.
\end{proof}

Using the continuity inequality we can define an extended ``Wasserstein-like'' distance 
$W_{\Ch}$ in $\cP^a(X)$ in the same spirit of the Benamou-Brenier
formula:
\begin{equation}\label{eq:matthias_distance}
W_{\Ch}^2(\rho_0\m,\rho_1\m):=\inf\bigg\{\int_0^1\|\rho_t'\|^2\dd t:\ \rho_t\in {\sf CE}^2(X,\Ch,\m)\biggr\}\;.
\end{equation}
This definition is also natural in view of Remark~\ref{rem:heat}. Even though it is conceptually convenient to think
to $W_{\Ch}$ as an extended distance in $\cP^a(X)$, we occasionally adopt we
the simpler notation $W_{\Ch}(\rho_0,\rho_1)$, i.e. we identify measures in $\cP^a(X)$ with their
densities w.r.t.~$\m$. The same remark applies to the other distances in $\cP^a(X)$ 
we shall introduce. Now we provide a few basic properties of $W_{\Ch}$.

\begin{proposition}[Properties of $W_{\Ch}$]\label{prop:propWCh}\quad
\begin{itemize}
\item[(a)] $(\cP^a(X),W_{\Ch})$ is an extended length metric space, and $W_{\Ch}\geq W_\sfd$.
\item[(b)] Assume that $\mu^n_t=\rho^n_t\m\in\cP(X)$ satisfy $\rho^n_t\to \rho_t$ weakly in $L^1(X,\m)$ for all $t\in [0,T]$ and 
that $\rho^n\in {\sf CE}^2(X,\Ch,\m)$ with $\|(\rho^n)_t'\|$ uniformly bounded in $L^2(0,T)$. Then, if $c(t)$
is any weak limit point of $\|(\rho^n)_t'\|$ as $n\to\infty$, one has 
$\rho_t\in {\sf CE}^2(X,\Ch,\m)$ with $\|\rho_t'\|\leq c(t)$ for a.e. $t\in (0,T)$.
\item[(c)] $W_{\Ch}^2$ is jointly convex in $(\cP^a(X))^2$.
\end{itemize}
\end{proposition}
\begin{proof} From \eqref{eq:abs2} with $T=1$ we obtain that $W_\sfd(\mu,\nu)\leq W_{\Ch}(\mu,\nu)$ whenever
$\mu,\,\nu\in\cP^a(X)$. This yields immediately that $W_{\Ch}(\mu,\nu)=0$ implies $\mu=\nu$.
The proof of the triangle property of $W_{\Ch}$ follows by a standard concatenation argument, noticing that for any $T>0$ one has
\begin{equation}\label{eq:feb8}
W_{\Ch}(\rho_0,\rho_T):=\inf\bigg\{\int_0^T\|\rho_t'\|\dd t:\ \rho_t\in {\sf CE}^2(X,\Ch,\m)\biggr\}\; .
\end{equation}
The length property also follows directly from \eqref{eq:feb8}, while the proof of (b) is a direct consequence of a passage to the limit as $n\to\infty$ in \eqref{eq:criterion_CE}.

In order to prove (c), notice that a convex combination
of \eqref{eq:criterion_CE} written for $\rho_t,\,\sigma_t\in {\sf CE}^2(X,\Ch,\m)$ gives
\begin{eqnarray*}
&&\biggl|\int f((1-\alpha)\rho_t+\alpha\sigma_t)\dd\m-\int f((1-\alpha)\rho_s+\alpha\sigma_s)\dd\m\biggr|\\&\leq&
\int_s^t(1-\alpha)\|\rho_r'\|\biggl(\int |\rmD f|_w^2\rho_r\dd\m\biggr)^{1/2}+
\alpha\|\sigma_r'\|\biggl(\int |\rmD f|_w^2\sigma_r\dd\m\biggr)^{1/2}\dd r\\
&\leq& \int_s^t \sqrt{ (1-\alpha)\|\rho_r'\|^2+\alpha\|\sigma_r'\|^2}
\biggl(\int |\rmD f|_w^2((1-\alpha)\rho_r+\alpha\sigma_r) \dd\m\biggr)^{1/2}\dd r
\end{eqnarray*}
for $0\leq s\leq t\leq T$. 
It follows that
$$
W^2_{\Ch}\bigl((1-\alpha)\rho_1+\alpha\sigma_1,(1-\alpha)\rho_0+\alpha\sigma_0\bigr)\leq
\int_0^1(1-\alpha)\|\rho_r'\|^2+\alpha\|\sigma_r'\|^2 \dd r
$$
and, by minimizing, we conclude.
\end{proof}

In the following corollary we reverse the inequality $W_\sfd\leq W_{\Ch}$ on probability measures with density
in $L^\infty(X,\m)$, at the level of absolutely continuous curves and metric derivatives. 

\begin{corollary}[Equality of metric derivatives]\label{cor:eqmet}
Assume that $(X,\sfd)$ is complete and let $(\rho_t)_{t\in [0,T]}$ 
be a curve of probability densities with $\sup_{t\in [0,T]}\|\rho_t\|_\infty <\infty$.
Then, for $\mu_t=\rho_t\m$, one has 
$$
\mu_t\in AC^2([0,T];(\cP(X),W_\sfd))\qquad\Longleftrightarrow\qquad
\mu_t\in AC^2([0,T];(\cP(X),W_{\Ch}))
$$
and the corresponding metric derivatives coincide a.e. in $(0,T)$.
\end{corollary}
\begin{proof} The implication $\Leftarrow$ is obvious, because $W_{\Ch}\geq W_\sfd$. In order to prove the converse
one, first apply the first part of the statement of Theorem~\ref{thm:abs_char} to obtain $\rho_t\in {\sf CE}^2(X,\Ch,\m)$ and
$\|\rho_t'\|\leq |\mu_t'|\in L^2(0,T)$. By the very definition of $W_{\Ch}$, this implies $\mu_t\in AC^2([0,T];(\cP(X),W_{\Ch}))$.
The coincidence of the metric derivatives is a simple consequence of
\eqref{eq:abs2}, \eqref{eq:abs1}.
\end{proof}

\subsection{A dual distance induced by subsolutions of the Hamilton-Jacobi equation}
\label{subsec:dualHJ}
We close this section by introducing another ``dual'' extended distance $W_{\Ch,*}$ in $\cP^a(X)$, motivated by the analogy with
the dual formulation of the optimal transport problem, the inequality $Q_1 f(x)-f(y)\leq \tfrac 12\sfd^2(x,y)$ and the subsolution
property \eqref{eq:subsolutionQt} of $Q_t f$.

\begin{definition} \label{def:Wdual} For $\mu_0=\rho_0\m,\,\mu_1=\rho_1\m\in\cP^a(X)$ we define
\begin{equation}\label{eq:defWdual}
W_{\Ch,*}^2(\rho_0,\rho_1):=2\sup_\phi\int (\phi_1\rho_1-\phi_0\rho_0)\dd\m\;,
\end{equation}
where the supremum runs in the convex set of all the 
bounded Borel maps $\phi(t,x)=\phi_t(x)$ satisfying
$\phi\in C_{\rm w*}([0,1];L^\infty(X,\m))\cap L^1(0,1;D(\Ch))$, 
and 
\begin{equation}\label{eq:HJdistr}
\ddt \phi_t+\frac 12 |\rmD \phi_t|^2_w\leq 0
\qquad\text{in $(0,1)\times X$, in the duality with $\mathcal A_{\Ch}$\;.}
\end{equation}
\end{definition}
The inequality \eqref{eq:HJdistr} has to be understood as
\begin{equation}
\ddt \int \phi_t\psi\dd\m+\frac 12 \int\psi |\rmD \phi_t|^2_w\dd\m\leq 0\qquad\text{in ${\mathcal D}'(0,1)$}\label{eq:56}
\end{equation}
for all $\psi\in\mathcal A_{\Ch}$ nonnegative.
\begin{lemma}[Equivalent admissible class of subsolutions to \eqref{eq:HJdistr}]
  \label{lem:equivalent-class}
  The supremum in formula \eqref{eq:defWdual} can be equivalently 
  taken w.r.t.~subsolutions $\phi$ to \eqref{eq:HJdistr} in the class 
  $\phi\in C^\infty([0,1];\cA_{\Ch})$.
\end{lemma}
\begin{proof}
By approximating any admissible $\phi$ in the definition of 
$W_{\Ch,*}$ with the functions
$$
\phi_\lambda(t,x):=\lambda\phi(\lambda t+(1-\lambda)/2,x)
$$
and by letting $\lambda\uparrow 1$, we see that is not restrictive to assume the existence of $a<0$ and $b>1$ such
that $\phi$ is bounded, $\phi\in C_{\rm w*}([a,b];L^\infty(X,\m))\cap L^1(a,b;D(\Ch))$ 
and $\partial_t\phi_t+|\rmD \phi_t|^2_w/2\leq 0$ in $(a,b)\times X$ according to \eqref{eq:56}. 
Then, by mollification w.r.t.~to $t$, which preserves the Hamilton-Jacobi subsolution
property,  we can also assume that 
$\phi\in C^\infty([0,1];L^\infty(X,\m))\cap C^\infty([0,1];D(\Ch))$ with
$\phi(\cdot,x)\in C^k([0,1])$, 
uniformly w.r.t.~$x$. Under this assumption, the subsolution property is satisfied
$\m$-a.e. in $X$, for all $t$, which also shows that the 
map
$t\mapsto |D\phi_t|_w$ is also uniformly bounded in $L^\infty(X,\m)$. It follows that $\phi$ is uniformly bounded
with values in $\cA_{\Ch}$ and strongly measurable with respect to the Borel sets induced by the $L^2$-topology. 
A further convolution in time (or the mollification by a semigroup in the first step) shows that we can 
also assume $\phi\in C^k([0,1];\cA_{\Ch})$.

\end{proof}
\begin{remark} [Elementary properties of $W_{\Ch,*}$] \label{rem:equivalent:HJ}
\ 

{\rm \
(1) By the scaling argument $\hat\phi(t,x)=\delta\phi(\delta t,x)$, it is easily seen that
$$
W_{\Ch,*}^2(\rho_0,\rho_1)=
2\sup_{(\delta,\phi)}\ \delta\int (\phi_\delta\rho_1-\phi_0\rho_0)\dd\m\;,
$$
where the supremum runs among all pairs $(\delta,\phi)$ with $\delta>0$ and $\phi$ bounded Borel map 
$\phi(t,x)=\phi_t(x)$ satisfying $\phi\in C_{\rm w*}([0,\delta];L^\infty(X,\m))\cap
L^1(0,\delta;D(\Ch))$, and 
$$
\ddt \phi_t+\frac 12 |\rmD \phi_t|^2_w\leq 0
\qquad\text{in $(0,\delta)\times X$, in the duality with $\mathcal A_{\Ch}$\;.}
$$
\ \ (2)
More generally, suppose that $\varphi\in C_{\rm w*}([a,b];L^\infty(X,\m))\cap L^1(a,b;D(\Ch))$ satisfies
\begin{equation}
  \label{eq:69}
  \ddt \varphi_t+\frac {\vartheta(t)}2 |\rmD \varphi_t|^2_w\leq 0
  \qquad\text{in $(a,b)\times X$, in the duality with $\mathcal
    A_{\Ch}$\;}
\end{equation}
where $\vartheta\in C([a,b])$ is a positive function. Then
\begin{equation}
  \label{eq:70}
  2\alpha(b) \int\big(\rho_1\varphi_b-\rho_0\varphi_a\big)\dd\m\le
  W_{\Ch,*}^2(\rho_0,\rho_1)\quad\text{where}\quad
  \alpha(t):=\int_a^t \vartheta(r) \dd r \;
\end{equation}
In fact, setting 
\begin{displaymath}
  \beta(t):=\alpha(t)/\alpha(b),\ t\in [a,b],\qquad
  \gamma(s):=\beta^{-1}(s)\;
\end{displaymath}
so that $\gamma$ is an increasing diffeomorphism between $[0,1]$ and $[a,b]$
satisfying $\gamma'(s)=\alpha(b)/\vartheta({\gamma(s)})$,
the curve 
$\tilde\varphi_s:=\alpha(b)\varphi_{\gamma(s)}$ solves
\begin{displaymath}
  \dds \tilde\varphi_s+\frac 12 |\rmD \tilde\varphi_s|^2_w\leq 0
  \qquad\text{in $(0,1)\times X$, in the duality with $\mathcal
    A_{\Ch}$\;}
\end{displaymath}
so that 
\begin{displaymath}
  2\alpha (b)\int\big(\rho_1\varphi_b-\rho_0\varphi_a\big)\dd\m=
  2\int\big(\rho_1\tilde\varphi_1-\rho_0\tilde\varphi_0\big)\dd\m
  \le 
  W_{\cE,*}^2(\rho_0,\rho_1)\;.
\end{displaymath}
\ \ (3)
It is not hard to prove that \emph{$W_{\Ch,*}$ is an extended distance:} indeed, the non-degeneracy condition follows by
the inequality $W_{\Ch,*}\geq W_\sfd$, proved in the next proposition. The symmetry property follows
easily by replacing $\phi(t,x)$ by $-\phi(\delta-t,x)$. In order to prove the triangle inequality, given probability
densities $\rho$, $\sigma$, $\lambda$, and
constants
\nc $\delta>0$ and $\delta'\in (0,\delta)$ we write
\begin{eqnarray*}
2\delta \int (\phi_\delta\lambda-\phi_0\rho)\dd\m&=&
2\delta \int (\phi_\delta\lambda-\phi_{\delta'}\sigma)\dd\m+2\delta \int (\phi_{\delta'}\sigma-\phi_0\rho)\dd\m\\
&\leq&\frac{2\delta}{2(\delta-\delta')}W_{\Ch,*}^2(\lambda,\sigma)+\frac{2\delta}{2\delta'}W_{\Ch,*}^2(\sigma,\rho)\;.
\end{eqnarray*}
Now we minimize w.r.t.~$\delta'$ and use the identity $\inf_{(0,1)}s^{-1}a^2+(1-s)^{-1}b^2=(a+b)^2$ to get
$$
2\delta \int \phi_\delta\lambda-\phi_0\rho\dd\m\leq \bigl(W_{\Ch,*}(\lambda,\sigma)+W_{\Ch,*}(\sigma,\rho)\bigr)^2\;.
$$
By taking the supremum w.r.t.~$(\delta,\phi)$ we conclude.

(4) 
\emph{$W^2_{\Ch,*}$ is jointly convex in $\cP^a(X)\times \cP^a(X)$ and
l.s.c.~with respect to the weak topology of $L^1(X,\m)$},
since it is defined as a supremum of a family of continuous linear
functionals on $L^1(X,\m)$. In particular, every closed sublevel of
the Entropy functional \eqref{eq:28} in $\cP^a(X)$ is complete
with respect to $W_{\Ch,*}$.
}\end{remark}

We can now refine the inequality between $W_{\Ch}$ and $W_\sfd$, proving that $W_{\Ch,*}$ is intermediate.

\begin{proposition}[Comparison of $W_{\Ch}$, $W_{\Ch,*}$ and $W_\sfd$]\label{prop:comparale}
$W_\sfd\leq W_{\Ch,*}\leq W_{\Ch}$ on $(\cP^a(X))^2$.
\end{proposition}
\begin{proof} We first prove that $W_{\Ch}\geq W_{\Ch,*}$.
By Lemma \ref{lem:equivalent-class} we can assume 
that $\phi\in C^1([0,1];\cA_{\Ch})$ with 
$\phi(\cdot,x)\in C^1([0,1])$, uniformly w.r.t.~$x$. Under this assumption, the subsolution property is satisfied
$\m$-a.e. in $X$, for all $t$; in addition, for
all $\rho_t\in {\sf CE}^2(X,\Ch,\m)$, the Leibniz rule and a density argument easily give that 
$t\mapsto\int\phi_t\rho_t\dd\m$ is absolutely continuous in $[0,1]$, and that
$$
\ddt \int \phi_t\rho_t\dd\m=\int \phi_t\ddt \rho_t\dd\m+\int\rho_t\ddt \phi_t\dd\m\leq
\int \phi_t\ddt \rho_t\dd\m-\frac 1 2\int\rho_t |\rmD \phi_t|_w^2\dd\m
$$ 
for a.e. $t\in (0,1)$. 
By the Young inequality, it follows that
$$
\biggl|\ddt \int \phi_t\rho_t\dd\m\biggr|\leq \frac{1}{2}\|\rho_t'\|^2
\qquad\text{for a.e. $t\in (0,1)$\;.}
$$
By integration in $(0,1)$ and by minimizing w.r.t.~$\rho_t$ the inequality follows.

Now we prove that $W_{\Ch,*}\geq W_\sfd$. Let $(\sfd_i)$ be an increasing net of bounded and $(\tau\times\tau)$-continuous
semidistances with $\sfd_i\uparrow\sfd$. Taking Theorem~\ref{thm:compact_joint} into account, it is sufficient to fix $i$ and prove that $W_{\sfd_i}\leq W_{\Ch,*}$.
In order to prove this property, taking \eqref{eq:HopfLax} and the comments immediately after into account, it suffices to show that 
$$Q^i_t\phi(x):=\inf_{y\in X}\phi(y)+\frac 1{2t}\sfd_i^2(x,y)$$
is admissible in \eqref{eq:defWdual} whenever $\phi$ is bounded
and $\sfd_i$-Lipschitz \OOO (thus $\tau$-continuous). 
\nc 
This follows combining 
the subsolution property (see \eqref{eq:subsolutionQt} and the comments after it) 
$$
\limsup_{s\to t}\frac{Q^i_s\phi(x)-Q^i_t\phi(x)}{s-t}+\frac 12\bigl({\rm Lip}^*_{a,\sfd_i}(Q^i_t\phi,x)\bigr)^2\leq 0
\quad\text{$\m$-a.e. in $X$}
$$
satisfied by $Q^i_t\phi$ for a.e. $t>0$ with the inequalities
$$
|\rmD Q^i_t \phi|_w(x)\leq {\rm Lip}_{a,\sfd}^*(Q^i_t\phi,x)\leq
{\rm Lip}_{a,\sfd_i}^*(Q^i_t\phi,x)\qquad\text{$\m$-a.e. in $X$}\;.
$$
\end{proof}
\nc
\begin{remark}\label{rem:anche_questa_bis}{\rm 
 One can also introduce the ``dual'' $L^1$ transport distance $W_{\Ch,*,1}$:
 \begin{equation}\label{eq:anche_questa_bis}
 W_{\Ch,*,1}(\rho_0,\rho_1) := \sup_\phi \int\phi(\rho_1 -\rho_0)\dd\m\;,
 \end{equation}
 where the supremum runs over all bounded and Borel maps $\phi\in D(\Ch)$ with $|\rmD \phi|_w\leq 1$
 $\m$-a.e. in $X$.
 It is not hard to see that 
 $$W_{\Ch,*,1}(\rho_0,\rho_1)\leq W_{\Ch,*}(\rho_0,\rho_1)\;.$$
 Indeed, fix $\phi$ with $|\rmD \phi|_w\leq1$ and put $\phi_t(x) =-\frac12 t + \phi(x)$, which is admissible in the
 definition of $W_{\Ch,*}$. Now for $\delta>0$ we have
 $$
   \int \phi(\rho_1-\rho_0)\dd\m = \int (\phi_\delta\rho_1-\phi_0\rho_0)\dd\m + \frac{\delta}{2} 
                                 \leq \frac{1}{2\delta} W_{\Ch,*}^2(\rho_0,\rho_1) +\frac{\delta}{2}\;.
 $$
 Optimizing in $\delta$ we find $\int \phi(\rho_1-\rho_0)\dd\m\leq W_{\Ch,*}(\rho_0,\rho_1)$ and taking the supremum over $\phi$
yields the claim.}
\end{remark}
 
\section{Identification of gradient flows}\label{sec:identiflows}

In this section we compare the metric gradient flows of $\ent$ w.r.t.~to the extended distances $W_{\sfd}$ and $W_{\Ch}$, relating also
them to the $L^2(X,\m)$ gradient flow $\sfP_t$ of $\frac 12\Ch$. 

The following result is a small improvement of \cite[Thm.~7.4]{AGS11a}, since we replace the slope of $\ent$ w.r.t.~$W_\sfd$ with the
slope w.r.t.~the (a priori larger) distance $W_{\Ch}$. It can be
obtained with the same proof.

\begin{lemma}[The Fisher information is dominated by the slope of the Entropy]\label{lem:fisherboundsslope}
For every probability density $\rho\in L^2_+(X,\m)$ one has 
\begin{equation}\label{eq:44pre}
4\int |\rmD \sqrt{\rho}|_w^2\dd\m\leq |\rmD^-_{W_{\Ch}}\ent|^2(\rho\,\m)\;.
\end{equation}
\end{lemma}
\begin{proof} Let $\rho_t=\sfP_t\rho$; we set $\mu_t:=\rho_t\,\m$ and $\mu=\rho\,\m$. 
Denoting by $|\dot\mu_t|$ the metric derivative w.r.t. $W_{\Ch}$, 
from Remark~\ref{rem:heat} we get  
\begin{equation}\label{eq:jan1}
|\dot\mu_t|^2\leq\mathsf F(\rho_t)\qquad\text{for a.e. $t>0$}\;.
\end{equation}
  Applying \eqref{eq:easter5} we get
  \begin{align}
    \label{eq:38}
    \ent(\mu)&-\ent(\mu_t)=\int_0^t \mathsf F(\rho_s)\dd s\ge \frac 12 \int_0^t \mathsf F(\rho_s)\dd s+\frac 12
    \int_0^t |\dot\mu_s|^2\dd s
    \\
    \notag&\ge \frac 12 \Big(\frac 1{\sqrt t}\int_0^t
    \sqrt{\mathsf F(\rho_s)}\dd s\Big)^2+\frac 12\Big(\frac 1{\sqrt t}\int_0^t |\dot\mu_s|\dd s\Big)^2
    \ge \frac 1t\Big(\int_0^t \sqrt{\mathsf F(\rho_s)}\dd s\Big)W_{\Ch}(\mu,\mu_t)\;.
  \end{align}
  Dividing by $W_{\Ch}(\mu,\mu_t)$ and passing to the limit as
  $t\downarrow0$ we get
  \eqref{eq:44pre}, since the lower semicontinuity of $\Ch$ yields
  \begin{displaymath}
    \sqrt{\mathsf F(\rho)}\le\liminf_{t\downarrow0}\frac 1t\int_0^t
    \sqrt{\mathsf F(\rho_s)}\dd s\; .
  \end{displaymath}
\end{proof} 

In order to identify the metric gradient flows of $\ent$ with $\sfP_t$, we will also use the following result, see 
\cite[Lem.~5.17, Thm.~8.1]{AGS11a}.
Its proof uses Proposition~\ref{prop:lisini}, Proposition~\ref{prop:calculus} (e), (g), the estimates \eqref{eq:contiCh}, \eqref{eq:easter5}, 
the convexity of $\mathsf F$ and the strict convexity of $\ent$, see also the next section for an analogous argument
involving the same ingredients.

\begin{theorem}
  \label{thm:coincidence_infty}
  Let $(\rho_t)_{t\in [0,T]}$ be a curve of bounded probability densities
  with $\sup_t\|\rho_t\|_{\infty}<\infty$. Assume
  that $\mu_t=\rho_t\m\in AC^2([0,T];(\cP(X),W_\sfd))$ and that
  $\mu_t$ satisfies the Entropy-Fisher dissipation inequality
  \begin{equation}
    \label{eq:95}
    \ent(\mu_0)\ge \ent(\mu_T)+\frac12\int_0^T|\dot{\mu}_t|^2\dd t
    +\frac12\int_0^T \mathsf F(\rho_t)\dd t\;.
  \end{equation}
  Then $\rho_t=\sfP_t\rho_0$ for all $t\in [0,T]$ and equality holds in \eqref{eq:95}.
  \end{theorem}

\begin{theorem}[Identification of gradient flows]\label{thm:mainidenti}
Let $(X,\tau,\sfd,\m)$ be an extended metric measure space with $(X,\sfd)$ complete.
Let $(\rho_t)_{t\in [0,\infty)}$ be a curve of probability densities with $\|\rho_t\|_\infty\in L^\infty_{\rm
  loc}([0,\infty))$ and set $\mu_t=\rho_t\m$ and let us consider the properties
\begin{itemize}
\item[(a)] $\mu_t$ is a metric gradient curve of $\ent$ relative to
  $W_{\sfd}$ starting from $\mu_0$;
\item[(b)] $\mu_t$ is a metric gradient curve of $\ent$ relative to
  $W_{\Ch}$ starting from $\mu_0$;
\item[(c)] $\rho_t=\sfP_t\rho_0$ for all $t\in [0,\infty)$.
\end{itemize} 
Then $(a)\Rightarrow (b)\Rightarrow (c)$. If $|\rmD_\sfd^-\ent|$ is lower semicontinuous in $L_+^1(X,\m)$, then $(c)\Rightarrow (a)$.
\end{theorem}
\begin{proof} By Corollary~\ref{cor:eqmet} and the inequality
  $W_{\sfd}\leq W_{\Ch}$, which yields a converse inequality at the
  level of slopes, the metric gradient curves in (a) are contained in
  the metric gradient curves in (b). On the other hand, by
  \eqref{eq:44pre} of Lemma~\ref{lem:fisherboundsslope}, the metric
  gradient curves in (b) satisfy the Entropy-Fisher dissipation
  inequality and therefore, thanks to
  Theorem~\ref{thm:coincidence_infty}, satisfy (c). Finally, 
    under the lower semicontinuity assumption on $|\rmD_\sfd^-\ent|$,
  the identity
$$
 \ent(\rho\m)-\ent(\sfP_t\rho \m)=\int_0^t \mathsf F(P_s\rho)\dd s
$$
and the inequality {$\mathsf F(\rho)\geq |\rmD^-_\sfd\ent|^2(\rho\,\m)$
(see \cite[Thm.~7.6]{AGS11a})} show that the class (c) is contained in the class (a).
\end{proof}

\begin{remark} {\rm By comparison, the implications above can also be stated with the
distance $W_{\Ch,*}$. This is possible because, according to Proposition~\ref{prop:comparale}, 
$W_{\Ch,*}$ is intermediate between $W_{\Ch}$ and $W_{\sfd}$.}
\end{remark}

\section{A stability result for Cheeger's energies}\label{sec:stabchee}

In this section we consider an extended metric-topological space $(X,\tau,\sfd)$ and a monotone family of 
$(\tau\times\tau)$-continuous approximating semidistances $\sfd_i:X\times X\to [0,\infty)$ as in Definition~\ref{def:luft1}. Given $\m\in\cP(X)$,
our goal is to prove a convergence results for the corresponding Cheeger energies. Since in view of the applications 
we have in mind we want to cover also the case when $\sfd_i$ are semidistances, we have
to adapt the construction of Section~\ref{sec:Chee}, thought for (extended) distances, to this slightly more general setting. 

Let $(X_i,\tilde\sfd_i)$ be the quotient metric space, with $\pi^i:X\to X_i$ the canonical projection. We choose in $X_i$ the standard
topology $\tau_i$ generated by the metric structure, so that ${\rm Lip}(X_i,\tau_i,\tilde\sfd_i)$ is a standard metric-topological space 
and $\pi^i:X\to X_i$ is continuous (thanks to the $(\tau\times\tau)$-continuity of $\sfd_i$). 
Thanks to the continuity of $\pi^i$ we can also define $\m_i =(\pi^i)_\#\m\in\cP(X_i)$, thus providing the structure
of  metric measure space to $X_i$. 

The map $g\mapsto\pi^i_*(g)=g\circ\pi^i$ provides a linear 
isometry of $L^2(X_i,\m_i)$ into $L^2(X,\m)$. Then, denoting by 
$\mathcal D_i$ the closure of ${\rm Lip}_b(X,\tau,\sfd_i)$ in $L^2(X,\m)$, we notice that
$\mathcal D_i\subset\pi^i_*(L^2(X_i,\m_i))$, because any function 
in ${\rm Lip}_b(X,\tau,\sfd_i)$ belongs to the image of $\pi^i_*$.

Denoting by $\tilde{\Ch}_i$ and $\tilde{\sfP}^i_t$ 
the Cheeger energy and its gradient flow in
$(X_i,\tau_i,\tilde\sfd_i,\m_i)$, the formulas 
\begin{equation}\label{def:Ch_i}
  \begin{aligned}
    \Ch_i(f):={}&
    \begin{cases}
      \tilde{\Ch}_i(g)&\text{if }f=g\circ\pi^i\in
      \pi^i_*(L^2(X_i,\m_i))\\
      +\infty&\text{otherwise}
    \end{cases}
    \\
    \sfP^i_tf:={}&(\tilde \sfP_t^ig)\circ\pi^i\qquad\quad\,\, \text{if }f=g\circ\pi^i\in
    \pi^i_*(L^2(X_i,\m_i)),\quad t\geq 0\;,
  \end{aligned}
\end{equation}
enable to lift the Cheeger energy $\tilde{\Ch}_i$ and its gradient
flow $\tilde{\sfP}^i$ to the subspaces $\pi^i_*(L^2(X_i,\m_i))$ of $L^2(X,\m)$, retaining
the metric gradient flow property. Since $\tilde{\Ch}_i$ have a dense domain in $L^2(X,\m_i)$ it follows that the closure of the domain of $\Ch_i$, namely
$\pi^i_*(L^2(X_i,\m_i))$, contains $\mathcal D_i$, so that Lemma~\ref{lem:sigmaalgebras} gives
\begin{equation}\label{eq:unione_densa}
\text{$\bigcup_{i\in I} D(\Ch_i)$ is dense in $L^2(X,\m)$\;.}
\end{equation}

The proof of the following theorem is inspired by various stability results based on $\Gamma$-convergence and on the energy
dissipation point of view, see for instance \cite{Serfaty},
\cite{Gigli10} and \cite{AGMS12}. At the level of $\Ch_i$, the only properties that
will play a role are \eqref{eq:unione_densa} and the energy dissipation inequality \eqref{eq:dissipando}. The latter easily follows from the
corresponding properties of $\tilde{\Ch}_i$, $\tilde\sfP^i$.

\begin{theorem} \label{thm:stabchee} Under the previous assumptions on $\sfd_i$ one has that
$\Ch$ coincides with the largest $L^2(X,\m)$ lower semicontinuous functional smaller than $\inf_i\Ch_i$.
\end{theorem}
\begin{proof} 
Let $L_*$ be the largest $L^2(X,\m)$ lower semicontinuous functional smaller than $L:=\inf_i\Ch_i$.
Since $\sfd_i\leq\sfd$, from the inequality
\begin{equation}\label{eq:compare_d_di}
\int |\rmD (g\circ\pi) |^2\dd\m\leq\int |\rmD_{\sfd_i} (g\circ\pi)|^2\dd\m=\int |\rmD_{\tilde\sfd_i}g|^2\dd\m_i
\end{equation}
we immediately get  $\Ch\leq\Ch_i$, hence $\Ch\leq L$ and the lower semicontinuity of $\Ch$ gives
$\Ch\leq L_*$. In order to prove the converse inequality, we fix a probability density $\rho$ with $C\geq\rho\geq c>0$ $\m$-a.e. in $X$ and denote by
$\rho^i_t$ the gradient curves of $\Ch_i$ starting from $\rho^i_0$, the $L^2(X,\m)$ projection of $\rho$ on $\overline{D(\Ch_i)}$. By
\eqref{eq:unione_densa} we know that $\rho^i_0\to\rho$ in $L^2(X,\m)$ and the stability of $\overline{D(\Ch_i)}$ under truncations
immediately gives $C\geq\rho^i_t\geq c$ $\m$-a.e. in $X$; in addition, using the regularization estimate 
\eqref{eq:regulaK} it is easily seen that 
$$
\limsup_{i\in I} {\rm Lip}(\rho^i_\cdot,(\delta,\infty))<\infty\qquad\forall\delta>0
$$
(where the Lipschitz constant is computed w.r.t.~the $L^2(X,\m)$ norm).
Hence, we may find a subnet $\beta:J\to I$ and a curve $\rho_t$ such that
$\lim_{j\in J}\rho^{\beta(j)}_t=\rho_t$ in the weak topology of $L^2(X,\m)$ for all $t\geq 0$, with $\rho:(0,\infty)\to L^2(X,\m)$ continuous. 

Our goal is to pass to the limit first w.r.t.~$j$ and then as $t\downarrow 0$ in the energy dissipation inequalities 
\begin{equation}\label{eq:dissipando}
\ent(\mu^{\beta(j)}_t)+\frac 12\int_0^t|\dot\mu^{\beta(j)}_s|_{\beta(j)}^2+{\mathsf F}^{\beta(j)}(\rho^{\beta(j)}_s)\dd s
\leq\ent(\rho^{\beta(j)}_0\m)\;,
\end{equation}
with $\mu^{\beta(j)}_t=\rho^{\beta(j)}_t\m$, $|\dot\mu^{\beta(j)}_t|_{\beta(j)}$ equal to the metric derivative of the curve $\mu^{\beta(j)}_t$ w.r.t.~$W_{\sfd_{\beta(j)}}$ and
${\mathsf F}^{\beta(j)}$ the Fisher information functionals associated to $\Ch_{\beta(j)}$, 
to prove that $\rho_t$ is the gradient curve of $\Ch$ starting from $\rho$.  

We first notice that the representation \eqref{eq:19} 
of the action as a supremum, together with the monotone convergence
$\lim_jW_{\sfd_{\beta(j)}}=W_{\sfd}$ imply that
$\mu_s\in AC^2([0,t];(\cP(X),W_\sfd))$ and that
\begin{equation}\label{eq:lsc_metrder}
\liminf_{j\in J}\int_0^t|\dot\mu^{\beta(j)}_s|_{\beta(j)}^2\dd s\geq \int_0^t|\dot\mu_s|^2\dd s\;,
\end{equation}
where $\mu_t=\rho_t\m$ and $|\dot\mu_t|$ denotes the metric derivative w.r.t.~$W_{\sfd}$.

Let us denote by ${\mathsf F}$ the Fisher information functional associated to $\Ch$ and notice that
${\mathsf F}_i\geq{\mathsf F}$. Hence, combining \eqref{eq:lsc_metrder} with \eqref{eq:dissipando}
and with $\liminf_{j\in J}\ent(\mu^{\beta(j)}_t)\geq\ent(\mu_t)$ we get
\begin{equation}\label{eq:dissipandolimitbis}
\ent(\mu_t)+\frac 12\int_0^t|\dot\mu_s|^2+{\mathsf F}(\rho_s)\dd s
\leq\ent(\rho\m)\;.
\end{equation}
Since $t$ is arbitrary this inequality, according to Theorem~\ref{thm:coincidence_infty}, proves that $\rho_t=\sfP_t\rho$, where $\sfP_t$ is the 
$L^2(X,\m)$-gradient flow of $\Ch$; in addition, we can still use Theorem~\ref{thm:coincidence_infty} to obtain that equality holds in \eqref{eq:dissipandolimitbis}. Therefore
we obtain from this limiting argument the additional informations
\begin{equation}\label{eq:additional_informations}
  \liminf_{j\in J}\ent(\rho^{\beta(j)}_t\m)=\ent(\rho_t\m)\;,\qquad
  \liminf_{j\in J}\int_0^t|\dot\mu^{\beta(j)}_s|_{\beta(j)}^2\dd s=\int_0^t|\dot\mu_s|^2\dd s\;,
\end{equation}
(that we shall exploit in the next theorem) as well as
\begin{equation}\label{eq:extra_fisher1}
  \liminf_{j\in J}\int_0^t{\mathsf F}^{\beta(j)}(\rho^{\beta(j)}_s)\dd s=\int_0^t {\mathsf F}(\rho_s)\dd s\;.
\end{equation}
If we assume 
\begin{equation}\label{eq:extra_fisher}
\limsup_{t\downarrow 0} \frac 1t\int_0^t {\mathsf F}(\rho_s)\dd s\leq {\mathsf F}(\rho)
\end{equation}
we can find, thanks to the convexity of ${\mathsf F}^{\beta(j)}$,  $t(j)\to 0$ such that the functions
$$
v_j:=\frac{1}{t(j)}\int_0^{t(j)}\rho^{\beta(j)}_s\dd s
$$
satisfy $\liminf_j{\mathsf F}^{\beta(j)}(v_j)\leq {\mathsf F}(\rho)$, so that
\begin{equation}\label{eq:extra_fisher2}
\liminf_{j\in J} \Ch_{\beta(j)}(\sqrt{v_j})\dd s\leq \Ch(\sqrt{\rho})\;.
\end{equation}
 
In order to prove that this implies $L_*(\sqrt{\rho})\leq\Ch(\sqrt{\rho})$ it is sufficient to show that $\sqrt{v_j}\to
\sqrt{\rho}$ in $L^2(X,\m)$. This can be proved as follows: since $W_{\sfd_{\beta(j)}}(v_j,\rho)\to 0$, we obtain
$$
\lim_{j\in J}\int fv_j\dd\m=\int f\rho\dd\m\qquad\text{for all $f\in \bigcup_{i\in I}{\rm Lip}(X,\tau,\sfd_i)$}\;.
$$
Hence, by $w^*$-compactness of closed balls in $L^\infty(X,\m)$ and density of $\cup_i {\rm Lip}(X,\tau,\sfd_i)$, $v_j\to\rho$ weakly$^*$ in $L^\infty(X,\m)$.
Then, the entropy bound $\ent(v_j\m)\leq\ent(\rho\m)$ implies convergence of $v_j$ in $\m$-measure.

Now we remove the assumption \eqref{eq:extra_fisher}. Given a probability density $\bar\rho\in D(\Ch)$ with $C\geq\bar\rho\geq c>0$ $\m$-a.e. in $X$ 
we obtain by the previous step applied to $\rho=\sfP_t\bar \rho$,
the inequality $\Ch(\sqrt{\sfP_t\bar\rho})\geq L_*(\sqrt{\sfP_t\bar\rho})$ for a.e. $t>0$. By the chain rule, since
$\Ch(\sfP_t\bar\rho)\to \Ch(\bar\rho)$ as $t\downarrow 0$ implies $|\rmD \sfP_t\bar\rho|_w\to |\rmD \bar\rho|_w$ in $L^2(X,\m)$, we obtain
$\Ch(\sqrt{\sfP_t\bar\rho})\to\Ch(\sqrt{\bar\rho})$ as $t\downarrow 0$ and therefore $\Ch(\sqrt{\bar\rho})\geq L_*(\sqrt{\bar\rho})$.

This proves the inequality $\Ch\geq L_*$ on all bounded and normalized functions $\rho$ with $\inf\rho>0$.
Finally, we can extend by standard approximation arguments the inequality first to all bounded functions
(by homogeneity and translation invariance) and eventually to all functions in $L^2(X,\m)$.
\end{proof} 

A byproduct of the previous proof and of the identification of gradient flows is the following stability result of gradient flows
of Cheeger's energies; the stability proof provides also a crucial regularity property of Cheeger's energies that we call, as in 
\cite{AGS12}, $\tau$-upper regularity (see also Definition~\ref{def:upper_regularity_cE}). 
We use the same notation of the statement of Theorem~\ref{thm:stabchee} and the notation
$$
|\rmD (g\circ\pi^i)|_{w,\sfd_i}:=|\rmD g|_w\circ\pi^i\qquad g\in D(\tilde{\Ch}_i)\;,
$$
consistent with the definition \eqref{def:Ch_i} of $\Ch_j$. It is not difficult to show, along the lines of \eqref{eq:compare_d_di}, that
$|\rmD f|_{w,\sfd_i}\geq |\rmD f|_{w,\sfd}$ $\m$-a.e. in $X$ for all $f\in D(\Ch_i)$.

\begin{theorem}[Stability of gradient flows and $\tau$-upper regularity of Cheeger energies]\label{thm:stability}\quad
Under the same assumptions of Theorem~\ref{thm:stabchee},
let $\rho_0\in L^\infty_+(X,\m)$ and let $\rho^i_t$ (resp. $\rho_t$) be the $L^2$ gradient curves of $\Ch_i$ (resp.
$\Ch$) starting from $\rho^i_0$, the $L^2(X,\m)$ projection of $\rho_0$ on $\overline{D(\Ch_i)}$. 
Then  $\rho^i_t\to\rho_t$ strongly in $L^2(X,\m)$ for all $t\geq 0$.\\ In addition, for
all $f\in D(\Ch)$ there exist a subnet $\beta:J\to I$, bounded and $\sfd_{\beta(j)}$-Lipschitz functions $f_j$ with $f_j\to f$ in $L^2(X,\m)$ and
${\rm Lip}_a(f_j,\sfd_{\beta(j)},\cdot)\to |\rmD f|_w$ in $L^2(X,\m)$.
\end{theorem}
\begin{proof} The weak convergence of $\rho^i_t$ to $\rho_t$ in $L^\infty(X,\m)$ has already been achieved in
the proof of the previous theorem. To show that the convergence is actually strong, we use the first equality in
\eqref{eq:additional_informations}, which can be improved to
$$
\lim_{i\in I}\ent(\rho^i_t\m)=\ent(\rho_t\m)
$$
since it can be applied to an arbitrary subnet. 

Also the last statement can be obtained with a small refinement of the proof of Theorem~\ref{thm:stabchee}: it suffices
to start from \eqref{eq:extra_fisher2} and  
then to proceed as in the rest of the proof to obtain $k_j\in D(\Ch_{\beta(j)})$ with $k_j\to f$ in $L^2(X,\m)$ and 
$\liminf_j\Ch_{\beta(j)}(k_j)\leq \Ch(f)$. Now, the inequality $|\rmD k_j|_{w,\sfd_{\beta(j)}}\geq |\rmD k_j|_{w,\sfd}$ $\m$-a.e. in $X$
gives
$$
\int \bigl(|\rmD k_j|_{w,\sfd_{\beta(j)}}-|\rmD f|_{w,\sfd}\bigr)^2\dd\m \leq \Ch_{\beta(j)}(k_j)+\Ch(f)-2\int |\rmD k_j|_{w,\sfd}^2\dd\m\;.
$$
Hence, $|\rmD k_j|_{w,\sfd_{\beta(j)}}\to |\rmD f|_{w,\sfd}$ in $L^2(X,\m)$ along a further subnet. 
Finally, writing $k_j=\tilde{k}_j\circ \pi_{\beta(j)}$, by applying 
Proposition~\ref{prop:calculus}(b) to $\tilde{k}_j$ we can find $\tilde f_j\in {\rm Lip}_b(X_{\beta(j)},\tilde{\sfd}_{\beta(j)})$ with
$$
\lim_{j\in J}\int\bigl| |\rmD \tilde{k}_j|_{w,\tilde\sfd_{\beta(j)}}-{\rm Lip}_a(\tilde f_j,\cdot)|^2\dd\m_{\beta(j)}=0\;.
$$
Setting $f_j=\tilde{f}_j\circ \pi_{\beta(j)}$ we obtain the final part of the statement.
\end{proof}

\section{Energy measure spaces}\label{sec:Energy}

In this section we study a class of extended distances in the framework of Dirichlet forms, basic references on this topic are
 \cite{Bouleau-Hirsch91}, \cite{Fukushima-Oshima-Takeda11}.

\subsection{Dirichlet forms, energy measure spaces and the
  Bakry-\'Emery condition}
\label{subsec:EMS}
\begin{definition}[Energy measure space]\label{def:extmmenergy}
We say that $(X,\mathcal B,\cE,\m)$ is an energy measure space if:
\begin{itemize}
\item[(a)] $\mathcal B$ is a $\sigma$-algebra in $X$ and $\m:\mathcal B\to [0,1]$ is a probability measure;
\item[(b)] $\cE$ is a strongly local and Markovian Dirichlet form in $L^2(X,\m)=L^2(X,\mathcal B,\m)$ 
whose domain $$\V
=\\V:=\{f\in L^2(X,\m):\ \cE(f,f)<\infty\}$$ is dense in $L^2(X,\m)$;  
\item[(c)] $\cE$ admits a  carr\'e du champs operator defined on $\V\cap L^\infty(X,\m)$.
\end{itemize}
\end{definition}

Recall that the Markovian property means $\cE(\phi\circ f,\phi\circ f)\leq \cE(f,f)$ for all $f\in\V$ and all $1$-Lipschitz $\phi:\R\to\R$.
We recall that the carr\'e du champs operator is
the bilinear form $\Gamma:\bigl(\V\cap L^\infty(X,\m)\bigr)^2\rightarrow L^1(X,\m)$ providing a local representation of $\cE$.
When $\Gamma$ exists (in more general situations it has to be understood as a measure-valued operator), it is characterized by the identity
\begin{equation}\label{def:carre}
\int \Gamma(f,f)\varphi\dd\m=-\frac 12\cE(f^2,\varphi)+\cE(f,f\varphi)
\qquad\forall f,\,\varphi\in \V\cap L^\infty(X,\m)\;.
\end{equation}
We use the standard abbreviations $\cE(f)$, $\Gamma(f)$ for
$\cE(f,f)$, $\Gamma(f,f)$, respectively, in the sequel. 
The domain $\V$ of $\cE$ is endowed with the Hilbert norm 
\begin{equation}\label{defV}
\|f\|_\V^2:=\|f\|_2^2+\cE(f)
\end{equation}
and we denote by $\mathcal A_\cE$ the Banach algebra $\{f\in\V:\ f,\,\Gamma(f)\in L^\infty(X,\m)\}$ 
endowed with the norm (see also \eqref{eq:26})
\begin{equation}
  \label{eq:26bis}
  \big\|f\big\|_{\cA_{\cE}}:=\big\|f\big\|_\infty+\big\|\,\Gamma(f)^{1/2}\,\big\|_\infty\;
\end{equation}
We now recall the main properties of the heat flow $\sfP^\cE$ associated to $\cE$. 
It can be defined as the unique locally absolutely continuous (in fact analytic) 
map $t\in (0,\infty)\mapsto f_t\in L^2(X,\m)$ satisfying
$$
\ddt f_t=\Delta_\cE f_t\quad\text{for a.e.~$t\in (0,\infty)$}\;,\qquad
\lim_{t\downarrow 0}f_t=f\quad\text{in $L^2(X,\m)$}\;,
$$
where $\Delta_\cE f$, the infinitesimal generator of the semigroup, is related to $\cE$ by
$$
v=\Delta_\cE f\quad\Longleftrightarrow\quad f\in \V,\ v\in L^2(X,\m),\quad
\int vg\dd\m=-\cE(f,g)\,\,\,\,\,\,\forall g\in\V\;.
$$
Using the characterization
$$
v=\Delta_\cE f\quad\Longleftrightarrow\quad -v\in \partial^-\frac 12\cE(f)
$$
it is easy to check that $\sfP^\cE_t$ is also the metric gradient flow of $\tfrac12 \cE$ with
respect to the $L^2(X,\m)$ distance, according to
\eqref{eq:EDI} and \eqref{def:evik} with $K=0$. 

We recall now a few basic properties of $\sfP^\cE$.

Since (thanks to the Markov property) $\sfP^\cE$ is a contraction also in the $L^1(X,\m)$ norm we can canonically extend it to
a linear semigroup in $L^1(X,\m)$, thanks to the density of $L^2(X,\m)$ in $L^1(X,\m)$. This extension of the semigroup, for which we 
retain the notation $\sfP^\cE$, obviously satisfies 
\begin{equation}\label{eq:ancora_dual}
\int g\sfP_t^\cE f\dd\m=\int f\sfP_t^\cE g\dd\m\qquad\forall
f\in L^1(X,\m),\,\,g\in L^\infty(X,\m),\,\,t\geq 0\;.
\end{equation} 

\begin{proposition}[Properties of $\sfP^\cE$ and derivative of the entropy]\label{prop:P_t}
$\sfP^\cE$ is a Markov self-adjoint linear semigroup in $L^2(X,\m)$, $\Delta_\cE$ has a dense domain and
$t\mapsto \int \sfP_t^\cE f\ln \sfP_t^\cE f\dd\m$ is locally absolutely continuous in $[0,\infty)$ for all $\mu=f\,\m\in D(\ent)$ with 
\begin{equation}\label{eq:derentropy}
-\ddt \int \sfP_t^\cE f\ln \sfP_t^\cE f\dd\m=4\cE(\sqrt{\sfP_t^\cE f})=\int_{\{\sfP_t^\cE f>0\}}\frac{\Gamma(\sfP_t^\cE f)}{\sfP_t^\cE f}\dd\m
\qquad\text{for a.e.~$t>0$}\;.
\end{equation}
In addition, if $L^2(X,\m)$ is separable, $\V$ is a separable Hilbert space.
\end{proposition}
\begin{proof}
The first properties are standard in the theory of semigroups, while \eqref{eq:derentropy} follows by the chain rule if
$f\geq c>0$ $\m$-a.e.~in $X$ and by an easy approximation, since $\m$
is finite, in the general case.

In order to prove separability, recall that, according to a standard results in the theory of semigroups (see for instance \cite[Lem.~4.9]{AGS11b}),
it suffices to find a separable and $\sfP^\cE$-invariant subspace $\V'\subset\V$. The subspace
$$
\V':=\bigcup_{t>0}\sfP^\cE_tL^2(X,\m)
$$
is $\sfP^\cE$-invariant. Its separability follows by the separability of $L^2(X,\m)$ and from the regularizing estimate
$\cE(\sfP^\cE_tf)\leq\|f\|_2^2/t$ for all $t>0$ and $f\in L^2(X,\m)$, 
which corresponds to \eqref{eq:regulaK} with $K=0$, $\bar x=f$
and $z=0$. 
\end{proof}

We recall one of the possible formulation of the functional Bakry-\'Emery condition \cite{Bakry-Emery84} 
for energy-measure spaces \cite{AGS12,Bakry-Gentil-Ledoux14};
other equivalent characterization in this abstract framework may be
found in \cite[Sect.~2.2]{AGS12}, see also \cite{Bakry06}.
\begin{definition}[Bakry-\'Emery condition via gradient contractivity]
  We say that the energy-measure space
  $(X,\cB,\cE,\m)$ satisfies the Bakry-\'Emery condition 
  $\BE K\infty$, $K\in \R$, if 
    \begin{equation}\label{eq:BE-grad}
      \tag{$\BE K\infty$}
      \text{for every $g\in \cA_\cE$}\qquad
      \Gamma(\sfP^\cE_tg) \leq\e^{-2Kt}\,
    \sfP^\cE_t\Gamma(g)\quad\text{$\m$-a.e.~in $X$, for all $t\geq
      0$}\;.
  \end{equation}
\end{definition}
\subsection{Extended distances induced by an energy measure space}
In this context the definition of ${\sf CE}^2(X,\Ch,\m)$ given in the metric setting can be immediately adapted, namely
a curve $\rho_s$ of probability densities belongs to ${\sf CE}^2(X,\cE,\m)$ if for some $c\in L^2(0,T)$ one has
\begin{equation}\label{eq:criterion_ce}
\biggl|\int f\rho_t\dd\m-\int f\rho_s\dd\m\biggr|\leq\int_s^tc(r)\biggl(\int\Gamma(f)\rho_r\dd\m\biggr)^{1/2}\dd r
\qquad\forall f\in\mathcal A_\cE
\end{equation}
for $0\leq s\leq t\leq T$. The least $c$ will still be denoted by $\|\rho_t'\|$.

Also the counterparts $W_\cE$ and $W_{\cE,*}$ of $W_{\Ch}$ and $W_{\Ch,*}$ can be immediately defined:

\begin{definition} \label{def:Wdual_cE} For $\mu_0=\rho_0\m,\,\mu_1=\rho_1\m\in\cP^a(X)$ we define
\begin{equation}\label{eq:matthias_distance_cE}
W_{\cE}^2(\mu_0,\mu_1):=\inf\bigg\{\int_0^1\|\rho_t'\|^2\dd t:\ \rho_t\in {\sf CE}^2(X,\cE,\m)\biggr\}
\end{equation}
and
\begin{equation}\label{eq:defWdual_cE}
W_{\cE,*}^2(\mu_0,\mu_1):=2\sup_\phi\int (\phi_1\rho_1-\phi_0\rho_0)\dd\m\;,
\end{equation}
where the supremum runs among all 
{$({\mathscr L}^1\otimes\cB)$-measurable}
bounded maps $\phi(t,x)=\phi_t(x)$ satisfying $\phi\in C_{\rm w*}([0,1];L^\infty(X,\m))\cap L^1([0,1];\V)$, and 
\begin{equation}\label{eq:HJdistr_cE}
\ddt \phi_t+\frac 12 \Gamma(\phi_t)\leq 0
\qquad\text{in $(0,1)\times X$, in the duality with $\mathcal A_{\cE}$\;.}
\end{equation}
\end{definition}
As for the metric theory, we will use use often
the simplified notation $W_{\cE}(\rho_0,\rho_1)$, $W_{\cE,*}(\rho_0,\rho_1)$ for $W_{\cE}(\rho_0\m,\rho_1\m)$, 
$W_{\cE,*}(\rho_0\m,\rho_1\m)$ respectively.

Arguing as in the proof of Proposition~\ref{prop:propWCh}, 
it is easily seen that $W_\cE$ is length and that $W_\cE^2$ is jointly
convex. In addition, with the same proof given in the metric setting, 
$t\mapsto\rho_t:=\sfP_t\rho\in {\sf CE}^2(X,\cE,\m)$ in any bounded interval $[0,T]$ with
\begin{equation}\label{eq:june2}
\|\rho_t'\|^2\leq \int_{\{\rho_t>0\}}\frac{\Gamma(\rho_t)}{\rho_t}\dd\m\qquad\text{for a.e. $t>0$}\;.
\end{equation}
Concerning $W_{\cE,*}$ one can also extend the same
considerations of Lemma~\ref{lem:equivalent-class}, obtaining in
particular an equivalent definition where the supremum in 
\eqref{eq:defWdual_cE} runs in $C^k([0,1];\cA_\cE)$;
as for Remark~\ref{rem:equivalent:HJ}, it is also easy to check
the joint convexity and the lower semicontinuity of $W^2_{\cE,*}$ with respect to the weak
$L^1$-topology. 

The following result can be obtained with the same proof given in the metric setting, see Proposition~\ref{prop:comparale}.
\begin{proposition}\label{prop:comparale_cE}
$W_\cE\geq W_{\cE,*}$ on $\cP^a(X)\times\cP^a(X)$.
\end{proposition}

Let us quickly discuss two cases when it is possible to prove that 
the distance $W_\cE$ (and a fortiori $W_{\cE,*}$) between two 
probability densities is finite.
\begin{lemma}
  \label{lem:connectivity1}
  Let us suppose that $\cE$ satisfies the global Poincar\'e inequality
  \begin{equation}
    \label{eq:46}
    \int \Big|f-\int f\dd\m\Big|^2\dd\m\leq {\mathrm c_P}\cE(f)\quad \text{for every }f\in \V\;
  \end{equation}
  Then if $\rho_0,\,\rho_1\in L^2(X,\m)$ are probability densities with
  $\rho_i\ge \varrho>0$ $\m$-a.e.~in $X$, $i=0,1$, we have
  \begin{equation}
    \label{eq:48}
    W_{\cE}^2(\rho_0,\rho_1)\le \frac{\mathrm c_P}{\varrho}\int |\rho_1-\rho_0|^2\dd\m\;.
  \end{equation}
\end{lemma}
\begin{proof}
  We just take the linear connecting curve $\rho_s:=(1-s)\rho_0+s\rho_1$,
  and we observe that, with $\bar f=\int f\dd\m$, for every $0\leq s<t\leq 1$ one has
  \begin{align*}
    \frac 1{t-s}\int f(\rho_t-\rho_s)\dd\m&=
    \int (\rho_1-\rho_0) f\dd\m=
    \int (\rho_1-\rho_0) (f-\bar f)\dd\m\\&
    \le\|\rho_1-\rho_0\|_2 \|f-\bar f\|_2
    \le \Big(\frac{\mathrm c_P}{\varrho}\Big)^{1/2} \|\rho_1-\rho_0\|_2 
    \Big(\int \rho_r \Gamma(f)\dd\m\Big)^{1/2} \;.
  \end{align*}
\end{proof}
\begin{lemma}
  \label{lem:connectivity2}
  Let us suppose that $\cE$ satisfies the Logarithmic Sobolev inequality
  \begin{equation}
    \label{eq:46bis}
    2 \ent(\rho\m)\leq{\mathrm c_{LS}} \int_{\{\rho>0\}} \frac{\Gamma(\rho)}\rho\dd\m=4{\mathrm c_{LS}}\cE(\sqrt \rho)\;
  \end{equation}
  for every probability density $\rho$ with $\sqrt \rho\in \V$.
  Then the Talagrand inequality holds
  \begin{equation}
    \label{eq:49}
    \frac 12 W^2_\cE(\mu,\m)\le {\mathrm c_{LS}}\ent(\mu)\quad \text{for every }\mu\in D(\ent)\;.
  \end{equation}
  In particular, $W_{\cE}(\mu_0,\mu_1)<\infty$ whenever $\mu_i=\rho_i\m\in
  D(\mathrm{Ent})$.
  \end{lemma}
\begin{proof} We follow the argument of \cite{Otto-Villani00} to use the
  Logarithmic-Sobolev inequality in order to
  show that the heat flow $\sfP_t^\cE \mu$ connects $\mu$ to
  $\m$ and to estimate its length.

   If $\mu=\rho\m\in D(\ent)$ the curve $\mu_t=\rho_t\m$ where
  $\rho_t=\sfP_t^\cE\rho$ belongs to $\mathsf{CE}^2(X,\cE,\m)$ in any
  bounded interval $[0,T]$, since for every $f\in \cA_\cE$ and for
  every $0\le s\le t$
  \begin{align}\nonumber
    \int f(\rho_t-\rho_s)\dd\m&=
    -\int_s^t \int \Gamma(f,\rho_r)\dd\m\dd r\\ \label{eq:LSI1}
    &\le
    \int_s^t \Big(\int_{\{\rho_r>0\}}
    \frac{\Gamma(\rho_r)}{\rho_r}\dd\m\Big)^{1/2}
    \Big(\int \rho_r\Gamma(f)\dd\m\Big)^{1/2} \dd r\ ;
  \end{align}
  the same formula shows that 
  \begin{equation*}
    \|\rho_t'\|\le \Big(\int_{\{\rho_t>0\}}
    \frac{\Gamma(\rho_t)}{\rho_t}\dd\m\Big)^{1/2}\ .
  \end{equation*}
  On the other hand, for every time $t$ with $\ent(\mu_t)>0$,
  \eqref{eq:derentropy} yields
  \begin{align}\label{eq:LSI2}
    -\ddt \Big(\ent(\mu_t)\Big)^{1/2}=
    \int_{\rho_t>0} \frac{\Gamma(\rho_t)}{\rho_t}\dd\m
    \cdot \Big(4\ent(\mu_t)\Big)^{-1/2}\ge 
    \Big(\frac{1}{{2\mathrm c_{LS}}}
    \int_{\{\rho_t>0\}}
    \frac{\Gamma(\rho_t)}{\rho_t}\dd\m\Big)^{1/2}\;.
  \end{align}
  Since the left hand side is
  integrable in $(0,\infty)$, also the right hand side is integrable
  and, in particular, the essential $\liminf$ of $\int_{\{\rho_t>0\}}\Gamma(\rho_t)/\rho_t\dd\m$ as $t\to\infty$ is null. 
  From \eqref{eq:46bis} and the monotonicity of entropy we conclude that 
  $\ent(\mu_t)\to0$ as $t\to\infty$ and thus $\rho_t\to1$ in $L^1(X,\m)$ as
  $t\to\infty$. Thus we can pass to the limit $t\to\infty$ in
  \eqref{eq:LSI1} and obtain that $t\mapsto\rho_t$ connects $\mu$ to
  $\m$ and is admissible (after reparametrization on a finite
  interval) in the definition of $W_\cE$. Moreover, \eqref{eq:LSI2}
  provides the velocity estimate
  \begin{equation*}
    \Big(\frac{1}{2{\mathrm c_{LS}}}\Big)^{1/2} \|\rho_t'\|\le -\ddt
    \Big(\ent(\mu_t)\Big)^{1/2}\ .
  \end{equation*}
  Eventually, integrating this last inequality from $0$ to $\infty$,
  recalling that $\ent(\mu)\ge0$ and that $\ent(\mu_t)\to0$ as
  $t\to\infty$ the Talagrand inequality \eqref{eq:49} 
  {follows}.
\end{proof}
When $\BE K\infty$ holds with $K>0$,
the well known argument of Bakry and \'Emery 
(we will also provide a proof based on the EVI formulation, which in turn follows by $\BE K\infty$,
see Corollary~\ref{cor:LS}) yields the validity of the Logarithmic Sobolev inequality
\eqref{eq:46bis} with $\mathrm c_{LS}=K$, \emph{provided}
$(X,\cB,\cE,\m)$ satisfies the \emph{irreducibility condition}
\begin{equation}
  \label{eq:53}
  f\in \V,\
  \cE(f)=0\qquad\Longrightarrow\qquad
  \text{$f=c$ $\m$-a.e.~in $X$ for some $c\in\R$}\;
\end{equation}
which is also equivalent to the $L^2$-\emph{ergodicity} of the semigroup
$\sfP^\cE$:
\begin{equation}
  \label{eq:54}
  \lim_{t\to\infty}\sfP^\cE_tf=\int f\dd\m\ \text{strongly in }L^2(X,\m),\quad
  \text{for every }f\in L^2(X,\m)\;.
\end{equation}
\begin{corollary}
  \label{cor:LScheat}
  If $(X,\cB,\cE,\m)$ is irreducible according to \eqref{eq:53} and 
  $\BE K\infty$ holds with $K>0$ then 
  \eqref{eq:46bis} and \eqref{eq:49} are satisfied with $\mathrm
  c_{LS}=K$. In particular every couple of probability measures $\mu_i=\rho_i\m\in
  D(\mathrm{Ent})$ has finite distance $W_{\cE}(\mu_0,\mu_1)<\infty$.
\end{corollary}
Our goal is now to prove that under the contractivity assumption 
$\BE K\infty$ the upper length distance associated to
$W_{\cE,*}$ coincides with $W_\cE$. In the proof the following lemma will play a crucial role.
\begin{lemma}
  \label{lem:speed2} If $\BE K\infty$ holds, 
  then for any curve $(\mu_t)_{t\in[0,1]}$ in
  $AC^2\big([0,1],(\cP(X),W_{\cE,*})\big)$ with $\mu_t=\rho_t \m$ 
  and any $\phi\in \mathcal A_\cE$ one has (denoting by $\abs{\dot\mu_t}$ the metric derivative w.r.t.~$W_{\cE,*}$)
\begin{align}\label{eq:speed-est-1}
  \biggl|\frac{\mathrm d}{\mathrm d s}\int \rho_s \phi\dd \m\biggl|_{s=t}\leq\abs{\dot\mu_t}\cdot\sqrt{\int \Gamma(\phi)\rho_t\dd \m}\qquad\text{for a.e.~$t\in (0,1)$}\;.
\end{align}  
\end{lemma}
\begin{proof} First, defining $W_{\cE,*,1}\leq W_{\cE,*}$ as in \eqref{eq:anche_questa_bis} of Remark~\ref{rem:anche_questa_bis}, we obtain that
$s\mapsto\int\phi\rho_s\dd\m$ is absolutely continuous for all $\phi\in\mathcal A_\cE$. 
Let $\psi\in\mathcal A_\cE\cap D(\Delta_\cE)$ with $\inf\psi>0$. Using the identity $-\psi^{-1}\Delta_\cE\psi+\Delta_\cE\log\psi=-\Gamma(\log\psi)$
and the gradient contractivity condition it is easy to check that, for $K\geq 0$,
$$
\psi_s:=2 \sfP^\cE_s\log \sfP^\cE_{\delta-s}\psi\qquad s\in [0,\delta]
$$ 
is admissible in $W_{\cE,*}$. We fix a point $t$ where $s\mapsto\mu_s$ is metrically differentiable and $s\mapsto \int\rho_s\log\psi\dd\m$
is differentiable. Fix $s>t$ and $\delta=\lambda(s-t)$
with $\lambda>0$. From the inequality 
$$
\delta\biggl|\int (\sfP^\cE_\delta\log\psi)\rho_t\dd\m-(\log\sfP^\cE_\delta\psi)\rho_s\dd\m\biggr|\leq
\frac 14 W_{\Ch,*}^2(\mu_s,\mu_t)\;,
$$
dividing both sides by $(s-t)^2$ and letting $s\to t$ gives
$$
\lambda\biggl|\lambda\int (\Delta_\cE\log\psi-\frac{\Delta_\cE\psi}{\psi}) \rho_t\dd\m-
\ddt \int\rho_t\log\psi\dd\m\biggr|\leq \frac 14 |\dot\mu_t|^2\;.
$$
It follows that
$$
|\ddt \int\rho_t\log\psi\dd\m|\leq \lambda \int \Gamma(\log\psi)\rho_t\dd\m+\frac 1{4\lambda}|\dot\mu_t|^2\;.
$$
By a simple approximation, the same inequality holds for a.e.~$t\in (0,1)$ for all $\psi\in\mathcal A_\cE$ with $\inf\psi>0$, i.e.
removing the assumption $\psi\in D(\Delta_\cE)$.
By minimizing w.r.t.~$\lambda$ and setting $\psi=\e^\phi$ we get the result.

In the case $K<0$ we need to consider the reparameterization $2 \sfP^\cE_{\theta(s)}\log\sfP^\cE_{\theta(\delta)-\theta(s)}\psi$, $s\in [0,\delta]$,
where $\theta(0)=0$ and $\theta'(s)=\e^{-2Ks}$. Since $\theta(s)=s+o(s)$ as $s\downarrow 0$, the same expansions above work with this modified
function. 
\end{proof}

\begin{remark}\label{rem:diffusion}{\rm
Notice that Lemma~\ref{lem:speed2} and the next proposition could be reproduced even in the metric setting, since
the proof of Lemma~\ref{lem:speed2} used only the diffusion formula $\Delta_\cE\phi(f)=\phi'(f)\Delta_\cE f+\phi''(f)\Gamma(f)$, known
to be true also in the metric setting (see \cite[Prop.~4.11]{Gigli}), where $\Delta$ might be nonlinear. 
On the other hand, we preferred to state these results in this section because gradient
contractivity conditions are expected to hold only in presence of quadratic energies.
In the subclass of Minkowski spaces, it is known (see \cite{OhtaSturm12}) that contractivity of the
heat flow w.r.t. to $W_\sfd$ holds if and only if the Minkowski structure is
induced by an inner product.
}\end{remark}

\begin{proposition}[$W_\cE$ is the upper length distance of $W_{\cE,*}$]\label{prop:comparalebis}
If \ref{eq:BE-grad} holds, then $W_\cE$
is the upper length extended distance associated to $W_{\cE,*}$ according to \eqref{eq:81}. 
\end{proposition}
\begin{proof}
Since $W_\cE$ is length and larger than $W_{\cE,*}$, one inequality is obvious. Let us apply Lemma~\ref{lem:speed2} to
obtain
$$
\biggl|\int f\rho_t\dd\m-\int f\rho_s\dd\m\biggr|\leq\int_s^t|\dot\mu_r|\biggl(\int \Gamma(f)\rho_r\dd\m\biggr)^{1/2}\dd r
$$
for all $\mu_t=\rho_t\m\in AC^2\big([0,1],(\cP(X),W_{\cE,*})\big)$ and all $f\in\mathcal A_\cE$. 
It follows that $\mu_t$ is absolutely continuous w.r.t.~$W_\cE$ and that 
$$
W_\cE^2(\mu_1,\mu_0)\leq \int_0^1|\dot\mu_t|^2\dd t\;,
$$ 
which provides, by the arbitrariness of $\mu_t$, the converse inequality.
\end{proof}
We conclude this section by proving that $\BE K\infty$ allows for a further 
regularization in the definition of $W_{\cE,*}$.
Let us first introduce the Banach space 
\begin{equation}
  \label{eq:34}
  D_\infty(\Delta_\cE):=\Big\{f\in L^\infty(X,\m)\cap D(\Delta_\cE):\ 
  \Delta_\cE f\in L^\infty(X,\m)\Big\}\;
\end{equation}
endowed with the graph norm $ \|f\|_{D_\infty}:=\|f\|_\infty+\|\Delta_\cE f\|_\infty$ and
let us recall (see \cite{Ambrosio-Mondino-Savare13}) that if $\BE K\infty$ holds
then $D_\infty(\Delta_\cE)$ is an algebra, continuously
imbedded in $\cA_\cE$; in particular, there exists a constant $C_K>0$ satisfying
\begin{equation}
  \label{eq:45}
  \big\|\Gamma(f)\big\|_\infty\le
  C_K\,\|f\|_\infty\,\|f\|_{D_\infty}
  \le C_K\|f\|_{D_\infty}^2\quad
  \text{for every }f\in D_\infty(\Delta_\cE)\;
\end{equation}
\begin{lemma}
  \label{lem:Dinfty}
  The extended distance $W_{\cE,*}$ can be expressed by the duality
  formula \eqref{eq:defWdual_cE} where the supremum runs among all $\phi\in
  C^\infty([0,1];D_\infty(\Delta_\cE))$.
\end{lemma}
\begin{proof}
  Arguing as in Lemma~\ref{lem:equivalent-class} it is not restrictive
  to take the supremum of \eqref{eq:defWdual_cE} assuming $\phi\in C^\infty([0,1];\cA_\cE)$. We then set
  $\phi^\eps_t:=\sfh^\eps \phi_t$ where $\sfh^\eps$ is the
  mollification of $\sfP^\cE$ introduced by \eqref{eq:sg-moll}. Since
  $\sfh^\eps$ is a bounded linear operator $\cA_\cE$ to $D_\infty(\Delta_\cE)$, the
  curve $t\mapsto \phi^\eps_t$ still belongs to $C^\infty([0,1];D_\infty(\Delta_\cE))$.
  On the other hand, the commutation property
  \eqref{eq:30} shows that $\phi^\eps$ is still a subsolution to
  \eqref{eq:HJdistr_cE}.
  Since $\phi^\eps_t \to \phi_t$ 
  weakly$*$ in $L^\infty(X,\m)$ as $\eps\downarrow0$, it is immediate to
  check that 
  \begin{displaymath}
    \lim_{\eps\downarrow0}2\int
    \big(\rho_1\phi_1^\eps-\rho_0\phi^\eps_0\big)\dd\m=
    2\int \big(\rho_1\phi_1-\rho_0\phi_0\big)\dd\m\;.
  \end{displaymath}
\end{proof}
\subsection{Bakry-\'Emery condition and contractivity of the Heat semigroup}
We say that $\sfP^\cE$ is $K$-contractive w.r.t.~$W_\cE$ if
\begin{equation}\label{eq:contraction_cE}
W_\cE(\sfP^\cE_t f,\sfP^\cE_t g)\leq \e^{-Kt}W_\cE(f,g)\qquad\forall t\geq 0
\end{equation}
for all $f,\,g\in L^1_+(X,\m)$ with $\int f\dd\m=\int g\dd\m=1$.

The proof of the Bakry--\'Emery gradient estimate below will use some
results of the theory developed recently in
\cite[Thm.~4.6]{AT14} for the continuity equation in metric measure spaces. Its proof uses
Hilbert space techniques and a vanishing viscosity argument, choosing as algebra
of functions the set $\V\cap L^\infty(X,\m)$.

\begin{theorem}\label{thm:CE-auton-gradient} Assume that $L^2(X,\m)$ is separable.
  Let $V\in D(\Delta_\cE)\cap L^\infty(X,\m)$ with $\Delta_\cE V\in L^\infty(X,\m)$. Then for any
  $\bar\rho\in L^\infty(X,\m)$ and any $T\in (0,\infty)$ there exists
  $\rho\in C^0_{w*}([0,T];L^\infty(X,\m))$ with $\rho_0=\bar\rho$ and 
  \begin{align}\label{eq:eneidV}
    \int f \rho_{s_1}\dd\m - \int f \rho_{s_2}\dd\m =\int_{s_1}^{s_2}\int\Gamma(f,V)\rho_r\dd\m\dd r
    \qquad\forall \text{$0\leq s_1\leq s_2\leq T$}
    \end{align}
    for all $f\in\V\cap L^\infty(X,\m)$. In particular $\rho\in {\sf CE}^2(X,\cE,\m)$ and the metric derivative of
    $s\mapsto \mu_s$ w.r.t.
    $W_\cE$ can be estimated by
   \begin{equation}\label{eq:stimamet2}
|\dot\mu_s|^2\leq \int \Gamma(V) \rho_s\dd\m\qquad\text{for a.e.~$s\in (0,T)$\;.}
\end{equation}
\end{theorem}

Now we prove that the Bakry--\'Emery
gradient estimate \ref{eq:BE-grad} is equivalent to $K$-contractivity of $\sfP^\cE$ w.r.t.~$W_\cE$. 

\begin{theorem}\label{thm:contraction-BE}
  If $\sfP^\cE$ is $K$-contractive w.r.t.~to $W_\cE$ and $L^2(X,\m)$
  is separable, then the Bakry-Emery condition $\BE K\infty$ holds.
  Conversely, if $\BE K\infty$ holds, then $\sfP^\cE$ is $K$-contractive w.r.t.~$W_\cE$.
\end{theorem}

\begin{proof} Fix $t\geq 0$. We first consider $g\in \V\cap L^\infty(X,\m)$ with $\Delta_\cE g\in
  L^\infty(X,\m)$. Fix a measure $\mu_0=\rho_0\m\in \cP(X)$ with $\rho_0\in
  L^\infty(X,\m)$. Let us denote by $(\rho_s)_{s\in [0,1]}$ the solution
  to the continuity equation driven by the gradient of the constant 
  (w.r.t.~the time parameter $s$) potential $V:=\sfP^\cE_tg$
  starting from $\rho_0$, given by Theorem~\ref{thm:CE-auton-gradient}. Now, let us first note that for any
  $h>0$ one can apply \eqref{eq:eneidV} with $f=\sfP^\cE_tg$ to get
  \begin{align}\label{eq:c1}
    \int \rho_h\sfP^\cE_tg\dd\m - \int \rho_0\sfP^\cE_tg\dd\m = \int_0^h\int\Gamma(\sfP^\cE_tg)\rho_s\dd\m\dd s\;.
  \end{align}
  On the other hand, putting $\mu^t_s=(\sfP^\cE_t\rho_s)\m$ and denoting by $\abs{\dot\mu^t_s}$ the metric derivative of the curve
  $s\mapsto\mu^t_s$, we can estimate:
  \begin{align*}
  \int \rho_h\sfP^\cE_tg\dd\m - \int \rho_0\sfP^\cE_tg\dd\m
   &=\int g \sfP_t^\cE \rho_h\dd\m - \int g\sfP_t^\cE\rho_0\dd\m\leq 
   \int_0^h\abs{\dot\mu^t_s}\left(\int \Gamma(g)\sfP^\cE_t\rho_s\dd\m\right)^{\frac12}\dd s\\
   ~&\leq~
   \e^{-Kt}\left(\int_0^h\abs{\dot\mu_s}^2\dd s\right)^{\frac12}\left(\int_0^h\int \Gamma(g)\sfP^\cE_t\rho_s\dd\m\dd s\right)^{\frac12}\\
   ~&\leq~
   \e^{-Kt}\left(\int_0^h\int\Gamma(\sfP^\cE_tg)\rho_s\dd\m\dd s\right)^{\frac12}\left(\int_0^h\int \sfP^\cE_t\Gamma(g)\rho_s\dd\m\dd s\right)^{\frac12}\;,
  \end{align*}
  where we have used first the inequality
  $\abs{\dot\mu^t_s}\leq\e^{-Kt}\abs{\dot\mu_s}$ (derived by the Wasserstein
  contraction \eqref{eq:contraction_cE}) and then \eqref{eq:stimamet2}. Combining with \eqref{eq:c1} we
  obtain
  \begin{align}\label{eq:c2}
    \int_0^h\int\Gamma(\sfP^\cE_tg)\rho_s\dd\m\dd s \leq\e^{-2Kt}\int_0^h\int \sfP^\cE_t\Gamma(g)\rho_s\dd\m\dd s\;.
  \end{align}
  Dividing by $h$, letting $h\downarrow0$ and using the weak$^*$
  continuity of the curve $\rho_s$ we finally get:
  \begin{align}\label{eq:c3}
    \int\Gamma(\sfP^\cE_tg)\rho_0\dd\m ~\leq~ \e^{-2Kt}\int \sfP^\cE_t\Gamma(g)\rho_0\dd\m\;.
  \end{align}
  By homogeneity, the same inequality holds for any $\rho_0\in L^\infty_+(X,\m)$.
  This clearly implies $\Gamma(\sfP^\cE_tg) \leq\e^{-2Kt} \sfP^\cE_t\Gamma(g)$ $\m$-a.e.~in $X$.

  To prove the assertion for arbitrary $g\in\V\cap L^\infty(X,\m)$ we argue by
  approximation. 
  Consider the following mollification of the semigroup, defined for
  $\eps>0$ and $f \in L^2(X,\m)$ via:
  \begin{align}\label{eq:sg-moll}
    \mathsf  h^\eps
    f=\int_0^\infty\frac{1}{\eps}\eta\left(\frac{t}{\eps}\right) 
  \rme^{(K\land 0)\, t}\,\nc \sfP^\cE_tf\,\dd t\;,
\end{align}
with a non-negative kernel $\eta\in C^\infty_c(0,\infty)$ satisfying
$\int_0^\infty\eta(t)\dd t=1$. It is easily seen that $\mathsf h^\eps$ is
a linear contraction in $\cA_\cE$ satisfying
\begin{equation}
  \label{eq:30}
  \|\sfh^\eps f\|_\infty\le \|f\|_\infty,\quad
  \Gamma(\sfh^\eps f)\le \sfh^\eps \Gamma(f),\quad 
  \text{for every }f\in \cA_\cE\;
\end{equation}
and
\begin{equation}
  \label{eq:31}
  f\in L^\infty(X,\m)\quad \Rightarrow\quad
  \Delta_\cE f\in L^\infty(X,\m),\quad
  \|\Delta_\cE f\|_\infty\le C_\eps \|f\|_\infty
\end{equation}
for some constant $C_\eps>0$. 
From the previous argument we thus obtain
  \begin{align*}
    \int\Gamma(\sfP^\cE_tg_\eps)\rho_0\dd\m\leq\e^{-2Kt}\int \sfP^\cE_t\Gamma(g_\eps)\rho_0\dd\m\;,
  \end{align*}
  for any $\rho_0\in L^\infty_+(X,\m)$. To conclude, it is sufficient to
  check that as $\eps\downarrow 0$ we have $\sfP^\cE_tg_\eps\to \sfP_t^\cE g$ and
  $g_\eps\to g$ in $\V$ as $\eps\to 0$. By convexity of $\cE$ this in turn follows from
  the fact that for any $f\in\V$ we have $\cE(\sfP^\cE_sf-f)\to0$ as
  $s\downarrow 0$.
  
  Eventually we can prove the statement for any $g\in\V$ with a truncation argument.
  
  In order to prove the converse statement, notice first that \ref{eq:BE-grad} implies that $\mathcal A_{\cE}$ is
  invariant under the action of the semigroup. Then, recalling the definition \eqref{eq:criterion_ce} of ${\sf CE}^2(X,\cE,\m)$,
  \eqref{eq:ancora_dual} shows that $\sfP^\cE_t$ maps curves $\rho_s\in {\sf CE}^2(X,\cE,\m)$ to curves 
  $\sigma_s:=\sfP^\cE_t\rho_s\in {\sf CE}^2(X,\cE,\m)$ with $\|\sigma_s'\|\leq \e^{-Kt}\|\rho_s'\|$. By minimization we obtain the
  contractivity property.
\end{proof}
We prove now, 
by standard methods, that \ref{eq:BE-grad} implies
$K$-contractivity of $\sfP^\cE$ also w.r.t.~$W_{\cE,*}$,
a property  that will also follow as a consequence of the EVI estimates 
of the next section.
We don't know if the converse implication, known to
be true for $W_\cE$ under the separability assumption on $L^2(X,\m)$, holds.
\begin{proposition} \label{prop:BE_contra_dual}
If \ref{eq:BE-grad} holds, then
$W_{\cE,*}(\sfP_t^\cE \rho_0,\sfP_t^\cE\rho_1)\leq \e^{-Kt}W_{\cE,*}(\rho_0,\rho_1)$
for all $t\geq 0$ and $\rho_0\m,\,\rho_1\m\in\cP^a(X)$.
\end{proposition}
\begin{proof} Take $\phi$ admissible in the definition of $W_{\cE,*}$ and note that
 $\psi=\e^{2Kt}\sfP^\cE_t\phi$ is again admissible. Indeed, \ref{eq:BE-grad} gives
 $$
   \frac {\dd}{\dd s} \psi +\frac12\Gamma(\psi) \leq \e^{2Kt}\sfP^\cE_t\Big[\dds \phi +\frac12\Gamma(\phi)\Big] \leq 0\;.
 $$
  Thus the definition of $W_{\cE,*}$ gives:
  $$
   \frac12 W_{\cE,*}^2(\rho_0,\rho_1)\geq\int\psi_1\rho_1\dd\m - \int\psi_0\rho_0\dd\m
   = \e^{2Kt}\int \bigl(\phi_1 \sfP^\cE_t\rho_1\dd\m - \phi_0\sfP_t^\cE\rho_0\bigr)\dd\m
   $$  
  and the statement follows by taking the supremum w.r.t.~$\phi$.
\end{proof}

\section{From gradient contractivity to EVI and consequences}\label{sec:actionnew}
In this section $(X,\mathcal B,\cE,\m)$ is an energy measure space
satisfying $\BE K\infty$ for some $K\in\R$.
The main result of this section is: 
\begin{theorem}[$\sfP^\cE$ satisfies $\EVI_K$ relative to $W_{\cE,*}$]\label{thm:EVI-mix}
  For all $\mu=\rho\m\in\cP^a(X)$, $\sigma\in D(\ent)$ with $W_{\cE,*}(\mu,\sigma)<\infty$  
  one has $\ent(\sfP_t^\cE\rho\,\m)<\infty$, $W_{\cE,*}(\sfP_t^\cE\rho\,\m,\sigma)<\infty$ for all $t>0$ and
  (recall that $\ddtr$ stands for upper right derivative)
  \begin{align}\label{eq:EVI-mix}
    \ddtr\frac12 W_{\cE,*}^2(\sfP^\cE_t\rho\,\m,\sigma) + \frac{K}{2}W_{\cE,*}^2(\sfP_t^\cE\rho\,\m,\sigma)\leq \big[\ent(\sigma)-\ent(\sfP_t^\cE\rho\,\m)\big]
  \qquad{ \forall t\geq 0}\;.
  \end{align}
  \end{theorem}
  
Before entering into the technical details, 
let us briefly explain the main idea of the proof.

First of all, thanks to the semigroup property of $\sfP^\cE_t$, 
it is sufficient to prove an ``integrated'' version of
\eqref{eq:EVI-mix}, namely
\begin{align}\label{eq:EVI-mix-ter}
  \frac12 W_{\cE,*}^2(\sfP_t^\cE\rho\,\m,\sigma) 
  +t
  \ent(\sfP_t^\cE\rho\,\m)\le 
  \frac {t}{{\mathrm I}_{2K}(t)}\frac12 W_{\cE,{*}}^2(\rho\m,\sigma)+ t\ent(\sigma)\qquad\forall t>0\;.
\end{align}
{Indeed, the expansion $t/{\mathrm I}_{2K}(t)=1-Kt+o(t)$ and the lower semicontinuity of $\ent$
provide \eqref{eq:EVI-mix} at $t=0$, and the semigroup property provides
the result for all positive times. Notice also that \eqref{eq:EVI-mix-ter} implies all finiteness properties in the statement of the theorem.}

We express the left-hand side by using a dual representation formula,
obtained by combining 
\eqref{eq:defWdual_cE} with the classical conjugate representation of the
Entropy functional
\begin{displaymath}
  \ent(\rho\,\m)=
  \sup_{\zeta\in L^\infty(X,\m)} \Big(\int \rho\zeta\dd\m-\int
  \rme^{\zeta-1}\dd\m\Big)
  =\sup_{\zeta\in L^\infty(X,\m)\cap \V} \Big(\int \rho\zeta\dd\m-\int \rme^{\zeta-1}\dd\m\Big)\;,
\end{displaymath}
where we have restricted the supremum to functions in $\V$ 
by a standard regularization argument (e.g.~by applying
\eqref{eq:sg-moll}).  After the simple transformation $\zeta=1+\psi/t$ yields
\begin{equation}
  \label{eq:35}
  t\ent(\rho\,\m)-t=
  \sup_{\psi\in L^\infty(X,\m)\cap\V} \Big(\int \rho\psi\dd\m-{t }\int \rme^{\psi/t}\dd\m\Big)\;.
\end{equation}
Replacing $\rho$ with $\sfP_t^\cE\rho$, adding the squared distance
term, and using the symmetry of $\sfP^\cE$,
we end up with
\begin{align}
  \label{eq:36}
  \frac12 W_{\cE,*}^2(\sfP_t^\cE\rho\,\m,\sigma) 
  +t
  \ent(\sfP_t^\cE\rho\,\m)-t=
  \sup_{(\phi_t),\psi}\Big(\int \rho\,\sfP_t^\cE(\phi_1+\psi)\dd \m-
  \int \phi_0\dd \sigma-{ t}\int \rme^{\psi/t}\dd\m\Big)\;,
\end{align}
where $\phi$ runs among subsolutions of \eqref{eq:HJdistr_cE} and
$\psi$ runs in $L^\infty(X,\m)\cap\V$. 

Let us now suppose that for every choice of $(\phi_s)_{s\in [0,1]}$ and $\psi\in
L^\infty(X,\m)\cap\V$ we can find a curve $(\psi_s)_{s\in [0,1]}$ in $\cA_\cE$ such that
\begin{align}
  \label{eq:39}
  \psi_1=\psi\quad\text{and}\quad 
  \dds \sfP_{ts}^\cE\big(\phi_s+\psi_s\big)+\frac
  {\rme^{2Kts}}2\Gamma\big(\sfP_{ts}^\cE(\phi_s+\psi_s)\big) \le 0\;.
\end{align}
Recalling \eqref{eq:69} and \eqref{eq:70} with $a=0$, $b=1$, $\vartheta(s)=\rme^{2Kts}$, 
such a curve provides the following upper bound for the term inside the ``$\sup$''
in \eqref{eq:36}
\begin{align}
  \notag\int &\rho\,\sfP_t^\cE(\phi_1+\psi)\dd \m-
  \int \phi_0\dd \sigma-{t }\int \rme^{\psi/t}\dd\m\\&=
  \int \rho\,\sfP_t^\cE(\phi_1+\psi_1)\dd \m-
  \int (\phi_0+\psi_0)\dd \sigma+
  \int \psi_0\dd\sigma-{t }\int \rme^{\psi_0/t}\dd\m+
  {t }\int \big(\rme^{\psi_0/t}-\rme^{\psi/t}\big)\dd\m
  \notag\\&
  \le\frac t{{\mathrm I}_{2K}(t)}
  W^2_{\cE,*}(\rho\,\m,\sigma)+t\ent(\sigma)-t+
  {t }\int \big(\rme^{\psi_0/t}-\rme^{\psi/t}\big)\dd\m\;.
  \label{eq:40}
\end{align}
Expanding \eqref{eq:39} and recalling that  
$\rme^{2Kts}\Gamma\big(\sfP_{ts}^\cE(\phi_s+\psi_s)\big)\le 
\sfP_{ts}^\cE\Gamma\big(\phi_s+\psi_s\big)$ by the 
Bakry-\'Emery condition, we see that 
\eqref{eq:39} is surely satisfied if
\begin{align}
  \label{eq:41}
  \sfP_{ts}^\cE\Big(\dds \psi_s+t\Delta_\cE \psi_s+
  \frac 12
  \Gamma(\psi_s)+t\Delta_\cE\phi_s+\Gamma(\phi_s,\psi_s)\Big)\le 0\;
\end{align}
where we used the fact that 
\begin{displaymath}
  \sfP_{ts}^\cE\Big(\dds \phi_s+\frac 12 \Gamma(\phi_s)\Big)\le 0\;
\end{displaymath}
since $\phi$ is a subsolution of \eqref{eq:HJdistr_cE} and
$\sfP_{ts}^\cE$ is positivity preserving.
This property and the non-negativity of $\Gamma(\psi_s)$ show that 
a candidate for \eqref{eq:41} is provided by the backward Cauchy problem
\begin{equation}
  \label{eq:42}
  \dds \psi_s+t\Delta_\cE \psi_s+
  \Gamma(\psi_s)+t\Delta_\cE\phi_s+\Gamma(\phi_s,\psi_s)=0,\quad 
  s\in [0,1],\qquad
  \psi_{1}=\psi\;
\end{equation}
which can be reduced to the linear backward parabolic problem
\begin{equation}
  \label{eq:44}
  \dds \zeta_s +t\Delta_\cE\zeta_s+
  \zeta_s \Delta_\cE\phi_s+\Gamma(\phi_s,\zeta_s)=0,\quad s\in
  [0,1],\qquad \zeta_1:=\rme^{\psi/t}
\end{equation} 
by applying the well known Hopf-Cole transformation
\begin{equation}
  \label{eq:43}
  \zeta_s:=\rme^{\psi_s/t}\;.
\end{equation}
In conclusion, we have found that 
solving \eqref{eq:44} and setting $\psi_s:=t\log \zeta_s$
we get the bound \eqref{eq:40}. 
Miraculously enough, 
since
\begin{displaymath}
  \int \zeta_s \Delta_\cE\phi_s\dd\m=-
  \int \Gamma(\phi_s,\zeta_s)\dd\m,\quad
  t\int \Delta_\cE\zeta_s\dd\m=0\;
\end{displaymath}
equation \eqref{eq:44} is mass preserving, so that
\begin{displaymath}
  \int \rme^{\psi_0/t}\dd\m=
  \int \zeta_0\dd\m=
  \int \zeta_1\dd\m=
  \int \rme^{\psi/t}\dd\m\;
\end{displaymath}
and with this particular choice the last integral term of
\eqref{eq:40} vanishes; since $\phi$ and $\psi$ are arbitrary,
we obtain \eqref{eq:EVI-mix-ter}. 
\begin{proof}[Let us now check the technical details of the above argument.]
  We divide the
  proof in a few steps: first of all, we will prove the existence of a
  sufficiently smooth solution to \eqref{eq:44}.  We will then show
  that it takes values in a compact interval of $(0,\infty)$, so that
  it will not be difficult to check that the logarithmic
  transformation $\psi_s=t\log \zeta_s$ provides an admissible
  solution to \eqref{eq:41}.

  \textbf{Step 1:} \emph{for every $t>0$, $\zeta_1\in \V$ and
    $\phi\in C^1([0,1];D_\infty(\Delta_\cE))$, there exists a 
    solution $\zeta\in W^{1,2}(0,1;L^2(X,\m))\cap
    L^2(0,1;D(\Delta_\cE))$ (and thus in $C^0([0,1];\V)$)
    of \eqref{eq:44}.}

  Reversing the time order setting $\tilde\zeta_s:= \zeta_{1-s},\
  \tilde\phi_s:=\phi_{1-s}$ and recalling the ``integration by parts''
  formula
  \begin{equation}
    \label{eq:60}
    -\int 
    (\Delta_\cE\tilde\phi_s)\,
    \tilde\zeta_s\eta\dd\m=\cE (\tilde\phi_s,\tilde\zeta_s \eta)=
    \int\tilde\zeta_s \Gamma(\tilde\phi_s,\tilde\eta)\dd\m+
    \int \eta\Gamma(\tilde\phi_s,\tilde\zeta_s)\dd\m\qquad \eta\in \V
  \end{equation}
  which holds since $\Delta\phi_s,\Gamma(\phi_s) \in L^\infty(X,\m)$,
  \eqref{eq:44} is equivalent to the forward Cauchy problem
  \begin{equation}
    \label{eq:57}
    \dds \tilde\zeta_s -t\Delta_\cE\tilde\zeta_s-
    \tilde\zeta_s \Delta_\cE\tilde\phi_s-\Gamma(\tilde\phi_s,\tilde\zeta_s)=0,\quad s\in
    [0,1],\qquad
    \tilde\zeta_0:=\rme^{\psi/t}\in L^2(X,\m)\;
  \end{equation}
  which admits the variational formulation
  \begin{equation}
    \label{eq:58}
    \dds \int \tilde\zeta_s\eta\dd\m+a_s(\tilde\zeta_s,\eta)=0\quad
    \text{in }(0,1)\quad \text{for every }\eta\in \V,
  \end{equation}
  where $(a_s)_{s\in [0,1]}$ is the continuous family of bounded
  bilinear forms in $\V\times \V$
  \begin{equation}
    \label{eq:59}
    a_s(\zeta,\eta):=t\cE(\zeta,\eta)+\int \zeta
    \Gamma(\tilde\phi_s,\eta)\dd\m\qquad \zeta,\,\eta \in \V\;
  \end{equation}
  Since $C:=\sup_s\|\Gamma(\phi_s)^{1/2}\|_\infty<\infty$, we get
  \begin{equation}
    \label{eq:61}
    \Big|\int \zeta
    \Gamma(\tilde\phi_s,\eta)\dd\m\Big|\le C\|\zeta\|_2 \cE(\eta)^{1/2}
  \end{equation}
  and we easily prove that there exist $\lambda,\,\alpha>0$ (depending
  on $t$) such that
  \begin{equation}
    \label{eq:62}
    a_s(\zeta,\zeta)+\lambda \|\zeta\|_2^2\ge \alpha
    \|\zeta\|_{\V}^2\;
  \end{equation}
  A (unique) variational solution $\tilde \zeta \in
  W^{1,2}(0,1;\V')\cap L^2(0,1;\V)$ (and therefore continuous with
  values in $L^2(X,\m)$) then follows by applying J.L.~Lions Theorem, see \cite[Sect.~4.4, Thm.~4.1]{LM72}.
  On the other hand,
  \eqref{eq:57} and the uniform $L^\infty$ bound on $\Delta_\cE\phi_s$
  show that
  \begin{displaymath}
    \dds \tilde\zeta_s-t\Delta_\cE\tilde\zeta_s\in
    L^2(0,1;L^2(X,\m))\; ;
  \end{displaymath}
  since $-t \Delta_\cE$ is the selfadjoint operator in $L^2(X,\m)$
  associated
  to the symmetric Dirichlet form $t\cE$ and since
  $\tilde \zeta_0\in \V$, the standard regularity results for variational evolution 
  equation in Hilbert spaces (see, e.g.~\cite[Chap.~III, Sect.~3]{Brezis73} yield
  $\zeta_s\in W^{1,2}(0,1;L^2(X,\m))$.
Eventually the equation \eqref{eq:57} provides the $L^2(0,1;D(\Delta_\cE))$ regularity.

  \textbf{Step 2.} \emph{Under the same assumptions of the previous
    step, if $|\Delta_\cE \phi_s|\le D$ $\m$-a.e. for every $s\in
    [0,1]$ and $0<\alpha\le \zeta_1\le \beta<\infty$ $\m$-a.e., then }
  \begin{equation}
    \label{eq:63}
    \alpha\rme^{-D(1-s)}\le \zeta_0\le \beta\rme^{D(1-s)}\quad 
    \text{$\m$-a.e.~in $X$ for every $s\in [0,1]$\;}
  \end{equation}
  We just observe that for every function $\theta\in C^1([0,1])$ the
  perturbed solution $\omega_s:=\tilde\zeta_s -\theta_s$ satisfies the
  equation
  \begin{equation}
    \label{eq:64}
    \dds \omega_s -t\Delta_\cE\omega_s-
    \omega_s
    \Delta_\cE\tilde\phi_s-\Gamma(\tilde\phi_s,\omega_s)=f_s,\quad
    f_s=-(\theta_s'+\theta_s \Delta_\cE\tilde\phi_s)
    ,\quad s\in
    [0,1]\;
  \end{equation}
  which can also be written as
  \begin{equation}
    \label{eq:65}
    {\dds} 
    \int  
    \omega_s \eta\dd\m+a_s(\omega_s,\eta)=\int
    f_s\omega_s\dd\m\quad \text{for every }\eta \in \V\quad
    \text{a.e.~in }(0,1).
  \end{equation}
  Choosing $\theta_s=\beta \rme^{Ds}$ we get $f_s\le0$ and
  $\omega_0\le 0$.  Choosing $\eta_s:=(\omega_s)_+$ in \eqref{eq:65}
  {and using the Leibniz rule (whose validity can easily be justified in this setting)} we get
  \begin{equation}
    \label{eq:66}
    \frac 12 \dds \int \eta_s^2 \dd\m-\lambda \int \eta_s^2 \dd\m\le 0
  \end{equation}
  where we used the fact that
  \begin{displaymath}
    a_s(\omega,\omega_+)=a_s(\omega_+,\omega_+)\ge -\lambda \int
    (\omega_+)^2\dd\m\quad
    \text{for every }\omega\in \V\;
  \end{displaymath}
  Since $\eta_0=0$, {\eqref{eq:66}} 
  yields $\eta_s=0$ $\m$-a.e.~in $X$
  for every $s$; we thus obtain $\omega_s\le 0$ and therefore
  $\tilde\zeta_s\le \beta\rme^{Ds}$.  The same argument, choosing
  $\theta_s:=\alpha\rme^{-Ds}$ and $\eta_s=(\omega_s)_-$ yields the
  other inequality $\tilde\zeta_s\ge \alpha\rme^{-Ds}$.

  \textbf{Step 3:} {\em If $\phi\in C^1([0,1];D_\infty(\Delta_\cE))$
    is a subsolution to \eqref{eq:HJdistr_cE} then $\psi_s:=t\log
    \zeta_s$ satisfy}
  \begin{equation}
    \label{eq:68}
    \dds \sfP_{ts}^\cE (\phi_s+\psi_s)+\rme^{2Kts}
    \frac 12\Gamma\big(\sfP_{ts}^\cE(\phi_s+\psi_s)\big)\le 0,\quad 
    s\in [0,1]\;
  \end{equation} 
  The transformation $\psi_s:=t\log \zeta_s$ is admissible thanks to
  the lower and upper bounds proved in the previous step; using the
  fact that
  \begin{displaymath}
    \Gamma(\phi_s,\psi_s)=\frac {{ t}}{\zeta_s}\Gamma(\phi_s,\zeta_s),\qquad
    \Delta_\cE \psi_s+t\Gamma(\psi_s)=\frac{t}{\zeta_s}\Delta_\cE
      \zeta_s\quad\text{in }L^1(X,\m)\;
  \end{displaymath}
  we obtain \eqref{eq:42}; notice that $\psi\in
  W^{1,2}(0,1;L^2(X,\m))\cap C^0([0,1];\V)$ and $\Delta_\cE
  \psi\in L^1(0,1;L^1(X,\m))$.

  Since $\phi$ is a subsolution to \eqref{eq:HJdistr_cE} we get
  \begin{equation}
    \label{eq:67}
    \dds (\phi_s+\psi_s)+t\Delta_\cE (\phi_s+\psi_s)+
    \frac 12\Gamma(\phi_s+\psi_s)\le 0,\quad 
    s\in [0,1]\;
  \end{equation} 
  applying the positivity preserving $\sfP_{ts}^\cE$ and observing
  that
  \begin{displaymath}
    \dds \big(\sfP_{ts}^\cE \varphi_s\big)=\sfP_{ts}^\cE\dds
    \varphi_s+t\Delta_\cE \sfP_{ts}^\cE \varphi_s=\sfP_{ts}^\cE\Big(
    \dds
    \varphi_s+t\Delta_\cE \varphi_s\Big)\;
  \end{displaymath}
  whenever $\varphi\in W^{1,2}(0,1;L^2(X,\m))$ with $\Delta_\cE
  \varphi\in L^1(0,1;L^1(X,\m))$, we get
  \begin{equation}
    \label{eq:68bis}
    \dds \sfP_{ts}^\cE (\phi_s+\psi_s)+\sfP_{ts}^\cE
    \frac 12\Gamma(\phi_s+\psi_s)\le 0,\quad 
    s\in [0,1]\;
  \end{equation}
  which yields \eqref{eq:68} by the $\BE K\infty$ gradient commutation
  property.
\end{proof}
The following corollary is a direct consequence of \eqref{eq:EVI-mix-ter}, see also the metric regularization estimate
\eqref{eq:may_31}.
\begin{corollary}[LlogL regularization] \label{thm:LlogL-Eulerian}
  For any $\mu=\rho\m\in\cP^a(X)$ and $\sigma\in D(\ent)$ we have:
  \begin{align}\label{eq:LlogL-tilde}
    \ent(\sfP^\cE_t\rho\,\m) \leq \ent(\sigma) + 
     \frac{K}{\e^{2Kt}-1} W_{\cE,*}^2(\mu,\nu)\qquad\forall t>0\;.
  \end{align}
\end{corollary}
Also the following corollary is a direct consequence of Theorem~\ref{thm:self-improvement} and of the fact that
$W_{\cE}$ is the upper length distance induced by $W_{\cE,*}$.
\begin{corollary}[$\sfP^\cE$ satisfies $\EVI_K$ relative to $W_{\cE}$]\label{cor:EVI-mix}
  For all $\mu=\rho\m\in\cP^a(X)$, $\sigma\in D(\ent)$ with $W_{\cE}(\mu,\sigma)<\infty$  
  one has $\ent(\sfP_t^\cE\rho\,\m)<\infty$, $W_{\cE}(\sfP_t^\cE\rho\,\m,\sigma)<\infty$ for all $t>0$ and
  \begin{align}\label{eq:EVI-mix-bis}
    \ddtr\frac12 W_{\cE}^2(\sfP^\cE_t\rho\,\m,\sigma) + \frac{K}{2}W_{\cE}^2(\sfP_t^\cE\rho\,\m,\sigma)\leq \ent(\sigma)-\ent(\sfP_t^\cE\rho\,\m)
  \qquad\forall t\geq 0\;.
  \end{align}
  \end{corollary}

We can now obtain the geodesic property of $D(\ent)$ and the convexity of $\ent$, relative to $W_\cE$.
This provides a link with the theory developed independently by Lott-Villani and Sturm 
of synthetic lower bounds on the Ricci tensor, based on convexity properties of $\ent$ (see \cite{Villani09}).
\begin{theorem}[Geodesic convexity of the entropy functional]\label{thm:geoconv}
  $(D(\ent),W_{\cE}))$ is an extended geodesic metric space:
    for every couple of measures $\mu,\,\nu\in D(\ent)$ with
    $W_{\cE}(\mu,\nu)<\infty$ there exists a $W_\cE$-Lipschitz curve $\mu_t:[0,1]\to D(\ent)$ such that 
    \begin{equation}
      \label{eq:82}
      \mu_0=\mu,\quad
      \mu_1=\nu,\quad
      W_{\cE}(\mu_s,\mu_t)=|t-s|W_{\cE}(\mu,\nu),\quad
      s,\,t\in [0,1]\; .
    \end{equation}
    In addition, the ``finitary'' length distance generated by $W_{\cE,*}$ according to \eqref{eq:71} coincides
    with $W_\cE$ on $D(\ent)\times D(\ent)$ and $\ent$ is $K$-convex on every curve as in \eqref{eq:82}:
    \begin{equation}
      \label{eq:83}
      \ent(\mu_t)\le (1-t)\ent(\mu_0)+t\ent(\mu_1)-\frac K2
      t(1-t)W_\cE^2(\mu,\nu)\;.
    \end{equation}
    Finally, for all $\mu=\rho\,\m\in D(\ent)$ the slope of the entropy coincides with the Fisher information
    \begin{equation}\label{eq:Fisher-slope}
    4\cE(\sqrt{\rho})=|\rmD^-_{W_\cE}\ent|^2(\rho\,\m)\;.
    \end{equation}
\end{theorem}
\begin{proof} 
We are going to apply Corollary~\ref{cor:appgeo} with $X=D(\ent)$, $\sfd=W_{\cE,*}$,
$F=\ent$ and $\sf S=\sf P$ (we identify here measures with probability densities, as usual). 
We know from Theorem~\ref{thm:EVI-mix} that ${\sf S}$ provides a $\EVI_K$-gradient flow of $F$ in $X$, hence
$(\cP^a(X),W_{\cE,*,\ell})$ (where $W_{\cE,*,\ell}$ is the length distance associated to $W_{\cE,*}$ as in \eqref{eq:71}) 
is a length space and the same holds for all sublevels $\{\ent\leq c\}$, $c\in [0,\infty)$. 
On the other hand, since the sublevels are compact w.r.t.~the weak $L^1(X,\m)$ topology (thanks to
Dunford-Pettis theorem), we immediately obtain that $(\{\ent\leq c\},W_{\cE,*})$ are complete, thanks
to the lower semicontinuity of $W_{\cE,*}$ w.r.t. the weak $L^1(X,\m)$ convergence. It follows by
Corollary~\ref{cor:appgeo} that on $D(\ent)\times D(\ent)$ the distance $W_{\cE,*,\ell}$ coincides with the upper
length distance induced by $W_{\cE,*}$, namely $W_\cE$ (by Proposition~\ref{prop:comparalebis}).

Thanks to compactness, the length properties of the sublevels can be improved to geodesic
properties by the remarks made after Corollary~\ref{cor:appgeo}. We can now use Corollary~\ref{cor:EVI-mix} 
to improve the $\EVI_K$ property from $W_{\cE,*}$ to $W_{\cE}$, getting then the convexity of $\ent$ along geodesics
of $W_\cE$.

We need only to prove the inequality $\geq$ in \eqref{eq:Fisher-slope}, since the converse inequality can be
proved independently of
curvature assumptions, see the proof of Lemma~\ref{lem:fisherboundsslope} in the metric setup and
recall \eqref{eq:june2}. We start from the observation that for any $\rho\in L^2_+(X,\m)$ one has
$$
4\cE(\sqrt{{\sf P}_t^\cE\rho})=|\rmD^-_{W_\cE}\ent|^2({\sf P}_t^\cE\rho\,\m)\qquad\text{for a.e. $t>0$}
$$
by looking at the energy dissipation rates from the $L^2$ point of view and by the $W_\cE$ point of view (the latter
is derived from the $\EVI_K$ property). Now, if $\rho\in\V$ with $\inf\rho>0$ we can pass to the limit as $t\downarrow 0$
along a suitable sequence and use the lower semicontinuity of $|\rmD_{W_\cE}^-\ent|$ w.r.t. $W_\cE$ convergence (derived from the convexity of entropy)
to obtain the inequality $\geq$ in \eqref{eq:Fisher-slope}. For general probability densities we argue by truncation,
using once more the lower semicontinuity of $|\rmD_{W_\cE}^-\ent|$.
\end{proof}
We conclude pointing out some standard consequences of the $K$-convexity of $\ent$.

\begin{corollary} [Convexity of bounded densities, Log-Sobolev and
  transport inequalities]\label{cor:LS}
\ \\
(i) If $K\geq 0$, then the sets $\{\mu=\rho\,\m\in\cP^a(X):\
\|\rho\|_\infty\leq c\}$ are geodesically convex w.r.t.~$W_\cE$, 
i.e.~every couple of measures $\mu_i=\rho_i\m$, $i=0,\,1$,
with $\|\rho_i\|_\infty\leq c$ and $W_\cE(\mu_0,\mu_1)<\infty$,
can be connected by a geodesic $\mu_t=\rho_t\m$ as in \eqref{eq:82}
such that $\|\rho_t\|_\infty\le c$ for every $t\in [0,1]$. 

\noindent
(ii) If $K>0$ and $\cE$ is irreducible according to \eqref{eq:53},
then the log-Sobolev inequality \eqref{eq:46bis} holds with ${\textrm c_{LS}=K}$.
In particular, thanks to Lemma~\ref{lem:connectivity2}, one has the Talagrand inequality
\begin{equation}\label{eq:talag}
\frac K2 W_\cE^2(\mu,\m)\leq\ent(\mu)\qquad\forall\mu\in\cP^a(X)\;.
\end{equation}
\end{corollary}
\begin{proof} For the geodesic convexity of the sets
  $\{\mu=\rho\,\m\in\cP^a(X):\ \|\rho\|_\infty\leq c\}$ the rescaling
  of transport plans as in \cite[Prop.~3.3]{AGS11b} applies (notice
  that the assumptions on the side of the supports made therein play a
  role only when $K<0$).

The metric argument of \cite[Lem.~2.4.13]{Ambrosio-Gigli-Savare08}, relying on $K$-convexity, can be applied to give 
$$
\ent(\rho_1\,\m)-\ent(\rho_0\,\m)\leq \frac{1}{2K}|\rmD^-\ent|^2(\rho_1\,\m)\;,
$$
for any pair of measures $\mu_0=\rho_0\,\m$, $\mu_1=\rho_1\,\m$ in $D(\ent)$ with $W_{\cE}(\mu_0,\mu_1)<\infty$. It
$\rho\in D(\ent)$ we apply this inequality with $\rho_0=\sfP_t\rho$ and $\rho_1=\rho$ and we let
$t\to\infty$ to get, by the irreducibility of $\cE$ and \eqref{eq:54},
$$
\ent(\rho\,\m)\leq \frac{1}{2K}|\rmD^-\ent|^2(\rho\,\m)\;.$$ 
Using \eqref{eq:Fisher-slope} we obtain the log-Sobolev inequality.
\end{proof}

\begin{corollary}\label{lem:june17}
The transport inequality \eqref{eq:talag} implies that the class of measures in $\cP^a(X)$ with bounded density w.r.t. 
$\m$ is dense in $D(\ent)$ w.r.t. $W_\cE$.
\end{corollary}
\begin{proof}
Let us approximate $\mu$ by the measures $\mu^k:=1_{\{\rho\leq k\}}\mu/a_k$, where $a_k=\int_{\{\rho\leq k\}}\rho\dd\m\uparrow 1$
are the normalization constants. If $\rho^k$ are the densities of $\mu^k$, writing $b_k=1-a_k=\int_{\{\rho>k\}}\rho\dd\m$ and
$$
\rho^k=1_{\{\rho\leq k\}}\rho+\bigl(\frac {1}{a_k}-1)1_{\{\rho\leq k\}}\rho=
1_{\{\rho\leq k\}}\rho+b_k\rho^k\;,
\qquad
\rho=1_{\{\rho\leq k\}}\rho+1_{\{\rho>k\}}\rho
$$
adding the constant term $1_{\{\rho\leq k\}}\rho$ to the solutions to the continuity inequality we get
$$
W_\cE^2(\mu^k,\mu)\leq b_k W_\cE^2(\mu^k,\frac{1}{b_k}1_{\{\rho>k\}}\rho\m)\;.
$$
It suffices then
to show that $\lim_kb_kW_\cE(\mu^k,b_k^{-1}1_{\{\rho>k\}}\rho\m)=0$. 
To this aim, we compare both measures with $\m$. The transport inequality then gives
$$
b_k W_\cE^2(\mu^k,\m)\leq \frac {2b_k} K \ent(\mu^k)\rightarrow 0
$$
and
\begin{equation}\label{eq:eforselultima}
b_k W_\cE^2(\frac {1}{b_k}1_{\{\rho>k\}}\rho,\m)\leq\frac 2 K \biggl[\int_{\{\rho>k\}}\rho\ln\rho\dd\m+
b_k\ln(\frac{1}{b_k})\biggr]\rightarrow 0\;.
\end{equation}
\end{proof}

\section{From differentiable to metric structures and conversely}\label{sec:transfer}

\subsection{Energy measure spaces induce extended metric measure spaces}
 
In this section $(X,\mathcal B,\cE,\m)$ is an energy measure space according to Definition~\ref{def:extmmenergy};
following the construction explained in Section~\ref{sec:3}, page \pageref{ssubsec:family}, we 
are going to introduce an extended metric-topological structure 
starting from given a family $\cL$ 
of \emph{pointwise defined} real functions such that
\begin{subequations}
\label{subeq:emts}
\begin{equation}\label{eq:ass1}
  \begin{gathered}
    \cL\subset \{f:X\to \R: f\text{ is $\mathcal B$-measurable
      and bounded, $\Gamma(f)\le 1$}\}\;,\\
    \quad \text{$\cL$ separates points of $X$}\;,
  \end{gathered}
\end{equation}
so that (equivalence $\m$-a.e.~classes of) elements of $\cL$ belong
to $\cA_\cE$. Then
\begin{equation}
  \label{eq:29}
  \tau \ \text{is the Hausdorff topology 
    in $X$ generated by $\cL$}\,,
\end{equation}
i.e.~$\tau$ is the coarsest topology
such that all the functions of $\cL$ are continuous: $(X,\tau)$ is
automatically completely regular. Restricting $\m$ to $\BorelSets{\tau}\subset\mathcal B$, we will
assume that 
\begin{equation}\label{eq:ass2}
\m\in\cP(X) \quad\text{(i.e. $\m$ is Radon in $\BorelSets{\tau}$)},\quad \supp \m=X\,;
\end{equation}
we can then consider the class 
\begin{equation}
  \label{eq:32}
  \cAs=\big\{f\in \cA_\cE\cap C_b(X,\tau): \Gamma(f)\le 1\big\}
\end{equation} 
(where we identify functions in $\cAs$ with their unique $\tau$-continuous representative)
containing $\cL$ and use $\mathcal A_\cE^*$  
to define canonically $\sfd_\cE:X\times X\to [0,\infty]$ by
\begin{equation}\label{eq:defCE}
\sfd_\cE(x,y):=\sup\left\{|f(x)-f(y)|:\ f\in \cAs\right\}\;,
\end{equation}
\end{subequations}
so that $(X,\sfd_\cE)$ is an extended metric space. 

In addition, if $I$ is the collection of finite subsets of $\mathcal A_\cE^*$, for $i\in I$ we define
\begin{equation}
\sfd_i(x,y):=\sup_{f\in i}|f(x)-f(y)|\;.\label{eq:47}
\end{equation}
Notice that $\sfd_i$ is only a semidistance, i.e. it is symmetric and it satisfies the triangle inequality. We shall
also use the fact that $\sfd_i(\cdot,y)\in\mathcal \cAs$ for all $y\in X$, with $\Gamma(\sfd_i(\cdot,y))\leq 1$
$\m$-a.e.~in $X$. If we endow $X$ with the semidistance $\sfd_i$ it is
immediately seen that $(X,\sfd_i)$ is separable. 

We thus get that $(X,\tau,\sfd_\cE)$ is an extended metric-topological space 
according to Definition~\ref{def:luft1}
and that $(X,\tau,\sfd_\cE,\m)$ is an extended metric measure space according to Definition~\ref{def:extmm}.

\begin{remark}
  \label{rem:typical}
  The typical case of this construction occurs when $\supp\m=X$ and
$\cL$ can be identified with a subset of $\mathcal A_\cE\cap C_b(X,\tau_0)$
for some preexisting topology $\tau_0$ in $X$;
in this case the condition $\supp\m=X$ provides uniqueness of the
continuous representative
and $\tau$ is coarser than $\tau_0$, so that \eqref{eq:ass2} is 
satisfied if $\m\in \cP(X,\tau_0)$. Notice that 
$\tau=\tau_0$ if $\tau_0$ is generated by $\cL$.
\end{remark}
\begin{proposition}\label{prop:Wchain_cE}
Under assumptions \eqref{eq:ass1}, \eqref{eq:ass2} one has $W_{\cE,*}\geq W_{\sfd_\cE}$, where
$W_{\cE,*}$ is defined in \eqref{eq:defWdual_cE}. 
\end{proposition}
\begin{proof}
Let $\sfd_i$ as above. By Theorem~\ref{thm:compact_joint} we need only to show that $W_{\sfd_i}\leq W_{\cE,*}$. 
In order to prove this property, taking \eqref{eq:HopfLax} and the comments immediately after into account, it suffices to show that 
$$Q^i_t\phi(x):=\inf_{y\in X}\phi(y)+\frac 1{2t}\sfd_i^2(x,y)$$
is admissible in \eqref{eq:defWdual} whenever $\phi$ is bounded and $\sfd_i$-Lipschitz. By applying Lemma~\ref{lem:3.12AOP} below
to $Q^i_t\phi(\cdot)$ we get $\Gamma(Q^i_t\phi(\cdot))\leq |\rmD_i Q^i_t\phi(\cdot)|^2$ $\m$-a.e.~in $X$, where $|\rmD_i f|$ denotes the slope w.r.t.~$\sfd_i$.
Taking into account the subsolution property \eqref{eq:subsolutionQt} of $Q^i_t\phi$, we obtain $\partial_tQ^i_t\phi+\Gamma(Q^i_t\phi)/2\leq 0$.
\end{proof}

Let $\sfd$ be a finite semidistance in $X$.
In the proof of the next lemma we are going to use in $(X,\sfd)$ the following links between the 
descending slope in \eqref{eq:descending_slope} computed w.r.t.~$\sfd$ and
the functions $Q_tf$, $f\in {\rm Lip}(X,\sfd)$, provided by the Hopf-Lax formula 
\begin{equation}\label{eq:HopfLax_sfdi}
Q_tf(x)=\inf_{y\in Y} f(y)+\sfd^2(x,y)/2t\;,
\end{equation}
see \cite[Sec.~3]{AGS12} for the proof (see also \cite[Lem.~3.1.5]{Ambrosio-Gigli-Savare08}):
\begin{equation}\label{eq:25}
|\rmD f|^2(x)\geq |\rmD^- f|^2(x)=\limsup_{t\downarrow 0}\int_0^1\Big(\frac{\rmD^+f(x,tr)}{tr}\Big)^2\dd r\;.
\end{equation}
Here, $\rmD^+f(x,t)$ is defined by
$$
\rmD^+f(x,t):=\sup\left\{\limsup_{n\to\infty}\sfd(x_n,x):\ \text{$(x_n)$ minimizing sequence in \eqref{eq:HopfLax_sfdi}}\right\}\;.
$$
It is not hard to show, by diagonal arguments, that $\rmD^+ f$ is upper semicontinuous in $X\times (0,\infty)$, endowed
with the product of $\sfd$ and of the Euclidean distance (see again \cite[Sec.~3]{AGS12}). These results are stated in
\cite{AGS12} for metric spaces, and they can be immediately adapted to degenerate space $(X,\sfd)$, just noticing 
that $\sfd(x,y)=0$ implies $f(x)=f(y)$, i.e. lifting them from the quotient metric space to $(X,\sfd)$.

\begin{lemma}\label{lem:3.12AOP}
Let $\sfd$ be a bounded $(\tau\times\tau)$-continuous semidistance in $X$ with 
$(X,\sfd)$ separable, $\sfd(\cdot,y)\in\V$ and $\Gamma(\sfd(\cdot,y))\leq 1$ $\m$-a.e.~in $X$
for all $y\in X$. Then, for all $f:X\to\R$ bounded and $\sfd$-Lipschitz, denoting by
$|\rmD f|$ the slope w.r.t.~$\sfd$, one has $\Gamma(f)\leq |\rmD f|^2$ $\m$-a.e.~in $X$.
\end{lemma}
\begin{proof} In the proof, which follows closely \cite[Lem.~3.12]{AGS12}, 
we will use the following weak stability property of the $\Gamma$ operator, which follows easily
by Mazur's lemma: if $f_n\in \V$ and $f_n\to f$ in $L^2(X,\m)$, then
\begin{equation}\label{eq:20}
\sqrt{\Gamma(f_n)}\to G\quad\text{weakly in $L^2(X,\m)$}
\qquad\text{implies}\qquad  G\geq\sqrt{\Gamma(f)}\quad\text{$\m$-a.e.~in $X$\;.}
\end{equation}
If $(z_i)$ is a countable $\sfd$-dense subset of $X$ we define
\begin{equation}\label{eq:37}
  Q_t^n f(x)=\min_{1\le i\le n} f(z_i)+\frac 1{2t}\sfd^2(z_i,x)\;,\qquad
  Q_t f(x)=\min_{y\in X} f(y)+\frac 1{2t}\sfd^2(y,x)\;,
\end{equation}
and we set $I_n(x):=\big\{i\in \{1,\ldots,n\}:\ z_i\text{ minimizes }\eqref{eq:37}\big\}$. By the density of $(z_i)$, it is clear
that $Q_t^n f\downarrow Q_t f$ as $n\to\infty$.
Therefore, if $\zeta_n(x)\in I_n(x)$, it turns out that $(\zeta_n(x))$ is a minimizing sequence for $Q_tf(x)$, namely
\begin{displaymath}
    \frac 1{2t}\sfd^2(x,\zeta_n(x))+f(\zeta_n(x))\to
    Q_t f(x)\quad\text{as }n\to\infty\;.
  \end{displaymath}
The very definition of $\rmD^+f(x,t)$ then gives
  \begin{equation}\label{eq:aop1}
    \limsup_{n\to\infty}\sfd(x,\zeta_n(x))\le \rmD^+f(x,t)\;.
  \end{equation}
  Since $Q^n_tf(x)=f(z_i)+\sfd^2(z_i,x)/2t$ on $\{x:\ \zeta_n(x)=z_i\}$, the locality property and the fact that  
  $\sfd(z_i,\cdot)$ belongs to $\mathcal A_\ce$ together with the chain rule 
  yield
\begin{displaymath}
  \Gamma(Q_t^n f)(x)\le\frac 1{t^2} \max_{i\in I_n(x)}\sfd^2(x,z_i)
   \quad\text{for $\m$-a.e.~$x\in \{\zeta_n=z_i\}$\;.}
\end{displaymath}
If we define $\zeta_n(x)$ as the smallest index $j$, among those that realize the maximum 
for $\sfd(z_i,x)$, $i\in I_n(x)$, the previous formula yields
  \begin{equation}\label{eq:24b}
   \Gamma(Q_t^n f)(x)\leq \frac 1{t^2} \sfd^2(x,\zeta_n(x))
   \qquad \text{for $\m$-a.e.~$x\in X$\;.}
 \end{equation}
  Since $Q_t^n f$ and $\Gamma(Q_t^n f)$ are uniformly bounded and $Q_t^n f$ converges pointwise to
  $Q_t f$, considering any weak limit point $G$ of
  $\sqrt{\Gamma(Q_t^n f)}$ in $L^2(X,\m)$ we obtain
  by \eqref{eq:20}, \eqref{eq:aop1} and \eqref{eq:24b} that
  \begin{equation}\label{eq:aop2}
    \Gamma(Q_t f)(x)\le G^2(x)\le\frac{\bigl(D^+f(x,t)\bigr)^2}{t^2}
    \quad\text{for $\m$-a.e.~$x\in X$\;.}
  \end{equation}
  
  Since $f$ is Lipschitz, it follows that $\rmD^+f(x,t)/t$ is uniformly
  bounded and since $\sfd$ is $(\tau\times\tau)$-continuous the function $\rmD^+f$ is Borel in $X\times (0,\infty)$.
  Integrating \eqref{eq:25} on an arbitrary Borel set $A$ and
  applying Fatou's Lemma, from \eqref{eq:aop2} we get
  \begin{align*}
    \int_A |\rmD f|^2\dd\m&\ge \int_A  \limsup_{t\downarrow 0}
    \int_0^1 \Big(\frac{D^+f(x,tr)}{tr}\Big)^2\dd r\dd\m(x)\\&\ge 
    \limsup_{t\downarrow0} \int_0^1 \int_A
    \Big(\frac{ D^+f(x,tr)}{tr}\Big)^2\dd\m(x)\,\dd r\\&\ge 
    \limsup_{t\downarrow0}\int_0^1 \int_A
    \Gamma(Q_{tr} f)(x)\dd \m(x)\,\dd r\\&\ge 
    \int_0^1 \liminf_{t\downarrow 0}
    \Big(\int_A \Gamma(Q_{tr} f)\dd\m\Big)\dd r \ge\int_A \Gamma(f)\dd\m\;,
  \end{align*}
  where in the last inequality we applied
  \eqref{eq:20} once more. Since $A$ is arbitrary we conclude.
\end{proof}

The $\tau$-upper regularity 
has already been identified in \cite{AGS12} as a crucial compatibility condition between
the topological and the metric/differentiable structures, needed to identify $\cE$ with a Cheeger energy.

\begin{definition}[$\tau$-upper regularity]\label{def:upper_regularity_cE}
Let $(X,\mathcal B,\cE,\m)$ be an energy measure space with $\BorelSets{\tau}\subset\mathcal B$ for some topology $\tau$ in $X$.
We say that $\cE$ is $\tau$-upper regular if for all $f\in \V$ there exist:
\begin{itemize}
\item[(a)] functions $f_n\in {\rm Lip}_b(X,\sfd,\tau )$ with $f_n\to f$ in $L^2(X,\m)$;
\item[(b)] bounded 
  $\tau$-upper semicontinuous functions $g_n$ with $g_n\geq\sqrt{\Gamma(f_n)}$ $\m$-a.e.~in $X$ with 
$\limsup_n\int g_n^2\dd\m\leq\cE(f)$.
\end{itemize}
\end{definition}

\begin{theorem}\label{thm:ultrap} Let $(X,\mathcal B,\cE,\m)$ be an
energy measure space and let $(X,\tau,\sfd_\cE,\m)$ be the extended
metric-topological structure associated to a set $\cL$ as in {\upshape (\ref{subeq:emts}a,b,c,d,e)}. 
Then $\cE\leq \Ch_{\sfd_{\cE}}$ and, in particular, $W_\cE\geq W_{\Ch_{\sfd_\cE}}$. The equality 
$\cE= \Ch_{\sfd_{\cE}}$ holds iff $\cE$ is $\tau$-upper regular. In particular, if $\cE$ is $\tau$-upper regular, one has:
\begin{itemize}
\item[(a)] the classes of $2$-absolutely continuous curves $\mu_t=\rho_t\m$ w.r.t.~$W_\cE$, $W_{\cE,*}$ and $W_{\sfd_{\cE}}$ with
$\rho_t\in L^\infty(L^\infty(X,\m))$ coincide and the same is true for the corresponding metric derivatives;
\item[(b)] {If \ref{eq:BE-grad} holds}, the metric gradient flows of $\ent$ w.r.t.~$W_\cE$, $W_{\cE,*}$ and $W_{\sfd_{\cE}}$ coincide with $\sfP^\cE_t$;
\item[(c)] If \ref{eq:BE-grad} holds  and
$\sfP^\cE_t$ maps $C_b(X)$ in $C_b(X)$, then $W_\cE$ is the upper length distance
in $\cP^a(X)$ associated to $W_{\sfd_\cE}$
according to \eqref{eq:81}.
\item[(d)] If \ref{eq:BE-grad} holds with $K\geq 0$ (resp. $K>0$ and $\sfP$ is irreducible),
then $W_\cE$ restricted to $\{\mu=\rho\,\m:\ \|\rho\|_\infty\leq c\}$ (resp. $D(\ent)$)
is the upper length distance in $\cP^a(X)$ associated to $W_{\sfd_\cE}$
according to \eqref{eq:81}.
\end{itemize}
\end{theorem}
\begin{proof} We prove first the inequality $\cE\leq \Ch_{\sfd_{\cE}}$.
Taking Theorem~\ref{thm:stabchee} and the lower semicontinuity of $\cE$ into account, it suffices to show that
$\cE\leq \Ch_i$, where $\Ch_i$ is the Cheeger energy associated to the semimetric measure space $(X,\tau,\sfd_i,\m)$ and 
$\sfd_i$ is the monotone approximation of $\sfd_{\cE}$ illustrated by \eqref{eq:47}. By Lemma~\ref{lem:3.12AOP}
we obtain
$$
\cE(f)\leq\int |\rmD_i f|^2\dd\m
$$
for all $f:X\to\R$ bounded and $\sfd_i$-Lipschitz, where $|\rmD_i f|$ denotes the slope w.r.t.~$\sfd_i$. Using
Proposition~\ref{prop:calculus}(b) and once more the lower semicontinuity of $\cE$ we conclude.

The necessity of $\tau$-upper regularity for the validity of the equality $\cE=\Ch_{\sfd_{\cE}}$ follows by
applying Theorem~\ref{thm:stability} to the $(\tau\times\tau)$-continuous semidistances $\sfd_i$ which monotonically converge
to $\sfd_\cE$: one obtains the $\tau$-upper regularity (along a subnet $i=\beta(j)$) with
$f_i\in {\rm Lip}(X,\tau,\sfd_i)$ and with the $\sfd_i$-upper semicontinuous 
  (and thus also $\tau$-upper semicontinuous) functions
$g_i={\rm Lip}_a(f_i,\sfd_i,\cdot)$. 

For the sufficiency of $\tau$-upper regularity we follow the argument
in \cite[Prop.~3.11]{AGS12}. 
Thanks to the $\tau$-upper regularity, in order
to prove that $\Ch_{\sfd_{\cE}}\leq\cE$ it is sufficient to show that $f\in D(\Ch_{\sfd_\cE})$ and that $|\rmD f|_{w,\sfd_\cE}\leq g$ $\m$-a.e.~in $X$
whenever $f\in {\rm Lip}(X,\tau,\sfd_\cE)$ 
and $g$ is a bounded $\tau$-upper semicontinuous function such that $g\geq \sqrt{\Gamma(f)}$
$\m$-a.e.~in $X$. By the very definition of $\Ch_{\sfd_\cE}$,
$f\in D(\Ch_{\sfd_{\cE}})$. In order to prove the inequality $|\rmD f|_{w,\sfd_\cE}\leq g$ $\m$-a.e.~in $X$ we will 
prove the inequality for the slope $|\rmD f|_{\sfd_\cE}$. We need only, thanks to the upper semicontinuity of $g$, 
to prove the pointwise inequality $|\rmD f|_{\sfd_\cE}\leq c$ in
the $\tau$-open set $U:=\{g<c\}$. By homogeneity, we can assume
$c=1$ and we fix $x_0\in U$; since $\tau$ is generated by $\cL$ we can find a finite
collection $(f_n)_{n=1}^N$ of elements of $\cL$ and $r>0$ such that 
\begin{displaymath}
  F:=\big\{x\in X:\ \max_{1\leq n\leq N} |f_n(x)-f_n(x_0)|\le r\big\}\subset U\;.
\end{displaymath}
Set
$$
\delta(x):=\max_{1\leq n\leq N} |f_n(x)-f_n(x_0)|,\quad
\nc {l(x)}:=\min\{r,|f(x)-f(x_0)|\}\;,\quad
h(x):=\max\{{l}(x),\delta(x)\nc \}\;,
$$
and notice that $\delta\in \cAs$. 
Since $\{h={l}\}=\{{l}\geq\delta\}\subset F  
\subset U$ and since $\Gamma(l)\leq 1$ in $U$, 
by locality we get $\Gamma(h)\leq 1$ $\m$-a.e.~in $X$;
since $h$ is $\tau$-continuous we get $h\in \cAs$, so that \eqref{eq:defCE} yields 
$$
h(x)=h(x)-h(x_0)\leq\sfd_\cE(x,x_0)\;.
$$
Since the topology induced by $\sfd_\cE$ is stronger than $\tau$,
the $\tau$-continuity of $f$ gives $h(x)\geq |f(x)-f(x_0)|$ for
$\sfd_\cE(x,x_0)$ sufficiently small. It follows that $|\rmD f|_{\sfd_\cE}(x_0)\leq 1$.

Finally, statements (a) and (b) follow by Corollary~\ref{cor:eqmet} and Theorem~\ref{thm:mainidenti} of the metric theory, taking also
the inequalities $W_{\sfd_\cE}\leq W_{\cE,*}\leq W_{\cE}$ into account.

Let us prove now statement (c). It suffices to show that $W_{\cE}^2(\rho_0\m,\rho_1\m)\leq
\int_0^1|\dot\mu_s|^2\dd s$ for any absolutely continuous curve $\mu_t$ w.r.t.~$W_{\sfd_\cE}$ contained in $\cP^a(X)$, 
where $|\dot\mu_t|$ denotes the metric derivative w.r.t.~$W_{\sfd_\cE}$. Since $\cE$ is $\tau$-upper regular, we can identify 
$\cE$ with $\Ch_{\sfd_\cE}$ and $\Gamma(f)$ with $|\rmD f|_w^2$. By the definition of $W_\cE$, it will be sufficient to prove the inequality
\begin{equation}\label{eq:gita2}
\biggl|\int\phi\rho_{s_1}\dd\m-\int\phi\rho_{s_2}\dd\m\biggr|\leq
 \int_{s_1}^{s_2}\biggl(\int \Gamma(\phi)\rho_s\dd\m\biggr)^{1/2}|\dot\mu_s|\dd s
\end{equation}
for all $\phi\in\mathcal A_\cE$ and $0\leq s_1\leq s_2\leq 1$.

We start from the observation that for any $\psi\in\mathcal A_\cE\cap C_b(X)$ and any bounded $\tau$-upper semicontinuous function
$g\geq \sqrt{\Gamma(\psi)}$ $\m$-a.e.~in $X$ the Lipschitz property of $\psi$ w.r.t.~$\sfd_\cE$ can be ``localized'' as in the
first part of the proof of the theorem to get that $\psi\circ\eta$ is absolutely continuous in $[0,1]$ for all $\eta\in AC([0,1];(X,\sfd_\cE))$ with
$|(\psi\circ\eta)'|\leq g(\eta)|\dot\eta|$ a.e.~in $(0,1)$. By integrating this inequality along a test plan $\eeta$ representing the curve $\mu_t$
we get
$$
\biggl|\int \psi\rho_{s_1}\dd\m-\int\psi\rho_{s_2}\dd\m\biggr|\leq
\int_{s_1}^{s_2}\biggl(\int g^2\rho_s\dd\m\biggr)^{1/2}|\dot\mu_s|\dd s\;.
$$
Now, as in \cite[Thm.~3.15]{AGS12} we can apply the regularization property
$$
\sfP_t^\cE:L^\infty(X,\m)\to\mathcal A_\cE\;,\qquad \|\Gamma(\sfP_t^\cE f)\|_\infty\leq c(K,t,\|f\|_\infty)\quad\forall f\in L^\infty(X,\m)
$$
derived from \ref{eq:BE-grad} with $\Gamma$-calculus techniques (see for instance \cite[Cor.~2.3]{AGS12} for a proof) and the Feller property
$\sfP^\cE_t:C_b(X)\to C_b(X)$ (which implies, by monotone approximation, 
that the class of bounded $\tau$-upper semicontinuous functions is invariant under $\sfP_t^\cE$)
to get from the previous inequality with $\psi=\sfP_t^\cE\phi$ the inequality
$$
\biggl|\int \sfP_t^\cE\phi\rho_{s_1}\dd\m-\int \sfP_t^\cE\phi\rho_{s_2}\dd\m\biggr|\leq
\int_{s_1}^{s_2}\biggl(\int (c\wedge \e^{-2Kt}\sfP^\cE_tg^2)\rho_s\dd\m\biggr)^{1/2}|\dot\mu_s|\dd s
$$
with $c=c(K,t,\|\phi\|_\infty)$. Now, for all $\phi\in\mathcal A_\cE$ we exploit once more the $\tau$-upper regularity assumption, using in
the previous inequality functions 
$\phi_n\in\mathcal A_\cE$ with $\|\phi_n\|_\infty\leq\|\phi\|_\infty$ and bounded $\tau$-upper semicontinuous functions $g_n\geq\sqrt{\Gamma(\phi_n)}$ 
with $\phi_n\to\phi$ and $g_n\to \sqrt{\Gamma(\phi)}$ in $L^2(X,\m)$, to get 
\begin{eqnarray*}
\biggl|\int \sfP_t^\cE\phi\rho_{s_1}\dd\m-\int \sfP_t^\cE\phi\rho_{s_2}\dd\m\biggr|&\leq&
\int_{s_1}^{s_2}\biggl(\int (c\wedge \e^{-2Kt}\sfP^\cE_t\Gamma(\phi))\rho_s\dd\m\biggr)^{1/2}|\dot\mu_s|\dd s\\&\leq&
\e^{-Kt}\int_{s_1}^{s_2}\biggl(\int \sfP^\cE_t\Gamma(\phi)\rho_s\dd\m\biggr)^{1/2}|\dot\mu_s|\dd s\;.
\end{eqnarray*}
Eventually we can take the limit as $t\downarrow 0$ to obtain \eqref{eq:gita2}.

The proof of (d) in the case $K\geq 0$
is similar and uses the convexity properties of $\{\mu=\rho\,\m:\ \|\rho\|_\infty\leq c\}$ to avoid the regularization based
on the Feller property $\sfP_t^\cE:C_b(X)\to C_b(X)$. In the case $K>0$ the result can be extended to $D(\ent)$ thanks
to the $W_\cE$-density of $\{\mu=\rho\,\m:\ \rho\in L^\infty(X,\m)\}$ in $D(\ent)$, ensured by Corollary~\ref{lem:june17}.
\end{proof}

\subsection{Extended metric measure spaces induce energy measure spaces}

In view of the results of this section, it is useful to consider the case when $\Ch$ is a quadratic form, namely
to assume that the parallelogram identity holds:
\begin{equation}\label{eq:parallelogram}
\text{$\Ch(f+g)+\Ch(f-g)=2\Ch(f)+2\Ch(g)$ for all $f,\,g\in L^2(X,\m)$\;.}
\end{equation} 

\begin{definition}[Asymptotically Hilbertian spaces]
We say that an extended metric measure space $(X,\tau,\sfd,\m)$ is asymptotically Hilbertian if $\Ch$ satisfies
the parallelogram identity \eqref{eq:parallelogram}.
\end{definition}

In the proof of the next theorem we will also need the following calculus property, borrowed from \cite{AGS11b}.

 \begin{lemma}[Plan representing the gradient, horizontal and vertical derivatives]
   \label{lem:horizontal-vertical}
    Assume that $(X,\sfd)$ is complete and let $u\in D(\Ch)\cap L^\infty(X,\m)$. Then there exists a test
    plan $\ssigma\in\cP(X^{[0,1]})$ representing the gradient of $u$ in the following
    sense:
  $$
  \lim\limits_{t\downarrow0} \frac{E_t}{t}= \lim\limits_{t\downarrow0} \frac{u\circ\e_0-u\circ\e_t}{E_t}=|\rmD u|_w\circ\e_0 \quad\text{in } 
  L^2(X^{[0,1]},\ssigma)\;,
  $$
  where $E_t(\eta):=\sqrt{t\int_0^t|\dot\eta(s)|^2\dd s}$.
   Moreover, for any $f\in D(\Ch)$ one has:
   \begin{align}\label{eq:horizontal-vertical}
     \liminf\limits_{t\downarrow0}\int\frac{f(\eta(t))-f(\eta(0))}{t}\dd\ssigma\geq
     \limsup\limits_{\eps\downarrow0}\int\frac{|\rmD u|^2_w(\eta(0))-|\rmD (u+\eps f)|^2_w(\eta(0))}{2\eps}\dd\ssigma\;.
    \end{align}
  \end{lemma}
  \begin{proof}
    The first result is proven as in \cite[Lem. 4.15]{AGS11b} and relies on
    Proposition~\ref{prop:lisini} applied to the semigroup $\sfP_t$ starting from $v:=c\e^u$ (with $c$ suitable
    normalization constant), and then defining $\ssigma:=v^{-1}\circ e_0\eeta$.
    The possibility to apply Proposition~\ref{prop:lisini} to the gradient flow is ensured by \eqref{eq:contiCh} and Theorem~\ref{thm:abs_char}.
    The second result is obtained as in \cite[Lem.~4.5]{AGS11b}.
  \end{proof}

\begin{theorem}\label{thm:ultrap1}
If $(X,\tau,\sfd,\m)$ is an asymptotically Hilbertian extended
metric-topological measure space with $(X,\sfd)$ complete, 
and if $\Ch$ denotes the associated Cheeger energy, defining $\cE:=\Ch$ one has that
$(X,\BorelSets{\tau},\cE,\m)$ is an energy measure space according to Definition~\ref{def:extmmenergy}.
In addition:
\begin{itemize}
\item[(a)] $\Gamma(f)=\abs{\rmD f}^2_w$ for any $f\in \V\cap L^\infty(X,\m)$ and the formula
$$
\Gamma(f,g):=\lim_{\epsilon\downarrow 0}\frac{|\rmD (f+\epsilon g)|_w^2-|\rmD f|_w^2}{2\epsilon}
 \qquad\text{in $L^1(X,\m)$}
 $$
extends the $\Gamma$ operator from $\V\cap L^\infty(X,\m)$ to the whole of $\V$.
\item[(b)] the extended distance  
\begin{equation}\label{eq:defCEbis}
\sfd_\cE(x,y):=\sup\left\{|f(x)-f(y)|:\ f\in \V\cap C_b(X),\,\,\Gamma(f)\leq 1\right\}
\end{equation}
satisfies
$\sfd_\cE\geq \sfd$ and, denoting by $\Ch_{\sfd_\cE}$ the Cheeger
energy associated to the new metric-topological structure
$(X,\tau,\sfd_{\cE},\m)$, one has $\Ch_{\sfd_\cE}=\cE$.
\item[(c)] $\sfd_{\cE}=\sfd$ if and only if $f\in \V\cap C_b(X)$ and $|\rmD f|_w\leq 1$ $\m$-a.e.~in $X$ implies 
$f\in {\rm Lip}_b(X,\tau,\sfd)$ with ${\rm Lip}(f,\sfd)\leq 1$. 
\end{itemize}
\end{theorem}
\begin{proof}
By the locality and chain rule properties of $f\mapsto |\rmD f|_w$ stated in Proposition~\ref{prop:calculus}(c,e), 
the asymptotically Hilbertian assumption ensures that $\cE:=\Ch$ is a strongly local and Markovian Dirichlet form in $L^2(X,\m)$.

The proof of statement (a) can be obtained as in \cite[Sec.~4]{AGS11b}, see in particular
\cite[Thm.~4.18]{AGS11b}. For convenience we briefly sketch the proof. Let us set for $f,\,g\in \V$
  \begin{align*}
    G(f,g):=\lim_{\epsilon\downarrow 0}\frac{|\rmD (f+\epsilon
      g)|_w^2-|\rmD f|_w^2}{2\epsilon} \qquad\text{in } L^1(X,\m)\;.
  \end{align*}
  Note that the limit above exists as a monotone limit by convexity of
  the minimal weak upper gradient, and that $\cE(f,g)=\int
  G(f,g)\dd\m$. One first establishes the following chain rule for
  $G$: for all $f,\,g\in \V$ and $\phi:\R\to\R$ non-decreasing, Lipschitz and $C^1$,
  one has
  \begin{align}\label{eq:G-chain}
    G(f,\phi(g))=\phi'(g)G(f,g)\;,\quad \int
    G(\phi(g),f)\dd\m=\int\phi'(g)G(g,f)\dd\m\;.
  \end{align}
  This is proven arguing as in \cite[Lem.~4.7]{AGS11b}. Moreover one follows
  \cite[Prop.~4.17]{AGS11b} using Lemma~\ref{lem:horizontal-vertical}
  to establish the Leibniz rule
  \begin{align}\label{eq:G-Leibniz}
   \cE(f,gh)=\int h G(f,g) + g G(f,h)\dd\m
  \end{align}
  for all $f,\,g,\,h\in \V\cap L^\infty(X,\m)$ with $g,\,h\geq0$.
  
  To prove the claim it is sufficient to show that $G$ is bilinear and
  symmetric and therefore in turn it is sufficient to prove that
  $f\mapsto \int h |\rmD f|_w^2\dd\m$ is quadratic for every nonnegative
  bounded Borel $h$ or, by approximation, $h\in \V\cap
  L^\infty(X,\m)$. Now use \eqref{eq:G-chain}, \eqref{eq:G-Leibniz} to
  write
  \begin{align*}
    \int h |\rmD f|_w^2\dd\m =\int h G(f,f)\dd\m = -\cE(\frac12 f^2,h) + \cE(f,fh)\;.
  \end{align*}
  We conclude by quadraticity of the terms on the right hand side.
  
  The proof of statement (b) is easy, since all functions $f\in {\rm
    Lip}_b(X,\tau,\sfd)$ belong to $C_b(X)$ and
  $\sqrt{\Gamma(f)}=|\rmD f|_w\leq {\rm Lip}(f)$ $\m$-a.e.~in $X$. It
  follows that all the distances $\sfd_i$ approximating $\sfd$ from
  below are admissible in \eqref{eq:defCEbis}, so that
  $\sfd_\cE\geq\sfd$. Since Theorem~\ref{thm:stability} ensures that
  $\cE=\Ch$ is $\tau$-upper regular, by Theorem~\ref{thm:ultrap} we
  obtain $\Ch_{\sfd_\cE}=\cE$.
   
  In order to prove statement (c), notice that the inclusion $\{f\in
  C_b(X)\cap \V:\ |\rmD f|_w\leq 1\}$ in the class of
  $1$-Lipschitz functions w.r.t.~$\sfd$ implies, by the very
  definition of $\sfd_\cE$, that $\sfd_\cE\leq\sfd$.  The converse is
  obvious, again by the definition of $\sfd_\cE$.
\end{proof}

\section{Examples}
\label{sec:examples}

Here we collect natural examples of energy measure spaces and
extended metric measure spaces.

\subsection{Degenerate Dirichlet forms}

Consider $X=\R^2$ equipped with the usual topology $\tau$,
the Borel $\sigma$-algebra $\cB(\tau)$,  and the standard Gaussian measure $\gamma(\mathrm{d} x)=(2\pi)^{-1}
\e^{-|x|^2/2}\dd x$. Consider further a Dirichlet form measuring energy only in the first coordinate, i.e.
\begin{align*}
  \cE(f)=\int |\partial_1 f(x_1,x_2)|^2\dd\gamma(x_1,x_2) 
\end{align*}
for $f\in L^2(\R^2,\m)$ with $f(\cdot,x_2)\in W^{1,2}_{\text{loc}}(\R)$
for a.e. $x_2\in\R$. Then $(\R^2,\cB(\tau),\cE,\gamma)$ is an energy
measure space according to Definition~\ref{def:extmmenergy}.

As the distance generated from the algebra $\cA^*_{\cE}$ according to \eqref{eq:32}, \eqref{eq:defCE} one obtains
\begin{align*}
  \sfd_\cE\big((x_1,x_2),(y_1,y_2)\big)=
  \begin{cases}
    |x_1-y_1| & \text{if } x_2=y_2\;,\\
    +\infty & \text{else}
  \end{cases}
\end{align*}
and $(X,\tau,\sfd_\cE,\gamma)$ is an extended metric measure space according
to Definition~\ref{def:extmm}.

Note that the Bakry--\'Emery condition $\BE 1\infty$ is satisfied for
this Dirichlet form, but $\cE$ is not irreducible (see
\eqref{eq:53}); in fact, it is easy to construct measures
with bounded densities but infinite $W_{\cE}$ distance from $\m$.
Using a standard approximation by restriction, one can check that $\cE$ is $\tau$-upper regular. It follows that
$\cE$ coincides with the Cheeger energy induced by
$\sfd_\cE$, by Theorem~\ref{thm:ultrap}.

\subsection{Abstract Wiener spaces}

Consider a separable Banach space $X$ (or, more generally, a locally convex topological space)
and a centered, non-degenerate Gaussian measure $\gamma$ in $X$.
The Cameron-Martin space $\cH\subset X$ is the image under the mapping 
$$
Rf:=\int f(x)x\dd\gamma(x)\qquad f\in L^2(X,\gamma)
$$
of the so-called reproducing kernel of $\gamma$, namely the closure $H$ in $L^2(X,\gamma)$ of $\{\langle x',x\rangle\}_{x'\in X'}$
(here $X'$ is the topological dual of $X$ and $\langle\cdot,\cdot\rangle$ is the canonical pairing). It is easily seen that $R:H\to\cH$ is
injective, hence $\cH$ inherits from $L^2(X,\gamma)$ a Hilbert structure. The structure $(X,\gamma,\cH)$ is called
abstract Wiener space.

Denote by $\cH^\infty_b$ the set of smooth, bounded cylinder
functions, i.e. the functions $f$ of the form
$f(x)=\phi\big(\langle x_1',x\rangle,\ldots,\langle x_n',x\rangle\big)$, for
$x_1',\dots,x_n'\in X'$ and $\phi:\R^n\to\R$ smooth and bounded.
For such a cylinder function we define its gradient via
\begin{align*}
  \nabla_\cH f(x) = \sum\limits_{i=1}^n\frac{\partial\phi}{\partial z_i}\big(\langle x_1',x\rangle,\ldots,\langle x_n',x\rangle\big)
  R(\langle x_i',\cdot\rangle)\;.
\end{align*}
It is well-known that the quadratic form
\begin{align*}
  \cE(f)=\int |\nabla_\cH f|_\cH^2\dd\gamma\qquad f\in\cH^\infty_b\;
\end{align*}
is closable, that $\cH^\infty_b$ is dense in $L^2(X,\gamma)$ and that $\cE$ admits a carr\'e du champ operator,
see \cite{Bouleau-Hirsch91}. 
Thus $(X,\cB(\tau),\cE,\gamma)$, with
$\tau$ being the weak or strong topology of the Banach space $X$ is an energy
measure space according to Definition~\ref{def:extmmenergy}.

The associated semigroup $\sfP$ is given by Mehler's formula
\begin{align*}
  \sfP_tf(x)=\int f(\e^{-t}x+\sqrt{1-\e^{-2t}}y) \dd\gamma(y)\;. 
\end{align*}
From this one can check that the Bakry--\'Emery condition
$\BE 1\infty$ holds. Using a Rademacher type theorem for abstract Wiener spaces, see \cite{ES93}, one
can check that the induced distance according to \eqref{eq:32}, \eqref{eq:defCE} is the Cameron--Martin distance
\begin{align*}
  \sfd_\cE(x,y)=
  \begin{cases}
    |x-y|_\cH & \text{if } x-y\in \cH\;,\\
    +\infty & \text{else}\;. 
  \end{cases}
\end{align*}
Thus $(X,\tau,\sfd_\cE,\gamma)$ is an extended metric measure space according to Definition~\ref{def:extmm}.

\subsection{Configuration spaces}
Configuration spaces appear naturally as the state space for systems
of infinitely many indistinguishable diffusing particles. Let $M$ be a
Riemannian manifold with metric tensor
$\langle\cdot,\cdot\rangle$. The configuration space $\Upsilon$ over
$M$ is the set of all locally finite counting measures, i.e.
\begin{align*}
  \Upsilon=\{\gamma\in \mathcal{M}(M) : \gamma(K)\in\N_0\ \forall K\subset M \text{ compact}\}\;.
\end{align*}
The space $\Upsilon$ is equipped with the vague topology, denoted by $\tau$, in
duality with continuous and compactly supported functions.

A natural differentiable and energy structure on the configuration
space has been introduced in \cite{AKR98}, by lifting the geometry on
$M$, as we shall briefly describe. The tangent space
$T_\gamma\Upsilon$ consists of all $\gamma$-square integrable
vector fields on $M$ and is equipped with the inner product
\begin{align*}
  \langle V_1,V_2\rangle_\gamma=\int\langle V_1(x),V_2(x)\rangle_x\dd\gamma(x)\;. 
\end{align*}
Let $Cyl_b^\infty$ be the set of smooth and bounded cylinder
functions, i.e. functions $F:\Upsilon\to\R$ of the form
$F(\gamma)=g\big(\gamma(\phi_1),\ldots,\gamma(\phi_n)\big)$ where
$g\in C^\infty_b(\R^n)$ and $\phi_1,\dots,\phi_n\in C^\infty_c(M)$ and
we write $\gamma(\phi)=\int\phi\dd\gamma$.  Given $F\in Cyl_b^\infty$
we define its gradient at $\gamma$ as the vector field on $M$ given by
\begin{align*}
  T_\gamma\Upsilon\ni\nabla^\Upsilon F(\gamma) 
  =\sum\limits_{i=1}^n\frac{\partial g}{\partial z_i}\big(\gamma(\phi_1),\ldots,\gamma(\phi_n)\big)\nabla\phi_i\;.
\end{align*}
Similarly, for a cylindrical ``vector field'' on $\Upsilon$ of the
form $W=\sum_i F_iV_i$ with $F_i\in Cyl^\infty_b$ and $V_i$ smooth,
compactly supported vector fields on $M$, we define its divergence as
\begin{align*}
  \rm{div}^\Upsilon(W)(\gamma)=\sum_i\langle\nabla^\Upsilon F_i,V_i\rangle_\gamma + F_i(\gamma) \gamma(\rm{div}V_i)\;.
\end{align*}
The natural reference measure on $\Upsilon $ is the Poisson measure
$\pi$, that can be characterized by its Laplace transform, i.e. for all $f\in C_b(M)$:
\begin{align*}
  \int_{\Upsilon}\exp\big(\gamma(f)\big)\dd\pi(\gamma) =\exp\left(\int_M\exp\big(f(x)\big)-1 \dd\text{vol}(x)\right)\;.
\end{align*}
This is (up to the intensity) the unique {probability}
measure such that the gradient
and the divergence are adjoint in $L^2(\Upsilon,\pi)$, see
\cite[Thm. 3.2]{AKR98}. The quadratic form
\begin{align*}
  \cE(F,F)=\int \langle\nabla^\Upsilon F,\nabla^\Upsilon \rangle\dd\pi
\end{align*}
defined on $Cyl^\infty_b$ is closable to a Dirichlet form admitting a
carr\'e du champ operator, see \cite[Cor. 1.4]{AKR98},
\cite[Prop. 1.4]{RS99}, so that $(\Upsilon, \cB(\tau),\cE,\pi)$ is a
Energy measure space according to Definition~\ref{def:extmmenergy}. The associated semigroup is given by the
evolution of infinitely many independent Brownian particles on $M$.

The induced distance according to \eqref{eq:32}, \eqref{eq:defCE} is known to be the
$L^2$-transport distance between (non-normalized) configurations (see
\cite[Thm. 1.5]{RS99}), i.e.
\begin{align*}
  \sfd_\cE(\gamma,\eta)=\inf\limits_q\sqrt{\int\sfd^2(x,y)\dd q(x,y)}\;,
\end{align*}
where {$\sfd$ is the Riemannian distance and} the infimum is over all couplings $q$ of $\gamma$ and $\eta$.
$(X,\tau,\sfd_\eps,\pi)$ is now an extended metric measure space
according to Definition~\ref{def:extmm}.

It is shown in \cite[Prop. 2.3]{EH14} that $\cE$ coincides with the Cheeger energy constructed
from $d_\cE$ (as defined in \cite{AGS11a} based on Lipschitz
constants, but similar arguments apply to the construction used here
based on asymptotic Lipschitz constants).

A detailed study of curvature bounds for configuration spaces can be
found in \cite{EH14}. It has been shown that various notions of
curvature bounds lift from the base space $M$ to the configuration
space. In particular, if the Ricci curvature of $M$ is bounded below
by $K$, an Evolution Variational Inequality and the Bakry--\'Emery
gradient estimate with constant $K$ for the semigroup hold on the
configuration space.

\end{document}